\documentclass[a4paper,10pt]{amsart}

\usepackage[T1]{fontenc}
\usepackage{lmodern}
\usepackage{amsmath}
\usepackage{amsfonts}
\usepackage{amssymb}
\usepackage{amsthm}
\usepackage{thmtools}
\usepackage{mathtools}
\usepackage{enumitem}
\usepackage{placeins}
\usepackage{multirow}
\usepackage{tikz}
\usetikzlibrary{chains,arrows,shapes,decorations.pathmorphing,decorations.markings,decorations.pathreplacing,positioning,cd,fit,shapes.geometric}
\usepackage{tikz-cd}
\usepackage[pdftex,
            pdfauthor={},
            pdftitle={Classification of Self-Dual Vertex Operator Superalgebras of Central Charge at Most 24},
            pdfsubject={},
            pdfkeywords={},
            pdfproducer={},
            pdfcreator={}]{hyperref}
\usepackage{caption}
\usepackage{tablefootnote}
\usepackage{longtable}
\usepackage{xfrac}

\theoremstyle{plain}
\newtheorem{thm}{Theorem}[section]
\newtheorem{prop}[thm]{Proposition}

\newtheorem{lem}[thm]{Lemma}

\newtheorem*{thm*}{Theorem}
\newtheorem*{conj*}{Conjecture}
\newtheorem*{claim*}{Claim}
\newtheorem*{prop*}{Proposition}

\newtheoremstyle{narrow}
  {.5em} 
  {.5em} 
  {\itshape} 
  {} 
  {\bfseries} 
  {.} 
  {.5em} 
  {} 

\theoremstyle{narrow}

\newtheorem*{thn*}{Theorem}

\newtheorem*{conjn*}{Conjecture}

\theoremstyle{definition}
\newtheorem{defi}[thm]{Definition}

\newtheorem{conv}[thm]{Convention}
\newtheorem*{nota*}{Notation}
\newtheorem{rem}[thm]{Remark}

\newcommand{\Q}{\mathbb{Q}}
\newcommand{\Z}{\mathbb{Z}}
\newcommand{\Ns}{\mathbb{Z}_{>0}}
\newcommand{\N}{\mathbb{Z}_{\geq0}}
\newcommand{\C}{\mathbb{C}}
\newcommand{\R}{\mathbb{R}}
\renewcommand{\H}{\mathbb{H}}

\newcommand{\tr}{\operatorname{tr}}
\renewcommand{\i}{\mathrm{i}}
\newcommand{\e}{\mathrm{e}}
\newcommand{\End}{\operatorname{End}}
\newcommand{\Aut}{\operatorname{Aut}}
\newcommand{\Hom}{\operatorname{Hom}}

\newcommand{\rk}{\operatorname{rk}}

\newcommand{\sign}{\operatorname{sign}}
\newcommand{\voa}{vertex operator algebra}

\newcommand{\VOA}{Vertex Operator Algebra}
\newcommand{\vosa}{vertex operator subalgebra}

\newcommand{\VOSA}{Vertex Operator Subalgebra}
\newcommand{\svoa}{vertex operator superalgebra}

\newcommand{\SVOA}{Vertex Operator Superalgebra}

\newcommand{\fpvosa}{fixed-point vertex operator subalgebra}

\newcommand{\vac}{\textbf{1}}
\newcommand{\ch}{\operatorname{ch}}
\newcommand{\Ch}{\operatorname{Ch}}

\newcommand{\id}{\operatorname{id}}

\newcommand{\SLZ}{\operatorname{SL}_2(\mathbb{Z})}

\newcommand{\ee}{\mathfrak{e}}
\newcommand{\g}{\mathfrak{g}}
\newcommand{\hh}{\mathcal{H}}

\newcommand{\Out}{\operatorname{Out}}

\newcommand{\so}{\mathfrak{so}}

\newcommand{\Irr}{\operatorname{Irr}}

\newcommand{\orb}{\operatorname{orb}}

\newcommand{\strat}{strongly rational}
\newcommand{\II}{I\!I}
\renewcommand{\O}{\operatorname{O}}
\newcommand{\Co}{\operatorname{Co}}

\newcommand{\Com}{\operatorname{Com}}
\newcommand{\no}{\,{\raise0.25em\hbox{$\mathop{\hphantom{\cdot}}\limits^{_{\circ}}_{^{\circ}}$}}\,}

\newcommand{\Rep}{\operatorname{Rep}}
\newcommand{\glob}{\operatorname{glob}}
\newcommand{\diag}{\operatorname{diag}}
\newcommand{\shs}{\mkern 1mu}
\newcommand{\ph}{\phantom{\shs\sfrac{1}{2}}}
\newcommand{\phh}{\phantom{\sfrac{1}{2}}}
\newcommand{\VB}{V\!B^\natural} 
\newcommand{\VO}{V\!O^\natural} 
\newcommand{\Vect}{\operatorname{Vect}_\mathbb{C}}
\newcommand{\sVect}{\operatorname{sVect}_\mathbb{C}}

\newcommand{\sAA}{1^{24}}
\newcommand{\sBB}{1^82^8}
\newcommand{\sCC}{1^63^6}
\newcommand{\sDD}{2^{12}}
\newcommand{\sEE}{1^42^24^4}
\newcommand{\sFF}{1^45^4}
\newcommand{\sGG}{1^22^23^26^2}
\newcommand{\sHH}{1^37^3}
\newcommand{\sII}{1^22^14^18^2}
\newcommand{\sJJ}{2^36^3}
\newcommand{\sKK}{2^210^2}

\newcommand{\gAA}{\II_{24,0}(1)}
\newcommand{\gBB}{\II_{16,0}(2_{\II}^{+10})}
\newcommand{\gCC}{\II_{12,0}(3^{-8})}
\newcommand{\gDD}{\II_{12,0}(2_{\II}^{-10}4_{\II}^{-2})}
\newcommand{\gEE}{\II_{10,0}(2_{2}^{+2}4_{\II}^{+6})}
\newcommand{\gFF}{\II_{8,0}(5^{+6})}
\newcommand{\gGG}{\II_{8,0}(2_{\II}^{+6}3^{-6})}
\newcommand{\gHH}{\II_{6,0}(7^{-5})}
\newcommand{\gII}{\II_{6,0}(2_{5}^{-1}4_{1}^{+1}8_{\II}^{-4})}
\newcommand{\gJJ}{\II_{6,0}(2_{\II}^{+4}4_{\II}^{-2}3^{+5})}
\newcommand{\gKK}{\II_{4,0}(2_{\II}^{-2}4_{\II}^{-2}5^{+4})}
\newcommand{\gLL}{\II_{0,0}(1)}

\renewcommand{\arraystretch}{1.15}

\begin{document}

\title[Classification of Self-Dual Vertex Operator Superalgebras]{Classification of Self-Dual\\Vertex Operator Superalgebras\\of Central Charge at Most 24}
\author[Gerald Höhn and Sven Möller]{Gerald Höhn\textsuperscript{\lowercase{a}} and Sven Möller\textsuperscript{\lowercase{b},\lowercase{c}}}
\thanks{\textsuperscript{a}{Kansas State University, Manhattan, KS, United States of America}}
\thanks{\textsuperscript{b}{Research Institute for Mathematical Sciences, Kyoto University, Kyoto, Japan}}
\thanks{\textsuperscript{c}{Universität Hamburg, Hamburg, Germany}}
\thanks{Email: \href{mailto:gerald@monstrous-moonshine.de}{\nolinkurl{gerald@monstrous-moonshine.de}}, \href{mailto:math@moeller-sven.de}{\nolinkurl{math@moeller-sven.de}}}

\begin{abstract}
We classify the self-dual (or holomorphic) \svoa{}s of central charge $24$, or in physics parlance the purely left-moving, fermionic $2$-dimensional conformal field theories with just one primary field.

There are exactly $969$ such \svoa{}s under suitable regularity assumptions (essentially strong rationality) and the assumption that the shorter moonshine module $\VB$ is the unique self-dual \svoa{} of central charge~$23\shs\sfrac{1}{2}$ whose weight-$\sfrac{1}{2}$ and weight-$1$ spaces vanish. Additionally, there might be self-dual \svoa{}s arising as fake copies of $\VB$ tensored with a free fermion $F$.

We construct and classify the self-dual \svoa{}s by determining the $2$-neighbourhood graph of the self-dual \voa{}s of central charge $24$ and also by realising them as simple-current extensions of a dual pair containing a certain maximal lattice \voa{}. We show that all \svoa{}s besides $\VB\otimes F$ and potential fake copies thereof stem from elements of the Conway group $\Co_0$, the automorphism group of the Leech lattice $\Lambda$.

By splitting off free fermions $F$, if possible, we obtain the classification for all central charges less than or equal to $24$.
\end{abstract}

\vspace*{-10pt}
\maketitle

\vspace*{-10pt}
\setcounter{tocdepth}{1}
\tableofcontents
\setcounter{tocdepth}{2}


\vspace*{-10pt}
\section{Introduction}

Vertex operator (super)algebras and their representations (or modules) axiomatise the notions of chiral algebras and fusion rings of $2$-dimensional conformal field theories in physics. While \voa{}s only have bosonic (or even) fields, \svoa{}s may also possess fermionic (or odd) fields. If the fusion ring is trivial, then the vertex operator (super)algebra is said to be \emph{self-dual} (or holomorphic or meromorphic). In physical language, this means that the conformal field theory has exactly one primary field.

\medskip

In this text, we classify the self-dual \svoa{}s of central charge $24$. By splitting off free fermions, this also yields a classification for all central charges less than $24$. Throughout, we assume that the \svoa{}s satisfy a certain regularity assumption (strong rationality and positive weight grading of the irreducible untwisted or canonically twisted modules other than the algebra itself) and call them \emph{nice}.

There are exactly $70$ non-isomorphic nice, self-dual (purely even) \voa{}s $V$ of central charge $24$ assuming that the weight-$1$ Lie algebra $V_1$ is non-zero. These \voa{}s were originally found by Schellekens \cite{Sch93} and his results were recently made rigorous (see, e.g., \cite{Hoe17,MS23,MS21,HM22,LM23,DSW23}). Each such \voa{} is up to isomorphism uniquely determined by $V_1$. In addition, there is the famous moonshine module $V^\natural$ \cite{FLM88} with $V^\natural_1=\{0\}$, but it is not known whether there are further (or \emph{fake}) nice, self-dual \voa{}s $V$ of central charge $24$ with $V_1=\{0\}$.

The situation for \svoa{}s is richer. While the central charge of a nice, self-dual \voa{} is a non-negative multiple of~$8$ and the first non-lattice theories occur in central charge $24$, the analogous \svoa{}s have half-integer central charges with the first non-lattice theory appearing in central charge $15\shs\sfrac{1}{2}$. A classification up to this central charge was given in \cite{Hoe95} contingent on then unproven representation-theoretic results. The method developed there uses commutants of the simple affine \voa{}s for the Lie algebras $\so_l$ at level $1$ inside the Schellekens \voa{}s and, in principle, facilitates a classification up to central charge~$23\shs\sfrac{1}{2}$. Moreover, the classification up to central charge $16$ was known to physicists~\cite{KLT86}, and partial results in central charge~$24$ were given in \cite{Mon98}. Here, we finally give a rigorous classification up to central charge $24$ with the exception of the possible existence of \emph{fake} copies (see \autoref{sec:results} for a definition) of $\VB$ and $\VB\otimes F$ where $\VB$ is the shorter moonshine module of central charge~$23\shs\sfrac{1}{2}$ constructed in \cite{Hoe95} and $F$ the (Clifford) free-fermion \voa{} of central charge~$\sfrac{1}{2}$.

\medskip

The classification of the nice, self-dual vertex operator (super)algebras generalises the classification of the positive-definite, integral, unimodular lattices. Indeed, each such lattice defines a unique lattice theory, which is even whenever the lattice is and has central charge equal to the rank of the lattice. The rank is a non-negative multiple of $8$ in the even case and a positive integer in the odd case.

Niemeier classified the positive-definite, even, unimodular lattices of rank~$24$ \cite{Nie68,Nie73}, among them the famous Leech lattice $\Lambda$, which gives the densest sphere packing in $24$ dimensions \cite{CKMRV17}. There are exactly $24$ such lattices, and that there are so few can be attributed to the fact that they are in a sense governed by said Leech lattice \cite{CPS82,Bor85}. In higher ranks, the numbers grow very rapidly and no complete classifications are known. For example, there are over a billion lattices in the next higher rank~$32$ \cite{Kin03}.

The exactly $273$ and $665$ positive-definite, odd, unimodular lattices of rank $24$ and $25$, respectively, were classified by Borcherds \cite{Bor85}. Once again, they are controlled by the Leech lattice $\Lambda$ and beyond these ranks their numbers increase rapidly. In rank~$24$ (or any multiple of $8$) these lattices are in bijection with neighbouring pairs (Kneser $2$-neighbours) of their even counterparts, the Niemeier lattices. The corresponding $2$-neighbourhood graph was also determined in~\cite{Bor85}.

\medskip

In this text, we generalise the neighbourhood method to nice, self-dual \svoa{}s of central charge $c\in8\Ns$ (see \autoref{prop:neighbourhood}) and classify them for $c=24$ by enumerating neighbouring pairs ($\Z_2$-orbifold constructions) of nice, self-dual \voa{}s, i.e.\ the Schellekens \voa{}s. To this end, we explicitly realise the automorphism groups of the latter \cite{BLS23}.

We find $969$ such \svoa{}s, in addition to the $71$ (purely even) Schellekens \voa{}s:
\begin{thn*}[Classification, \autoref{thm:class}]
Up to isomorphism, there are exactly 968 nice, self-dual \svoa{}s of central charge~24 for which the weight-1 space or that of the canonically twisted module is non-zero.

In addition, there is the nice, self-dual \svoa{} $\VB\otimes F$ of central charge 24 obtained by tensoring the shorter moonshine module with a free fermion.
\end{thn*}
We cannot exclude the possibility that there are fake copies of $\VB\otimes F$, analogously to the moonshine uniqueness problem in the bosonic case.

We list the nice, self-dual \svoa{}s of central charge~$24$ in \autoref{table:969} and give a summary in \autoref{table:genera}. Splitting off free fermions (see \autoref{prop:holsplit}), we also obtain a classification in all central charges less than~$24$. The corresponding numbers of non-isomorphic theories are given in \autoref{table:numbers}.

In contrast to the Schellekens \voa{}s, the nice, self-dual \svoa{}s of central charge $24$ are in general not uniquely determined by their (non-zero) weight-$1$ Lie algebras.

\medskip

The central idea in the classification of the Schellekens \voa{}s in \cite{Hoe17,MS23} is that they, similarly to the Niemeier lattices, almost all come from the Leech lattice \voa{} $V_\Lambda$, now together with an element in the Conway group $\Co_0=\O(\Lambda)$. In this text, we prove the same statement for the nice, self-dual \svoa{}s of central charge $24$:
\begin{thn*}[Heisenberg Commutants, \autoref{thm:commconway}]
Let $V$ be a nice, self-dual \svoa{} of central charge 24. If the weight-1 space or that of the canonically twisted module is non-zero, then $V$ is a simple-current extension of the dual pair $V_{\Lambda_\mu}^{\hat\mu}\otimes V_L$, where $\Lambda$ is the Leech lattice, $\mu$ one of 18 conjugacy classes in $\Co_0$ and $L$ a certain positive-definite, even lattice of rank $\rk(\Lambda^\mu)$.
\end{thn*}
Here, $\Lambda_\mu=(\Lambda^\mu)^\bot$ denotes the coinvariant lattice, the orthogonal complement of the fixed-point lattice $\Lambda^\mu$. Further, $L$ is a certain maximal lattice we call the \emph{associated lattice}, and $V_{\Lambda_\mu}^{\hat\mu}$ is the corresponding \emph{Heisenberg commutant} (see \autoref{sec:cartan}).

We label the $18$ Heisenberg commutants (or the corresponding conjugacy classes $\mu\in\Co_0$) by the letters A to K and M to S (see \autoref{table:19}). The $11$ conjugacy classes A to K already appeared in the Schellekens \voa{}s. The additional seven classes M to S have Frame shapes $1^{-8}2^{16}$, $1^82^{-8}4^8$, $2^44^4$, $1^42^13^{-4}6^5$, $1^22^15^{-2}10^3$, $2^14^16^112^1$ and $1^{-24}2^{24}$ (the involution $-\!\id$), respectively. Type~S corresponds to the odd moonshine module $\VO$~\cite{DGH88,Hua96b}.

Moreover, we label the Heisenberg commutant of the \svoa{} $\VB\otimes F$ by the letter T. Like for type~S, it has trivial associated lattice $L$ so that the Heisenberg commutant is equal to the even part of $\VB\otimes F$. For type~T (and similarly for potential fake copies) the Heisenberg commutant does not come from $\Co_0=\O(\Lambda)$ but corresponds to the 2A-involution of the monster group.

This is analogous to the (bosonic) moonshine module $V^\natural$ (and potential fake copies), which equals its Heisenberg commutant labelled by L and does not come from $\Co_0$ in the above sense.

\medskip

Most \svoa{}s associated with positive-definite, odd, unimodular lattices of rank $24$ have an $\mathcal{N}\!=\!4$ superconformal structure~\cite{HM18}. The odd moonshine module $\VO$ has an $\mathcal{N}\!=\!1$ superconformal structure~\cite{DGH88,GJ22,MS23b}. It is an open question which of the other self-dual \svoa{}s found in this paper have $\mathcal{N}\!=\!1$, $2$ or $4$ superconformal structures.


\subsection*{Outline}

In \autoref{sec:svoas} we review the theory of \svoa{}s and their representations. Then we introduce the regularity assumptions necessary for this work. We also discuss some examples of \svoa{}s.

In \autoref{sec:hol} we specialise to self-dual \svoa{}s and describe in particular the representation category and characters of their even parts. The former gives rise to the neighbourhood method if $8$ divides the central charge.

Then, in \autoref{sec:voas} we review the notion of the associated lattice of a \voa{} and how to decompose a \voa{} with respect to the corresponding dual pair.

In \autoref{sec:holvoa} we use the results from the preceding section to explain an enumeration procedure for self-dual \voa{}s. As an application, we describe the classification of these \voa{}s in central charge $24$.

Applying those results also to the even subalgebra of a self-dual \svoa{} yields one of the two main tools for classifying the latter. This is developed in \autoref{sec:evensub}.

In \autoref{sec:orbifold} we study the neighbourhood method, i.e.\ $\Z_2$-orbifolds of self-dual \voa{}s, which is the other main tool used in this text.

Finally, in \autoref{sec:class} we combine the results of the two preceding sections to classify the self-dual \svoa{}s of central charge $24$.


\subsection*{Acknowledgements}

We thank Thomas Creutzig, Philipp Höhn, Ching Hung Lam, Geoffrey Mason, Brandon Rayhaun, Nils Scheithauer, Hiroki Shimakura, Yuji Tachikawa and Hiroshi Yamauchi for helpful discussions. Gerald Höhn also expresses his gratitude posthumously to his advisor, Friedrich Hirzebruch, for his support of the project.

We also thank Yuto Moriwaki for bringing our attention to related, contemporaneous work of Yuji Tachikawa and his collaborators, who in turn made us aware of work of Brandon Rayhaun in a similar direction.

Gerald Höhn received funding through a Collaboration Grant for Mathematicians (award number 355294) provided by the \emph{Simons Foundation}. Additionally, he was supported by a Research Fellowship from the \emph{Deutsche Forschungsgemeinschaft} for two years from $1995$ to $1997$, with the primary aim of advancing the present classification project.

Sven Möller acknowledges support through the Emmy Noether Programme by the \emph{Deutsche Forschungsgemeinschaft} (project number 460925688). Sven Möller was also supported by a Postdoctoral Fellowship for Research in Japan and Grant-in-Aid KAKENHI 20F40018 by the \emph{Japan Society for the Promotion of Science}.

\medskip

This manuscript is submitted in coordination with two other groups, who independently explored questions similar to the ones addressed in this work.

In \cite{BKLTZ24} the authors classify the purely left-moving, modular invariant (i.e.\ having only one primary field), fermionic $2$-dimensional conformal field theories with central charge at most $16$ with the motivation of enumerating certain heterotic string theories. The mathematical translation is the classification of self-dual \svoa{}s of central charge at most $16$, and they also employ the $2$-neighbourhood method used in this text.

In \cite{Ray23} the author classifies most (suitably regular) \voa{}s of central charge at most $24$ with not more than four irreducible modules using a coset approach similar to \cite{SR23} (or \cite{Hoe95}). A fortiori, a classification of the self-dual \svoa{}s of central charge at most $22\shs\sfrac{1}{2}$ is obtained.

To the extent that they overlap, the results obtained by the other groups agree perfectly with our own classification results.


\section{Vertex Operator Superalgebras}\label{sec:svoas}
In this section we review \svoa{}s and their representation theory and discuss certain regularity conditions that will be assumed in this text.

Vertex (operator) superalgebras generalise vertex (operator) algebras \cite{Bor86,FLM88} by allowing odd (or fermionic) fields in addition to even (or bosonic) fields. We point out that there are some subtleties regarding the precise definition and we shall explain in the following which convention we use.


\subsection{Definition}\label{sec:defi}

Vertex superalgebras $V=V^{\bar0}\oplus V^{\bar1}$ (as defined, e.g., in \cite{Li96b,Kac98}) have a \emph{parity grading} by $\Z_2=\{\bar0,\bar1\}$. The even part $V^{\bar0}$ corresponds to the bosonic fields and the odd part $V^{\bar1}$ to the fermionic ones. To obtain the notion of a \svoa{} one adds a conformal structure and some finiteness and lower-boundedness conditions on the resulting \emph{$L_0$-weight grading}, which is usually assumed to have weights in $\frac{1}{2}\Z$.

Often, one also assumes that a \svoa{} $V$ has ``correct statistics'', i.e.\ that $V^{\bar0}=\bigoplus_{n\in\Z}V_n$ and $V^{\bar1}=\bigoplus_{n\in\frac{1}{2}+\Z}V_n$. This is perhaps the most common definition of a \svoa{} and also the one used in this text.\footnote{We remark, however, that there are interesting classes of \svoa{}s that do not satisfy this assumption, like, e.g., affine \svoa{}s, which are $\Z$-graded. A classification of self-dual, $\Z$-graded \svoa{}s is discussed in \cite{EM21}.}\textsuperscript{,}\footnote{We also point to \cite{RSW23} for a discussion of this (important) subtlety in the context of fermionic $2$-dimensional conformal field theories.}\textsuperscript{,}\footnote{The \svoa{} $V$ having ``correct statistics'' is closely related to the representation category of $V^{\bar0}$ satisfying a certain evenness property discussed in \autoref{sec:mtc}.}
For definiteness we state the complete definition:
\begin{defi}[\SVOA{}]
A \emph{\svoa{}} $V=(V,Y,\vac,\omega)$ of \emph{central charge} $c\in\C$ is a $\frac{1}{2}\Z$-graded super vector space
\begin{equation*}
V=\bigoplus_{n\in\frac{1}{2}\Z}V_n=V^{\bar0}\oplus V^{\bar1}
\end{equation*}
satisfying $V_n=\{0\}$ for $n$ sufficiently small, $\dim(V_n)<\infty$ for all $n\in\frac{1}{2}\Z$ and
\begin{equation*}
V^{\bar0}=\bigoplus_{n\in\Z}V_n\quad\text{and}\quad V^{\bar1}=\!\!\bigoplus_{n\in\frac{1}{2}+\Z}\!\!V_n,
\end{equation*}
equipped with a linear map called the \emph{state-field correspondence}
\begin{align*}
Y\colon V&\to\End(V)[[x,x^{-1}]],\\
v&\mapsto Y(v,x)=\sum_{n\in\Z}v_nx^{-n-1},\quad\text{with \emph{modes} }v_n\in\End(V),
\end{align*}
(the $Y(v,x)$ are called \emph{fields} or \emph{vertex operators}) such that for $u,v\in V$, $u_nv=0$ for $n$ sufficiently large and with two distinguished even vectors, the \emph{vacuum vector} $\vac\in V_0$ and the \emph{conformal vector} $\omega\in V_2$, with the convention
\begin{equation*}
Y(\omega,x)=\sum_{n\in\Z}L_nx^{-n-2},
\end{equation*}
i.e.\ $L_n=\omega_{n+1}$, satisfying the following conditions: the \emph{vacuum axioms}
\begin{equation*}
Y(\vac,x)=\id_V\quad\text{and}\quad Y(v,x)\vac\in V[[x]],\; Y(v,x)\vac|_{x=0}=v
\end{equation*}
for all $v\in V$, the \emph{translation axiom}
\begin{equation*}
[L_{-1},Y(v,x)]=Y(L_{-1}v,x)=\partial_xY(v,x)
\end{equation*}
for all $v\in V$, the \emph{Virasoro relations of central charge} $c\in\C$
\begin{equation*}
[L_m,L_n]=(m-n)L_{m+n}+\frac{m^3-m}{12}\delta_{m+n,0}\,c\,\id_V
\end{equation*}
for $m,n\in\Z$, as well as
\begin{equation*}
L_0v=nv
\end{equation*}
($n$ is called the \emph{$L_0$-weight}) for all $v\in V_n$, $n\in\frac{1}{2}\Z$ and the \emph{Jacobi identity}
\begin{align*}
&x_0^{-1}\delta\left(\frac{x_1-x_2}{x_0}\right)Y(u,x_1)Y(v,x_2)-(-1)^{|u||v|}x_0^{-1}\delta\left(\frac{x_2-x_1}{-x_0}\right)Y(v,x_2)Y(u,x_1)\\
&=x_2^{-1}\delta\left(\frac{x_1-x_0}{x_2}\right)Y(Y(u,x_0)v,x_2)
\end{align*}
for all $\Z_2$-homogeneous $u,v\in V$ with parity $|u|$, $|v|\in\Z_2$ where $\delta(x)=\sum_{n\in\Z}x^n$ and $(x_i-x_j)^n$ is to be expanded as a formal power series in $x_j$.
\end{defi}

If the odd part vanishes, i.e.\ if $V^{\bar1}=\{0\}$, then the definition reduces to that of a \voa{}. More generally, the even part $V^{\bar0}$ of a \svoa{} $V$ is a \voa{} and the odd part $V^{\bar1}$ is a $V^{\bar0}$-module.

\begin{conv}\label{conv:noteven}
Unless otherwise noted, we shall always assume the odd part $V^{\bar1}$ of a \svoa{} $V$ to be non-zero (similar to the distinction between even and odd integral lattices, see below).
\end{conv}

An automorphism of a \svoa{} $V$ is a vector-space automorphism $g$ satisfying $gY(v,x)g^{-1}=Y(gv,x)$ for all $u\in V$ and fixing $\vac$ and $\omega$. Then $g$ must preserve the $L_0$-weight and hence the parity. Every \svoa{} $V$ admits the \emph{canonical automorphism} $\sigma$ of order~$2$ in the centre of $\Aut(V)$ defined by $\sigma v=(-1)^{|v|}v$ for any $\Z_2$-homogeneous $v\in V$ with parity $|v|$. Then $V^{\bar0}=V^\sigma$ and $V^{\bar1}$ is the eigenspace associated with eigenvalue $-1$.

\medskip

By convention, all vertex operator (super)algebra extensions in this text will be \emph{conformal} extensions, i.e.\ extensions with the same conformal vector, unless noted otherwise. Similarly, by a vertex operator sub(super)algebra we mean a \emph{full} subalgebra, i.e.\ one with the same conformal vector.


\subsection{Modules}\label{sec:modules}

As for \voa{}s, there exist notions of weak, admissible and ordinary modules for \svoa{}s $V$ (in decreasing generality). Similarly, there are $g$-twisted versions of these for all $g\in\Aut(V)$ of finite order (see, e.g., \cite{DZ06,DZ05}). In this text, we shall only consider untwisted modules. Moreover, we shall only be dealing with \svoa{}s whose weak, admissible and ordinary (untwisted or twisted) modules are a direct sum of irreducible (ordinary) modules (see the notions of rationality and regularity discussed below). In that case, there is no distinction between weak and admissible modules, and for irreducible modules all three notions coincide.

Note that any untwisted or canonically twisted (i.e.\ $\sigma$-twisted) $V$-module is also an untwisted module for the even part $V^{\bar0}$.

We mention two subtleties regarding the definition of modules for \svoa{}s. First, consider admissible untwisted modules. The \svoa{} $V$ comes equipped with a decomposition into an even and an odd part, and the even part is distinguished by the fact that it contains the vacuum vector~$\vac$ (and the conformal vector $\omega$). Similarly, any admissible $V$-module $W$ may be equipped with a compatible parity $\Z_2$-grading, namely such that the vertex operator modes of even and odd elements of $V$ are even and odd maps of $W$, respectively. This may be stated as saying that any untwisted admissible $V$-module is $\sigma$-stable (see \autoref{sec:repcat}). However, this $\Z_2$-grading is not unique and there is in general no distinguished way to choose it.
\begin{conv}\label{conv:parity}
We shall not treat a compatible $\Z_2$-grading as part of the structure of a module for a \svoa{}.
\end{conv}

For admissible $\sigma$-twisted $V$-modules, the situation is slightly more complicated. In that case, there may be modules that are $\sigma$-stable and those that are not. The former are called $\sigma$-twisted \emph{super} $V$-modules in \cite{DNR21}, and only these afford a $\Z_2$-grading compatible with the one on $V$ in the above sense (again, see \autoref{sec:repcat}).


\subsection{Regularity}\label{sec:reg}

The \svoa{}s in this text will be assumed to satisfy a number of regularity conditions. First, recall for \voa{}s the notions of regularity and rationality (as defined in, e.g., \cite{DLM97}), $C_2$-cofiniteness, self-contragredience, CFT-type and simplicity. Note that a self-contragredient \voa{} of CFT-type is automatically simple. Also, under the assumption of CFT-type, regularity is equivalent to rationality and $C_2$-cofiniteness \cite{Li99,ABD04}. Recall further that if a \voa{} $V$ is rational and $C_2$-cofinite, then the central charge $c$ and the $L_0$-weights of all irreducible $V$-modules are rational numbers \cite{DLM00}. Rational or $C_2$-cofinite \voa{}s only have finitely many irreducible modules up to isomorphism \cite{DLM98,DLM00}.

All these notions can be defined analogously for \svoa{}s. Specifically, a vertex operator (super)algebra $V$ is of CFT-type if $V_n=\{0\}$ for $n<0$ and $\dim(V_0)=1$. A vertex operator (super)algebra $V$ is called $C_2$-cofinite if $V/\langle v_{-2}u\,|\,v,u\in V\rangle$ is finite-dimensional. Regularity (rationality) means that every weak (admissible) $V$-module is a direct sum of irreducible ordinary modules. Again, if a \svoa{} $V$ is of CFT-type, then regularity is equivalent to rationality and $C_2$-cofiniteness \cite{HA15} (see also \cite{DNR21}). Also for a \svoa{} $V$, if $V$ is rational, $\sigma$-rational (see, e.g., \cite{DZ06}) and $C_2$-cofinite, then the central charge $c$ and the $L_0$-weights of all irreducible untwisted and $\sigma$-twisted $V$-modules are rational numbers \cite{DZ05,DNR21}. Rational \svoa{}s only have finitely many irreducible modules up to isomorphism \cite{DZ06}.

A vertex operator (super)algebra is called \emph{\strat{}} if it is simple, rational, $C_2$-cofinite, self-contragredient and of CFT-type (and hence also regular). For a \voa{} $V$, one key consequence of this is Zhu's modular invariance result \cite{Zhu96} and that the representation category $\Rep(V)$ of $V$ is a modular tensor category \cite{Hua08b} (see \autoref{sec:mtc}).

Moreover, a simple \voa{} $V$ is said to satisfy the \emph{positivity condition} if all irreducible modules not isomorphic to $V$ have positive $L_0$-grading. For a \svoa{} we slightly modify this definition by requiring that also all irreducible $\sigma$-twisted modules have positive $L_0$-grading (see, e.g., \cite{DRY22}).

\medskip

A \svoa{} $V$ satisfying these regularity conditions is essentially equivalent to its even part $V^{\bar0}$ satisfying them. In one direction:
\begin{prop}
Let $V$ be a \strat{} \svoa{}. Then $V^{\bar0}$ is \strat{}, and $V^{\bar1}$ is an irreducible $V^{\bar0}$-module with positive $L_0$-grading.
\end{prop}
\begin{proof}
Suppose that the \svoa{} $V$ is \strat{}. The even part $V^{\bar0}$ is given as the \fpvosa{} $V^\sigma$ under the canonical automorphism $\sigma$ of $V$. Since $V$ is simple, so is $V^{\bar0}$ (and $V^{\bar1}$ must be an irreducible $V^{\bar0}$-module). Clearly, $V^{\bar0}$ must be of CFT-type and $V^{\bar1}$ has only positive $L_0$-weights. Then, it is not difficult to see that $V^{\bar0}$ is self-contragredient. As explained in \cite{DNR21,DRY22}, one can show that $V^{\bar0}$ is regular using the orbifold results in \cite{Miy15,CM16}.
\end{proof}
Moreover \cite{DNR21}:
\begin{prop}
Let $V$ be a \strat{} \svoa{} satisfying the positivity condition. Then $V^{\bar1}$ is a simple current of order~2 as a $V^{\bar0}$-module.
\end{prop}
Here, $V^{\bar1}$ being a simple current of order~$2$ simply means that the fusion product (see \autoref{sec:mtc}) of $V^{\bar1}$ with itself is isomorphic to $V^{\bar0}$.

Conversely:
\begin{prop}
Let $V$ be a \svoa{} whose even part $V^{\bar0}$ is \strat{} and whose odd part $V^{\bar1}$ is a simple current of order~2 (and in particular irreducible and self-contragredient) as $V^{\bar0}$-module, with positive $L_0$-grading. Then $V$ is \strat{}.
\end{prop}
\begin{proof}
The assumptions on the $L_0$-grading imply that $V$ is of CFT-type. It is not difficult to see that $V$ is simple and self-contragredient. The rationality of $V^{\bar0}$ implies the rationality of $V$ \cite{DH12}. Moreover, $V$ is $C_2$-cofinite by \cite{ABD04}.
\end{proof}
We note that a priori (without using \cite{Miy15,CM16}, see \cite{DRY22}) the assumption of rationality of $V^{\bar0}$ is stronger than the assumption of rationality of $V$. Indeed, the former is equivalent to $V$ being rational and $\sigma$-rational (under some assumptions) \cite{DH14}. If we wanted to avoid using the heavy machinery of \cite{Miy15,CM16}, for the purposes of this text, we could simply \emph{assume} the regularity of $V^{\bar0}$.

We can also relate the positivity conditions for a \svoa{}~$V$ and for its even part $V^{\bar0}$:
\begin{prop}
Let $V$ be a \strat{} \svoa{} satisfying the positivity condition. Then also $V^{\bar0}$ satisfies the positivity condition.
\end{prop}
\begin{proof}
This follows from the result in \cite{DRY22} that every irreducible $V^{\bar0}$-module appears as a submodule of an irreducible untwisted or $\sigma$-twisted $V$-module.
\end{proof}
Conversely, by definition, any untwisted or $\sigma$-twisted $V$-module is an untwisted $V^{\bar0}$-module. Hence, if $V^{\bar0}$ is simple, rational and satisfies the positivity condition, then also $V$ satisfies the positivity condition.

\medskip

For the purposes of this text we define:
\begin{defi}
A vertex operator (super)algebra is called \emph{nice} if it is \strat{} and satisfies the positivity condition.
\end{defi}

Summarising the above results, if $V$ is a nice \svoa{} (recalling also \autoref{conv:noteven} that $V$ is not purely even), then:
\begin{enumerate}
\item The central charge $c$ and the $L_0$-weights of all $V^{\bar0}$-modules and of all untwisted and $\sigma$-twisted $V$-modules are in $\Q$. Moreover, $c>0$.\footnote{There is the nice (and self-dual) \voa{} $\C\vac$ of central charge~$0$, but it is purely even.}
\item $V^{\bar0}$ is nice.
\item $\Rep(V^{\bar0})$ is a pseudo-unitary modular tensor category (see \autoref{sec:mtc}).
\item $V^{\bar1}$ is a simple current of order~$2$ (and in particular irreducible and self-contragredient) as $V^{\bar0}$-module.
\end{enumerate}
The fact that the central charge must be positive follows from the non-existence of certain holomorphic modular forms of negative weight (see, e.g., \cite{DM04b} for a proof under the condition that $V^{\bar0}$ is nice; cf.\ \autoref{sec:chars}).

\medskip

A simple vertex operator (super)algebra $V$ is called \emph{self-dual} (or \emph{holomorphic}) if $V$ itself is the only irreducible $V$-module (cf.\ \autoref{conv:parity}), i.e.\ if $\Rep(V)\cong\Vect$ in the rational case. These \svoa{}s will be the main focus of this text.


\subsection{Modular Tensor Categories and Evenness}\label{sec:mtc}

In this section we briefly review modular tensor categories (over $\C$), which appear, for instance, as the representation categories of \strat{} \voa{}s, and allude to a subtlety regarding the signs of the categorical dimensions. The special case of the representation category of the even part $V^{\bar0}$ of a nice \svoa{}~$V$ shall be described in more detail in \autoref{sec:repcat}.

\emph{Modular tensor categories} are fusion categories (certain rigid, semisimple, $\Vect$-enriched monoidal categories with simple tensor unit $\mathbf{1}$ and only finitely many isomorphism classes of simple objects) equipped with a ribbon structure (i.e.\ compatible braiding, twist and duality structures) satisfying a certain non-degeneracy condition on the braiding (see, e.g., \cite{BK01,Tur10,EGNO15}). Dropping this non-degeneracy requirement, one speaks of premodular or ribbon fusion categories.

Huang showed that the representation category $\Rep(V)$ of a \strat{} \voa{} $V$ admits the structure of a modular tensor category \cite{Hua08b}. For instance, the tensor-product bifunctor $\otimes$ on $\Rep(V)$ is taken to be the Huang-Lepowsky tensor (or fusion) product $\boxtimes_{P(1)}$ (or simply denoted by $\boxtimes$), the unit object~$\mathbf{1}$ is the \voa{} $V$ itself, the dual objects are the contragredient (or gradewise) duals $W^*=W'$ and the twist isomorphisms are given by $\theta_W=\e^{2\pi\i L_0}$ for a $V$-module $W$.

\medskip

We now discuss a certain sign ambiguity for the categorical dimensions (or the pivotal structure) in a modular tensor category.

In a rigid monoidal category, an object $X$ is called \emph{invertible} (or \emph{simple current}) if the corresponding evaluation and coevaluation morphisms are isomorphisms. Then, $X\otimes X^*\cong X^*\otimes X\cong\mathbf{1}$ where $X^*$ denotes the dual object of $X$. In a fusion category, all invertible objects are simple and an object is invertible if and only if its Frobenius-Perron dimension (which is always a positive real number) is $1$. Given a fusion category $\mathcal{C}$, the isomorphism classes of invertible objects form an abelian group whose group structure is induced by the tensor product. This group is often called $G(\mathcal{C})$, but we simply denote it by $A$.

Modular tensor categories (and more generally ribbon categories) are examples of spherical pivotal categories, with the pivotal structure determined by the braiding and the twist. More precisely, with the braiding fixed, there is a bijection between the possible spherical pivotal and ribbon structures. In spherical pivotal categories, the categorical dimensions take values in $\R$ \cite{Mue03}. Given a modular tensor category $\mathcal{C}$, the spherical pivotal structures on the underlying braided fusion category are in bijection with $\Hom(A,\{\pm1\})\cong\Hom(A/(2A),\C^\times)$ \cite{BNRW16,GN08}, where we replaced the role of $A$ with its dual via the non-degenerate braiding.

A fusion category is called \emph{pseudo-unitary} if its global categorical dimension equals the global Frobenius-Perron dimension. Equivalently, such a fusion category admits a (unique) spherical pivotal structure with respect to which the categorical dimensions of all simple objects are positive, and hence coincide with the Frobenius-Perron dimensions \cite{ENO05}.

We call a fusion category with a choice of (spherical) pivotal structure (like a modular tensor category) \emph{even} if the categorical dimensions of all simple objects are positive (so that this category is in particular pseudo-unitary; and any pseudo-unitary fusion category can be made even).

Given a nice \voa{} $V$ (in particular satisfying the positivity condition), the representation category $\Rep(V)$, endowed with Huang's modular tensor category structure, is always even (and hence pseudo-unitary). This follows from the fact that the categorical dimensions equal the quantum dimensions (defined via the characters), which are positive under the given regularity assumptions \cite{DJX13,DLN15}.

If $V$ is a \voa{} whose representation category $\Rep(V)$ is a (not necessarily even) modular tensor category and $M$ is an invertible simple object of order~$2$, the different possible signs of the self-braiding and the ribbon twist on $M$ lead to variations on the definition of a \svoa{} (like $\Z$-graded \svoa{}s), realised on $V\oplus M$ \cite{CKM17,CKL20,RSW23}. But if $\Rep(V)$ is even (for instance, because $V$ is nice), then any extension of $V$ to a \svoa{} must automatically have ``correct statistics'', which are the \svoa{}s we consider in this text (see \autoref{sec:defi}).

\medskip

Of particular interest to us (see, e.g., \autoref{sec:holcat} and \autoref{sec:cartan}) will be pointed modular tensor categories. A fusion category is called \emph{pointed} if all its simple objects are invertible.

It follows from the abelian cohomology theory of Eilenberg and Mac Lane \cite{EML53,EML54} that pointed braided fusion categories up to braided monoidal equivalence are in natural bijection with pairs $(A,q)$ of a finite abelian group $A$ and a quadratic form $q\colon A\to\Q/\Z$ \cite{JS93}. (An intimately related statement in the language of abelian intertwining algebras is given in \cite{DL93}; see also \cite{MS89}.)

Let us denote the braided fusion category associated with the pair $(A,q)$ by $\mathcal{C}(A,q)$. The isomorphism classes of simple objects of $\mathcal{C}(A,q)$ are given by $W^\alpha$, $\alpha\in A$, with tensor product $W^\alpha\otimes W^\beta\cong W^{\alpha+\beta}$ for $\alpha$, $\beta\in A$. The quadratic form $q$ describes the self-braiding and the associated bilinear form (defined by $b(\alpha,\beta)\coloneqq q(\alpha+\beta)-q(\alpha)-q(\beta)$ for all $\alpha$, $\beta\in A$) the double-braiding.

All pointed braided fusion categories are pseudo-unitary, i.e.\ there is a (spherical) pivotal structure on $\mathcal{C}(A,q)$ such that the categorical dimensions $\dim(W^\alpha)$ of all simple objects are equal to $1$. Equivalently, we can choose a ribbon structure, thus making $\mathcal{C}(A,q)$ into a premodular category, such that the ribbon twist coincides with the self-braiding on all simple objects, i.e.\ $\theta_{W^\alpha}=\e^{2\pi\i q(\alpha)}\id_{W^\alpha}$. In the following, we denote by $\mathcal{C}(A,q)$ this even choice of premodular category. The $S$- and $T$-matrix of $\mathcal{C}(A,q)$ are $S_{W^\alpha,W^\beta}=\e^{-2\pi\i b(\alpha,\beta)}$ and $T_{W^\alpha,W^\beta}=\delta_{\alpha,\beta}\e^{2\pi\i q(\alpha)}$.\footnote{The definitions of the $S$- and $T$-matrix used here are essentially the inverses (or Hermitian transposes) of the definitions in \cite{EGNO15}. This convention is compatible with the $S$- and $T$-matrix obtained from Zhu's modular invariance of \voa{} characters \cite{DLN15}.}

If the quadratic form $q$ on $A$ is non-degenerate, i.e.\ if the bilinear form $b$ associated with $q$ is non-degenerate, we call $(A,q)$ a discriminant form (see, e.g., \cite{Nik80}). In that case, $\mathcal{C}(A,q)$ is an (even) modular tensor category. Every even, pointed modular tensor category $\mathcal{C}(A,q)$ can be realised as the representation category of a lattice \voa{} $V_L$ for a positive-definite, even lattice $L$ with discriminant form $L'/L\cong(A,q)$ \cite{DL93,Nik80} (see \autoref{sec:latsvoa}), and some also appear in the form of twisted Drinfeld doubles $\mathcal{D}_\omega(G)$ in the context of cyclic orbifolds of nice, self-dual \voa{}s \cite{Moe16,EMS20a} (see \autoref{sec:graphmeth}).

Finally, recalling the discussion above, we note that the different modular tensor category structures afforded by the braided fusion category $\mathcal{C}(A,q)$ are in bijection with $\Hom(A,\{\pm1\})\cong\Hom(A,(\frac{1}{2}\Z)/\Z)$. More concretely, the quadratic form $q\colon A\to(\frac{1}{2}\Z)/\Z$, which describes the self-braiding, differs from the ribbon twist by a homomorphism in $\Hom(A,\{\pm1\})$ given by the categorical dimensions $\dim(W^\alpha)$ so that the twist is also described by a quadratic form $q_\theta\colon A\to(\frac{1}{2}\Z)/\Z$, which has the same associated bilinear form $b=b_\theta$ as $q$. The sign ambiguity in the categorical dimensions is nothing but the failure, in general, of a fixed bilinear form to define a unique quadratic form.

In the context of a \voa{} $V$ with pointed representation category $\Rep(V)$ (equivalently described by an abelian intertwining algebra), the possible inequality of the quadratic forms $q$ and $q_\theta$ was discussed in \cite{Moe16,EMS20a} (see also \cite{Car12b,Car20}, from where we borrow and generalise the notion of evenness).


\subsection{Representation Category}\label{sec:repcat}

In the following, we briefly describe the representation category $\Rep(V^{\bar0})$ of the even part of a nice \svoa{} $V$ \cite{DNR21}, which fits well into the framework of the \emph{16-fold way} \cite{BGHNPRW17,LKW17b,JR21}, a statement about minimal modular extensions of slightly degenerate modular tensor categories. In the special case of interest to this paper, where the \svoa{} $V$ is self-dual, i.e.\ $\Rep(V)\cong\Vect$, the category $\Rep(V^{\bar0})$, which is a minimal modular extension of $\sVect$, and the $16$-fold way were already described in \cite{Kit06} (see also \cite{Hoe95}). This shall be discussed in more detail in \autoref{sec:holcat}.

Let $V$ be a \svoa{} and recall that $\sigma$ denotes the canonical involution of $V$. Let $M=(M,Y_M)$ be an untwisted or canonically twisted (i.e.\ $\sigma$-twisted) $V$-module. Then we define
\begin{equation*}
M\circ\sigma\coloneqq(M,Y_{M\circ\sigma})\quad\text{with}\quad Y_{M\circ\sigma}(\,\cdot\,,x)\coloneqq Y_M(\sigma\,\cdot\,,x),
\end{equation*}
which is again an untwisted or $\sigma$-twisted module, respectively. The module $M$ is called $\sigma$-stable if $M\circ\sigma\cong M$, or equivalently if there is a linear isomorphism $\phi_\sigma\colon M\to M$ satisfying
\begin{equation*}
\phi_\sigma Y_M(v,x)\phi_\sigma^{-1}=Y_M(\sigma v,x)
\end{equation*}
for all $v\in V$. If $M$ is irreducible, then by Schur's lemma, this $\phi_\sigma$ is unique up to a non-zero scalar, and since $\langle\sigma\rangle$ is cyclic, we may impose the condition that $\phi_\sigma^2=1$ so that we obtain a representation of $\langle\sigma\rangle$ on $M$. There is still a remaining freedom of replacing $\phi_\sigma$ by $-\phi_\sigma$ on each $\sigma$-stable, irreducible untwisted or $\sigma$-twisted $V$-module, except on $V$ itself, where $\phi_\sigma=\sigma$ is fixed (see the discussion in \autoref{sec:modules}).

Given a $\sigma$-stable, irreducible untwisted or $\sigma$-twisted $V$-module $M$, we define $M^{\bar0}$ and $M^{\bar1}$ as the eigenspaces of $\phi_\sigma$ associated with eigenvalues $1$ and $-1$, respectively. Recall that all untwisted $V$-modules $M$ are $\sigma$-stable. If such a module $M$ is irreducible, then the $L_0$-weights of $M$ lie in a coset of $\frac{1}{2}\Z$ and $M^{\bar0}$ and $M^{\bar1}$ correspond to the two resulting cosets of $\Z$. For an irreducible $\sigma$-twisted $V$-module $M$, the $L_0$-weights lie in a coset of $\Z$, and $M$ may or may not be $\sigma$-stable.

Now, let $V$ be a nice \svoa{}. Then $\Rep(V^{\bar0})$ is an even (and hence pseudo-unitary) modular tensor category and the simple objects, i.e.\ the irreducible $V^{\bar0}$-modules, are exactly given by:
\begin{enumerate}
\item\label{item:class1} $M^{\bar0}$ and $M^{\bar1}$ for each irreducible $V$-module $M$,
\item\label{item:class2} $M^{\bar0}$ and $M^{\bar1}$ for each $\sigma$-stable, irreducible $\sigma$-twisted $V$-module $M$,
\item\label{item:class3} $M$ for each pair of $\sigma$-unstable, irreducible $\sigma$-twisted $V$-modules $\{M,M\circ\sigma\}$.
\end{enumerate}
Note that in the last case $M$ and $M\circ\sigma$, while not isomorphic as $V$-modules, are isomorphic as $V^{\bar0}$-modules.

\medskip

Up to a $16$-fold ambiguity, it is possible to describe the modular tensor category $\Rep(V^{\bar0})$ in terms of the full subcategory $\Rep^{\bar0}(V^{\bar0})$ generated by the irreducible $V^{\bar0}$-submodules of the untwisted $V$-modules, i.e.\ by the modules under item~\eqref{item:class1}. The category $\Rep^{\bar0}(V^{\bar0})$ is closed under the tensor product and hence a premodular subcategory of $\Rep(V^{\bar0})$. However, it is no longer modular as it has the non-trivial Müger centre $\sVect$, generated by the simple objects $V^{\bar0}$ and $V^{\bar1}$. On the other hand, one can show that $\Rep^{\bar0}(V^{\bar0})$ is the Müger centraliser of $V^{\bar1}$ in $\Rep(V^{\bar0})$, and hence $\Rep(V^{\bar0})$ is a minimal modular extension of $\Rep^{\bar0}(V^{\bar0})$ \cite{DNR21}.

In general, a pseudo-unitary (or even) premodular category is called super-modular (or slightly degenerate) if its Müger centre is $\sVect$. It was conjectured in \cite{BGHNPRW17} that any super-modular category has exactly $16$ pseudo-unitary (or even) minimal modular extensions up to braided monoidal equivalence. This conjecture was proved in \cite{LKW17b} under the assumption that there is at least one such extension, and the existence of an extension was established in \cite{JR21}.

It is, in fact, not difficult to list the $16$ minimal modular extensions of $\Rep^{\bar0}(V^{\bar0})$, of which $\Rep(V^{\bar0})$ is one. The other extensions can be realised by considering the \voa{} $V\otimes F^{\otimes l}$ for $l\in\Z$, where $F$ is the Clifford \svoa{} we introduce in \autoref{sec:split}. Indeed, the premodular categories $\Rep^{\bar0}(V^{\bar0})$ and $\Rep^{\bar0}((V\otimes F^{\otimes l})^{\bar0})$ are braided monoidally equivalent and the $\Rep((V\otimes F^{\otimes l})^{\bar0})$ for $l\in\Z$ realise all minimal modular extensions, where $\Rep((V\otimes F^{\otimes l})^{\bar0})\cong\Rep((V\otimes F^{\otimes k})^{\bar0})$ if and only if $l=k\pmod{16}$.


\subsection{Free Fermions}\label{sec:split}

In this section we introduce a well-known free-field \svoa{} and state a related splitting result that has important consequences for the structure theory of \svoa{}s.

The \emph{Clifford \svoa{}} $F$ \cite{KW94} is one of two important \svoa{}s describing free fermions, the other being the symplectic fermion \svoa{} \cite{Kau00}. $F$ is nice, self-dual and has central charge $c=\sfrac{1}{2}$. The tensor square $F^{\otimes 2}$ is often called the $bc$-ghost system in physics parlance, and by the boson-fermion correspondence it is isomorphic to $V_\Z$, the \svoa{} associated with the odd, unimodular lattice $\Z$ called the standard lattice (see \autoref{sec:latsvoa}).

Moreover, $F$ may be decomposed as $F\cong L(\sfrac{1}{2},0)\oplus L(\sfrac{1}{2},\sfrac{1}{2})$ in terms of the irreducible modules $F^{\bar0}\cong L(\sfrac{1}{2},0)$ and $F^{\bar1}\cong L(\sfrac{1}{2},\sfrac{1}{2})$ for the simple Virasoro \voa{} $L(\sfrac{1}{2},0)$ at central charge $\sfrac{1}{2}$, also called the Virasoro minimal model for $(p,q)=(3,4)$ or the critical Ising model in physics.

\medskip

The tensor power $F^{\otimes l}$, $l\in\N$, is a nice, self-dual \svoa{} (a \voa{} for $l=0$) of central charge $c=\sfrac{l}{2}$. It was shown in \cite{Hoe95} that the even part of $F^{\otimes l}$ is isomorphic to
\begin{equation*}
(F^{\otimes l})^{\bar0}\cong\begin{cases}
L(\sfrac{1}{2},0)&\text{if }l=1,\\
V_{2\Z}&\text{if }l=2,\\
L_{\so_l}(1,0)&\text{if }l\geq3,\\
\end{cases}
\end{equation*}
where $L_{\g}(k,0)$ denotes the simple affine \voa{} associated with the simple Lie algebra $\g$ at level $k\in\Ns$, also called a Wess-Zumino-Witten model (and $V_{2\Z}$ denotes the \voa{} associated with the even sublattice $2\Z$ of the standard lattice $\Z$). We remark that
\begin{equation*}
L_{\so_l}(1,0)\cong\begin{cases}
L_{A_1}(2,0)&\text{if }l=3,\\
L_{B_{(l-1)/2}}(1,0)&\text{if }l\geq5\text{ odd},\\[+3pt]
L_{A_1}(1,0)^{\otimes2}\cong V_{A_1^2}&\text{if }l=4,\\
L_{A_3}(1,0)\cong V_{A_3}&\text{if }l=6,\\
L_{D_{l/2}}(1,0)\cong V_{D_{l/2}}&\text{if }l\geq8\text{ even},\\
\end{cases}
\end{equation*}
where we applied the exceptional Lie algebra isomorphisms and $V_{A_1^2}$, $V_{A_3}$ and $V_{D_{l/2}}$ denote the lattice \voa{}s associated with the corresponding root lattices (see \autoref{sec:latsvoa}).

As we have seen in \autoref{sec:repcat}, the representation category of $(F^{\otimes l})^{\bar0}$ only depends on $l$ modulo $16$. It will be described in more detail in \autoref{sec:holcat}.

\medskip

In the following, we describe a well-known result that facilitates the splitting off of free fermions from a \svoa{} (see, e.g., \cite{GS88,Hoe95,Tam99,DH12}).

Let $V$ be a simple, self-contragredient \svoa{} of CFT-type. We denote by $\langle V_{1/2}\rangle$ the vertex subsuperalgebra of $V$ generated by the vectors of $L_0$-weight $\sfrac{1}{2}$ and equip it with a conformal structure of central charge $\dim(V_{1/2})/2$, as described in \cite{DH12}. Then this \svoa{} is isomorphic to tensor copies of the Clifford \svoa{} $F$. Indeed:
\begin{prop}\label{prop:split}
Let $V$ be a simple, self-contragredient \svoa{} of CFT-type of central charge $c$. Then
\begin{equation*}
\langle V_{1/2}\rangle\cong F^{\otimes l}
\end{equation*}
with $l=\dim(V_{1/2})$. In particular, $\langle V_{1/2}\rangle$ is nice, self-dual and of central charge $\sfrac{l}{2}$. Moreover, $V$ decomposes into a tensor product
\begin{equation*}
V\cong F^{\otimes l}\otimes\bar{V}
\end{equation*}
where $\bar{V}\coloneqq\Com_V(\langle V_{1/2}\rangle)$ is of central charge $c-\sfrac{l}{2}$ and satisfies $\bar{V}_{1/2}=\{0\}$.
\end{prop}

Here, $\Com_V(A)$ denotes the \emph{commutant} (or \emph{centraliser} or \emph{coset}) of some subset $A$ of $V$ \cite{FZ92} (see also \cite{LL04}). We call $\bar{V}$ the \emph{stump} of $V$. We remark that the stump $\bar{V}$ can be a \svoa{} or a \voa{}, i.e.\ $\bar{V}$ does not need to have any fermionic fields left.

The following result is straightforward (cf.\ \cite{DH12}):
\begin{prop}
In the situation of the above proposition, also the stump $\bar{V}$ is simple, self-contragredient and of CFT-type. Moreover,
\begin{enumerate}
\item $V$ is rational if and only if $\bar{V}$ is,
\item $V$ is $C_2$-cofinite if and only if $\bar{V}$ is,
\item $V$ is \strat{} if and only if $\bar{V}$ is,
\item $V$ is nice if and only if $\bar{V}$ is,
\item $V$ is self-dual if and only if $\bar{V}$ is.
\end{enumerate}
\end{prop}

The proposition shows in particular that if $V$ is nice, then so is $\bar{V}$ and hence $l=\dim(V_{1/2})/2$ is bounded from above by the central charge $c$ of $V$ (cf.\ \autoref{sec:cartan}, where we recall that the rank of the Lie algebra $V_1$ is also bounded by $c$).

The above results entail that in order to study \svoa{}s it suffices to look at their stumps, i.e.\ those vertex operator (super)algebras without vectors of $L_0$-weight $\sfrac{1}{2}$. Conversely, if one has studied \svoa{}s of a certain central charge, then one automatically obtains results on \svoa{}s in smaller central charges as well. We shall make this more precise in \autoref{sec:holsplit} in the setting of self-dual \svoa{}s.


\subsection{Lattice \SVOA{}s}\label{sec:latsvoa}

Examples of nice vertex operator (super)algebras are given by the vertex operator (super)algebras $V_L$ associated with positive-definite, integral lattices $L$.

A \emph{lattice} $L$ is a free $\Z$-module of finite rank equipped with a non-degenerate, symmetric bilinear form $\langle\cdot,\cdot\rangle\colon L\times L\to\Q$. A lattice is called \emph{integral} if $\langle\alpha,\beta\rangle\in\Z$ for all $\alpha,\beta\in L$. It is called \emph{even} if $\langle\alpha,\alpha\rangle/2\in\Z$ for all $\alpha\in L$, in which case it is also integral. An integral lattice that is not even is called \emph{odd}. Given an integral lattice $L$, we define the \emph{even sublattice} $L_\text{ev}=\{\alpha\in L\,|\,\langle\alpha,\alpha\rangle/2\in\Z\}$, which is a sublattice of $L$ of index $1$ if $L$ is even and $2$ if $L$ is odd.

Let $L'=\{\alpha\in L\otimes_\Z\Q\,|\,\langle\alpha,\beta\rangle\in\Z\text{ for all }\beta\in L\}$ denote the \emph{dual lattice} of~$L$. Then $L$ is integral if and only if $L\subseteq L'$. In that case, the quotient $L'/L$ is a finite abelian group. If $L$ is even, the \emph{norm} map $q_L\colon L'/L\to\Q/\Z,\alpha+L\mapsto\langle\alpha,\alpha\rangle/2+\Z$ is a non-degenerate quadratic form on $L'/L$ and thus endows $L'/L=(L'/L,q_L)$ with the structure of a \emph{discriminant form} (see, e.g., \cite{Nik80}). An integral lattice is called \emph{unimodular} if $L'=L$, i.e.\ if $L'/L$ is trivial.

If $L$ is an odd lattice and $L_\text{ev}$ its even sublattice, then $|(L_\text{ev})'/L_\text{ev}|=4|L'/L|$ and $L$ is an index-$2$ extension of $L_\text{ev}$ corresponding to an element in $(L_\text{ev})'/L_\text{ev}$ of order~$2$ and norm $\sfrac{1}{2}\pmod{1}$. Indeed, choose some $h\in L\setminus L_\text{ev}$. Then $\langle h,h\rangle/2\in\sfrac{1}{2}+\Z$, $2h\in L_\text{ev}$ and $L=L_\text{ev}\cup(h+L_\text{ev})$. Moreover, $L'=\{\alpha\in(L_\text{ev})'\,|\,\langle h,\alpha\rangle\in\Z\}$ and $(L_\text{ev})'\setminus L'=\{\alpha\in(L_\text{ev})'\,|\,\langle h,\alpha\rangle\in\sfrac{1}{2}+\Z\}$.

Similarly, $L'$ has index $2$ in $(L_\text{ev})'$. Choose some $\tilde{h}\in(L_\text{ev})'\setminus L'$. Then $2\tilde{h}\in L'$ and $(L_\text{ev})'=L'\cup(\tilde{h}+L')$. Moreover, $\langle h,\tilde{h}\rangle\in\sfrac{1}{2}+\Z$, $L_\text{ev}=\{\alpha\in L\,|\,\langle\tilde{h},\alpha\rangle\in\Z\}$ and $L\setminus L_\text{ev}=\{\alpha\in L\,|\,\langle\tilde{h},\alpha\rangle\in\sfrac{1}{2}+\Z\}$.

\medskip

For a positive-definite, integral lattice $L$, the construction $V_L$ (see, e.g., \cite{FLM88,Kac98}) yields a nice \voa{} if $L$ is even and a nice \svoa{} if $L$ is odd. The central charge is $c=\rk(L)\in\N$. In the odd case, the even part of $V_L$ is given by $V_L^{\bar0}=V_{L_\text{ev}}$.

Recall that for an even lattice $L$ the irreducible modules of the lattice \voa{} $V_L$ are up to isomorphism given by $V_{\alpha+L}$ for $\alpha+L\in L'/L$, with $L_0$-grading in $\langle\alpha,\alpha\rangle/2+\Z$ \cite{DL93}. Moreover, they are all simple currents and satisfy the fusion rules $V_{\alpha+L}\boxtimes V_{\beta+L}\cong V_{\alpha+\beta+L}$ for $\alpha,\beta\in L$. In addition, their quantum and categorical dimensions are all equal to $1$. In other words, the modular tensor category $\Rep(V_L)$ is pointed, even and equivalent to $\mathcal{C}(L'/L)$ (see \autoref{sec:mtc}).

Also for lattice \svoa{}s $V_L$, $L$ an odd lattice, the irreducible modules up to isomorphism are $V_{\alpha+L}$ for $\alpha+L\in L'/L$ \cite{BK04}, with $L_0$-grading in $\langle\alpha,\alpha\rangle/2+\frac{1}{2}\Z$, and their quantum dimensions are all equal to $1$.

Consequently, a lattice vertex operator (super)algebra $V_L$ is self-dual if and only if the lattice $L$ is unimodular.

The canonical automorphism $\sigma$ of order $2$ on $V_L$ is given by the inner automorphism $\sigma=\e^{2\pi\i\tilde{h}(0)}$ for any choice of $\tilde{h}\in(L_\text{ev})'\setminus L'$. In particular, $V_L^\sigma=V_L^{\bar0}=V_{L_\text{ev}}$, and $V_L^{\bar1}=V_{h+L_\text{ev}}$ is the eigenspace of $\sigma$ associated with the eigenvalue $-1$. As $\sigma$ is an inner automorphism, the $\sigma$-twisted modules are described by \cite{Li96} and are given by $V_{\tilde{h}+\alpha+L}$ for $\alpha+L\in L'/L$, with $L_0$-grading in $\langle\tilde{h}+\alpha,\tilde{h}+\alpha\rangle/2+\Z$.

One can show that for lattice \svoa{}s $V_L$ all $\sigma$-twisted $V_L$-modules are $\sigma$-stable. Then, each irreducible untwisted or $\sigma$-twisted $V_L$-module decomposes into two irreducible modules for $V_L^{\bar0}=V_{L_\text{ev}}$, namely corresponding to the decomposition of a coset of $L$ into two cosets of $L_\text{ev}$, i.e.\ $V_{\alpha+L}=V_{\alpha_0+L_\text{ev}}\oplus V_{\alpha_0+h+L_\text{ev}}$ and $V_{\tilde{h}+\alpha+L}=V_{\tilde{h}+\alpha_0+L_\text{ev}}\oplus V_{\tilde{h}+\alpha_0+h+L_\text{ev}}$ for all $\alpha+L\in L'/L$. Here, $\alpha_0+L_\text{ev}$ denotes a twofold choice of $\alpha+L$ modulo $L_\text{ev}$, which except for $0_0=0$ is arbitrary (cf.\ \autoref{sec:repcat}).

\medskip

The representation category of the even part $V_L^{\bar0}=V_{L_\text{ev}}$ is the pointed modular tensor category $\Rep(V_L^{\bar0})\cong\mathcal{C}((L_\text{ev})'/L_\text{ev})$ associated with the discriminant form $(L_\text{ev})'/L_\text{ev}$ of the even lattice $L_\text{ev}$. It is a minimal modular extension of the full subcategory $\Rep^{\bar0}(V_L^{\bar0})$ corresponding to the subgroup $L'/L_\text{ev}$, which inherits the now degenerate quadratic form from $(L_\text{ev})'/L_\text{ev}$. In total, there are $16$ non-equivalent minimal modular extensions, of which eight are again pointed. The latter are characterised by the discriminant forms $((L\oplus\Z^l)_\text{ev})'/(L\oplus\Z^l)_\text{ev}$ for $l\in\N$ (modulo~$8$), where $\Z^l$ is the \emph{standard lattice} of rank $l$, which is positive-definite, odd (for $l>0$) and unimodular.

In fact, given a positive-definite, odd lattice $L$, the $2l$ many vectors of norm $\sfrac{1}{2}$ (where $l\in\N$ and $l\leq\rk(L)$) span a sublattice isomorphic to $\Z^l$. Moreover, $L$ decomposes into an orthogonal direct sum $L=\Z^l\oplus\bar{L}$ where $\bar{L}$ is the orthogonal complement of $\Z^l$ in $L$, which has rank $\rk(L)-l$ and no vectors of norm $\sfrac{1}{2}$. The lattice $\bar{L}$ is usually positive-definite and odd, but it can also be even. By considering the corresponding \svoa{}s, this splitting result can be viewed as a special case of \autoref{prop:split}.


\section{Self-Dual \SVOA{}s}\label{sec:hol}

In this section we study nice, self-dual \svoa{}s and describe the representation category of the even part, which only depends on the central charge $c$ modulo $8$. We then explain the first half of our classification strategy, based on the $2$-neighbourhood method. Finally, we describe the vector-valued character of the even part, specialising to central charge $24$.

Recall that a simple vertex operator (super)algebra $V$ is called self-dual (or holomorphic) if $V$ itself is the only irreducible $V$-module.


\subsection{Representation Category}\label{sec:holcat}

Nice, self-dual \voa{}s $V$ only exist for central charge $c\in8\N$ and the corresponding modular tensor category $\Rep(V)\cong\Vect$ is by definition trivial. For $c\leq24$ they are classified, with the exception of the uniqueness of the moonshine module $V^\natural$ (see \autoref{sec:holvoa}).

By contrast, for nice, self-dual \svoa{}s $V$ the situation is richer. By \cite{DZ05}, in addition to the unique irreducible untwisted module $V$, the \svoa{} $V$ either has one irreducible $\sigma$-twisted $V$-module that is $\sigma$-stable or two inequivalent irreducible $\sigma$-twisted $V$-modules that are not $\sigma$-stable (but isomorphic as $V^{\bar0}$-modules). Correspondingly, $\Rep(V^{\bar0})$ has either three or four irreducible modules, respectively (see \autoref{sec:repcat}).

However, the subcategory $\Rep^{\bar0}(V^{\bar0})$ is the same for all nice, self-dual \svoa{}s:
\begin{prop}
Let $V$ be a nice, self-dual \svoa{}. Then $\Rep^{\bar0}(V^{\bar0})\cong\sVect$ as (even) premodular category.
\end{prop}
Indeed, the category $\Rep^{\bar0}(V^{\bar0})$ is generated by the two irreducible $V^{\bar0}$-modules $V^{\bar0}$ and $V^{\bar1}$ with $L_0$-weights in $\Z$ and $\sfrac{1}{2}+\Z$, respectively. Since we assumed that $V$ and hence $V^{\bar0}$ is nice, $\Rep^{\bar0}(V^{\bar0})$ is the even (i.e.\ pseudo-unitary) premodular category braided monoidally equivalent to $\sVect=\mathcal{C}(\Z_2,q)$ with the degenerate quadratic form $q$ defined by $q(\bar0)=0+\Z$ and $q(\bar1)=\sfrac{1}{2}+\Z$ (see \autoref{sec:mtc}).

As explained in \autoref{sec:repcat}, $\Rep(V^{\bar0})$ must be one of the $16$ possible minimal modular extensions of $\Rep^{\bar0}(V^{\bar0})\cong\sVect$, which are described in \cite{Kit06}. We shall see that this extension is determined by the central charge $c$ of $V$ modulo~$8$ (see also \cite{Hoe95}, where the following is proved under stronger assumptions):
\begin{prop}
Let $V$ be a nice, self-dual \svoa{}. Then the central charge $c$ of $V$ is in $\frac{1}{2}\Ns$. Moreover, $\Rep(V^{\bar0})$ is a minimal modular extension of $\sVect$ and is uniquely determined by $c$ modulo~8 (listed in \autoref{table:16}).
\end{prop}
\begin{proof}
We just saw that $\Rep^{\bar0}(V^{\bar0})\cong\sVect$ as premodular category and that $\Rep(V^{\bar0})$ is an (even) minimal modular extension of $\Rep^{\bar0}(V^{\bar0})$. There are $16$ such extensions \cite{Kit06,LKW17b}, which are listed in \autoref{table:16}.

Eight of these extensions are pointed, and for definiteness we consider one of them, say $\mathcal{C}(2_{\II}^{+2})$, associated with the discriminant form $2_{\II}^{+2}$. By \cite{DNR21}, there exists an $l\in\N$, unique modulo $16$, such that $\Rep((V\otimes F^{\otimes l})^{\bar0})\cong\mathcal{C}(2_{\II}^{+2})$, where $F$ denotes the Clifford \svoa{}. As this category is even and pointed, the central charge $c+\sfrac{l}{2}$ of $(V\otimes F^{\otimes l})^{\bar0}$ must be in $\Z$ and moreover equal to the signature of the discriminant form modulo $8$ \cite{Moe16,EMS20a}. Here, $\sign(2_{\II}^{+2})=0\pmod{8}$ so that $c+\sfrac{l}{2}=0\pmod{8}$. Hence, $c\in\frac{1}{2}\Z$. Moreover, as $l$ was unique modulo $16$, this determines $c$ uniquely modulo $8$.

So, there is a well-defined map from the $16$ possible structures of $\Rep(V^{\bar0})$ to the residue class of the central charge $c$ in $(\frac{1}{2}\Z)/(8\Z)$. This map is surjective (take $V=F^{\otimes l}$ for $l$ in $\{1,\dots,16\}$) and hence injective.
\end{proof}

In \autoref{table:16} we list the structure of the modular tensor category $\Rep(V^{\bar0})$ for each value of the central charge $c$ modulo $8$, and in particular the fusion algebra and the $L_0$-weights modulo~$1$ (i.e.\ the ribbon twist) of the irreducible modules. Each category may be realised by the even part of a tensor power of the Clifford \svoa{} $F$, i.e.\ by $\Rep((F^{\otimes l})^{\bar0})$ for $l\in\Ns$, which is isomorphic to the simple affine \voa{} $L_{\so_l}(1,0)$ for $l\geq3$ (see \autoref{sec:split}).

\begin{table}[ht]\caption{The $16$ modular tensor categories for the even part of nice, self-dual \svoa{}s $V$ (cf.\ \cite{Kit06}).}
\begin{tabular}{r|l|l|l}
$c\pmod{8}$ & $\Rep(V^{\bar0})$ & Fusion alg. & Weights (mod 1)\\\hline
$0\ph$              & $\mathcal{C}(2_{\II}^{+2})$ & $\C[\Z_2\times\Z_2]$ & $[0]$, $[\sfrac{1}{2}]$, $0$, $0$ \\
$\sfrac{1}{2}$      & Ising                       & Ising                & $[0]$, $[\sfrac{1}{2}]$, $(\sfrac{1}{16})$ \\
$1\ph$              & $\mathcal{C}(4_1^{+1})$     & $\C[\Z_4]$           & $[0]$, $\sfrac{1}{8}$, $[\sfrac{1}{2}]$, $\sfrac{1}{8}$ \\
$1\shs\sfrac{1}{2}$ & $(A_1,2)$                   & Ising                & $[0]$, $[\sfrac{1}{2}]$, $(\sfrac{3}{16})$ \\
$2\ph$              & $\mathcal{C}(2_2^{+2})$     & $\C[\Z_2\times\Z_2]$ & $[0]$, $[\sfrac{1}{2}]$, $\sfrac{1}{4}$, $\sfrac{1}{4}$ \\
$2\shs\sfrac{1}{2}$ & $(B_2,1)$                   & Ising                & $[0]$, $[\sfrac{1}{2}]$, $(\sfrac{5}{16})$ \\
$3\ph$              & $\mathcal{C}(4_3^{-1})$     & $\C[\Z_4]$           & $[0]$, $\sfrac{3}{8}$, $[\sfrac{1}{2}]$, $\sfrac{3}{8}$ \\
$3\shs\sfrac{1}{2}$ & $(B_3,1)$                   & Ising                & $[0]$, $[\sfrac{1}{2}]$, $(\sfrac{7}{16})$ \\
$4\ph$              & $\mathcal{C}(2_{\II}^{-2})$ & $\C[\Z_2\times\Z_2]$ & $[0]$, $[\sfrac{1}{2}]$, $\sfrac{1}{2}$, $\sfrac{1}{2}$ \\
$4\shs\sfrac{1}{2}$ & $(B_4,1)$                   & Ising                & $[0]$, $[\sfrac{1}{2}]$, $(\sfrac{9}{16})$ \\
$5\ph$              & $\mathcal{C}(4_5^{-1})$     & $\C[\Z_4]$           & $[0]$, $\sfrac{5}{8}$, $[\sfrac{1}{2}]$, $\sfrac{5}{8}$ \\
$5\shs\sfrac{1}{2}$ & $(B_5,1)$                   & Ising                & $[0]$, $[\sfrac{1}{2}]$, $(\sfrac{11}{16})$ \\
$6\ph$              & $\mathcal{C}(2_6^{+2})$     & $\C[\Z_2\times\Z_2]$ & $[0]$, $[\sfrac{1}{2}]$, $\sfrac{3}{4}$, $\sfrac{3}{4}$ \\
$6\shs\sfrac{1}{2}$ & $(B_6,1)$                   & Ising                & $[0]$, $[\sfrac{1}{2}]$, $(\sfrac{13}{16})$ \\
$7\ph$              & $\mathcal{C}(4_7^{+1})$     & $\C[\Z_4]$           & $[0]$, $\sfrac{7}{8}$, $[\sfrac{1}{2}]$, $\sfrac{7}{8}$ \\
$7\shs\sfrac{1}{2}$ & $(B_7,1)$                   & Ising                & $[0]$, $[\sfrac{1}{2}]$, $(\sfrac{15}{16})$ \\
\end{tabular}
\label{table:16}
\end{table}

$\Rep(V^{\bar0})$ is pointed if and only if $c\in\Z$. In that case, all irreducible $V^{\bar0}$-modules are simple currents and have categorical and quantum dimensions equal to $1$. The global dimension, i.e.\ the sum of the squares of the dimensions, is $4$, as it must be \cite{DNR21}. $\Rep(V^{\bar0})$ is characterised by a discriminant form (see \autoref{sec:mtc}), whose genus symbol we list (see, e.g., \cite{CS99}). The fusion algebra is the group algebra $\C[\Z_2\times\Z_2]$ if $c\in2\Z$ or $\C[\Z_4]$ if $c\in1+2\Z$. The self-dual \svoa{} $V$ is the extension of $V^{\bar0}$ corresponding to a subgroup isomorphic to $\Z_2$ with $L_0$-weights in $\Z$ and $\sfrac{1}{2}+\Z$ (placed in square brackets), and except for $c=4\pmod{8}$ this subgroup is unique with these properties. The other two irreducible modules both have $L_0$-weights in $\sfrac{c}{8}+\Z$.

If $c\in\sfrac{1}{2}+\Z$, there are exactly three irreducible modules, of which two are simple currents with dimension $1$ and $L_0$-weights in $\Z$ and $\sfrac{1}{2}+\Z$ and one has dimension $\sqrt{2}$ and $L_0$-weights in $\sfrac{c}{8}+\Z$ (which we place in brackets). Again, the global dimension is $4$, as it must be. The self-dual \svoa{} $V$ is the extension of $V^{\bar0}$ corresponding to the two simple currents (again placed in square brackets). For $c=\sfrac{1}{2}\pmod{8}$ we obtain the representation category of the simple Virasoro \voa{} $L(\sfrac{1}{2},0)\cong F^{\bar0}$ (or critical Ising model). For other $c\in\sfrac{1}{2}+\Z$ we obtain similar modular tensor categories (listed, e.g., in \cite{RSW09}) with the same fusion rules. These are symbolically given by $[0]\times[0]=[0]$, $[0]\times[\sfrac{1}{2}]=[\sfrac{1}{2}]$, $[0]\times(\sfrac{c}{8})=(\sfrac{c}{8})$, $[\sfrac{1}{2}]\times[\sfrac{1}{2}]=[0]$, $[\sfrac{1}{2}]\times(\sfrac{c}{8})=(\sfrac{c}{8})$ and $(\sfrac{c}{8})\times(\sfrac{c}{8})=[0]+[\sfrac{1}{2}]$, where we name the irreducible modules by their $L_0$-weights modulo $1$.


\subsection{Free-Fermion Splitting}\label{sec:holsplit}

The classification of nice, self-dual \svoa{}s of some central charge $c\in\frac{1}{2}\Ns$ also yields the classification of all of nice, self-dual \svoa{}s of central charge at most $c$. Indeed, the splitting result in \autoref{sec:split} implies:
\begin{prop}\label{prop:holsplit}
For every $c\in\frac{1}{2}\N$ and $l\in\N$ the map $V\mapsto V\otimes F^l$ defines a bijection between the isomorphism classes of nice, self-dual \svoa{}s $V$ of central charge $c$ with $\dim(V_{1/2})=k$ and those of central charge $c+\sfrac{l}{2}$ with $\dim(V_{1/2})=k+l$.
\end{prop}
Strictly speaking, for $k=0$ and $8\,|\,c$ we have to include (purely even) \voa{}s in the domain of the bijection.

In other words, a nice, self-dual vertex operator (super)algebra $V$ with $V_{1/2}=\{0\}$ (of some central charge $c\in\frac{1}{2}\N$) defines an infinite family of vertex operator superalgebras
\begin{equation*}
\{V\otimes F^l\,|\,l\in\N\}
\end{equation*}
of central charges $c+\sfrac{l}{2}$ and with $\dim((V\otimes F^l)_{1/2})=l$ that grows from the common stump $V$, and each nice, self-dual \svoa{} lives in exactly one such family.


\subsection{Neighbourhood Graph Method}\label{sec:graphmeth}

Specialising the results of \autoref{sec:holcat} to central charge a multiple of $8$, we arrive at the first main classification method used in this text, based on the concept of $2$-neighbourhood ($\Z_2$-orbifold construction). This generalises Kneser $2$-neighbours for lattices, which can be used to classify the positive-definite, odd, unimodular lattices of rank (at most) $24$ \cite{Bor85} (see also Chapter~17 in \cite{CS99}).

We briefly review the neighbourhood graph method for lattices. Recall that two positive-definite, even, unimodular lattices $\{M_1,M_2\}$ (in the same ambient vector space $M_1\otimes_\Z\R=M_2\otimes_\Z\R$, necessarily of dimension or rank a multiple of $8$) are called \emph{(Kneser) $2$-neighbours} if their intersection $M_1\cap M_2$ has index~$2$ in one (and hence both) of them \cite{Kne57} (see also, e.g., \cite{Bor85,CS99,CL19}). Then $K\coloneqq M_1\cap M_2$ is even and has discriminant form $K'/K\cong 2_{\II}^{+2}$.

The structure of $K'/K$ entails that $K$ has exactly three non-trivial extensions to integral lattices, namely the two even, unimodular lattices $M_1$ and $M_2$, corresponding to the two elements $\gamma_1,\gamma_2\in K'/K$ of order~$2$ and norm $0\pmod{1}$, and further an odd, unimodular lattice $L$ with $L_\text{ev}=K$, corresponding to the element $\gamma_3\in K'/K$ of order~$2$ and norm $\sfrac{1}{2}\pmod{1}$.

In particular, it follows that the $2$-neighbours $\{M_1,M_2\}$ are uniquely determined by the positive-definite, even lattice $K$ with $K'/K\cong 2_{\II}^{+2}$ (and vice versa), and every lattice $K$ with these properties, which can only exist if $8\,|\,\rk(K)$, specifies a pair of $2$-neighbours.

The $2$-neighbours $\{M_1,M_2\}$ are also uniquely determined by the positive-definite, odd, unimodular lattice $L$ (and vice versa), and every such lattice, now assuming that $8\,|\,\rk(L)$, defines a pair of $2$-neighbours. Indeed, the even sublattice $K\coloneqq L_\text{ev}$ of $L$ must have discriminant form $K'/K\cong 2_{\II}^{+2}$ and has exactly two non-trivial extensions to even lattices $M_1$ and $M_2$.

The situation is depicted in \autoref{fig:neighbourhood1}.
\begin{figure}[ht]
\caption{The notion of $2$-neighbourhood for positive-definite, even, unimodular lattices.}
~\\
\begin{tikzcd}[ampersand replacement=\&]
M_1\arrow[-,dashed]{r}\& L\arrow[-,dashed]{r}\& M_2\&\text{unimodular}\\
\\
\& K=L_\text{ev}\arrow[hookrightarrow]{uul}{\text{ind.\ 2}}\arrow[hookrightarrow]{uu}[pos=0.7]{\text{ind.\ 2}}\arrow[hookrightarrow]{uur}[swap]{\text{ind.\ 2}} \&\& {|K'/K|=4}
\end{tikzcd}
\label{fig:neighbourhood1}
\end{figure}
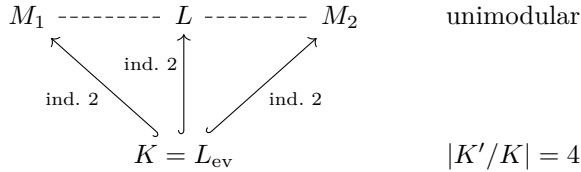

In the end, we only want to consider lattices up to isomorphism. This leads to the notion of the \emph{2-neighbourhood graph}. Fix the rank $d\in8\N$. The finitely many isomorphism classes of positive-definite, even, unimodular lattices are the \emph{nodes} of this graph. An \emph{edge} is drawn between two (possibly identical) nodes, i.e.\ two isomorphism classes of positive-definite, even, unimodular lattices $M_1$ and $M_2$, if there is a lattice $K$ with $K'/K\cong 2_{\II}^{+2}$ such that the two non-trivial even extensions of $K$ are isomorphic to $M_1$ and $M_2$.

Note that the lattice $K$ is not necessarily uniquely determined up to isomorphism by $M_1$ and $M_2$, meaning that nodes may be joined by multiple (but finitely many) edges, each labelled by a different isomorphism class of lattice $K$.

\begin{prop}
Let $d\in8\N$. There is a bijection between the isomorphism classes of positive-definite, odd, unimodular lattices $L$ of rank $d$ and the edges in the 2-neighbourhood graph in that rank, labelled by the isomorphism classes of positive-definite, even lattices $K$ with $K'/K\cong 2_{\II}^{+2}$.
\end{prop}
The edges $K$ incident with a node $M$ correspond to the $\O(M)$-orbits of vectors~$v$ of norm $\langle v,v\rangle/2=0\pmod{2}$ in $M/(2M)\setminus\{0\}$, with the lattice $K$ given by $K=\{x\in M\,|\,\langle v,x\rangle=0\pmod{2}\}$. The edges are of three different types: the two ends $M=M_1$ and $M_2$ can be (1) non-isomorphic, or the two lattices are isomorphic but the ends belong either to (2)~different or to (3)~the same $\O(M)$-orbit.

This can be used to classify the positive-definite, odd, unimodular lattices of rank $24$ \cite{Bor85,CS99}. The neighbourhood graph for $d=8$, $16$ and $24$ is shown in Figure~17.1 of \cite{CS99}. There, only the edges are drawn for which the positive-definite, odd, unimodular lattice $L$ does not contain vectors of norm~$\sfrac{1}{2}$, i.e.\ does not split off copies of the standard lattice $\Z$ (see last paragraph of \autoref{sec:latsvoa}). All of the not drawn edges must be loops (cf.\ \autoref{prop:suffloop} below). For $d\leq 24$, there are only edges of type (1) and (3), but it is expected that edges of type (2) exist for large enough $d$.\footnote{Richards Borcherds, personal communication.}

\medskip

We now describe the analogous result for nice, self-dual \svoa{}s $V$. To this end we generalise the notion of $2$-neighbourhood to nice, self-dual \voa{}s, which is nothing but the cyclic orbifold construction of order~$2$ \cite{FLM88,EMS20a} (cf.\ \cite{GJ22}).

Two nice, self-dual \voa{}s $\{W^{(1)},W^{(2)}\}$, of some central charge $c\in8\N$, are called \emph{2-neighbours} if they are different (but possibly isomorphic as \voa{}s) $\Z_2$-extensions (i.e.\ simple-current extensions) of the same nice \voa{} $U$. This implies that $\Rep(U)$ is the modular tensor category $\mathcal{C}(2_{\II}^{+2})$ associated with the discriminant form $2_{\II}^{+2}$.

As we just saw, the latter has three non-zero elements $\gamma_1$, $\gamma_2$ and $\gamma_3$, all of order $2$ with norm $0$, $0$ and $\sfrac{1}{2}\pmod{1}$, respectively. We denote the irreducible $U$-modules by $W^\gamma$, $\gamma\in 2_{\II}^{+2}$. Then $U=W^0$. The simple-current extensions of order~$2$ corresponding to the elements $\gamma_1$, $\gamma_2$ and $\gamma_3$ form the two self-dual \voa{}s $W^{(1)}=W^0\oplus W^{\gamma_1}$ and $W^{(2)}=W^0\oplus W^{\gamma_2}$ (or vice versa) and the self-dual \svoa{} $V=W^0\oplus W^{\gamma_3}$ with $V^{\bar0}=U$, respectively.

Again, the $2$-neighbours $\{W^{(1)},W^{(2)}\}$ are uniquely determined by the nice \voa{} $U$ with $\Rep(U)\cong\mathcal{C}(2_{\II}^{+2})$, and every such \voa{}, necessarily of central charge $c\in8\N$, specifies a pair of $2$-neighbours.

Similarly, the $2$-neighbours $\{W^{(1)},W^{(2)}\}$ are also determined by the nice, self-dual \svoa{} $V$, and every such \svoa{}, now assumed to be of central charge $c\in8\N$, defines a pair of $2$-neighbours. Indeed, the representation category of the even part $U\coloneqq V^{\bar0}$ of a nice \svoa{} $V$ was described in \autoref{sec:holcat} and depends on the central charge $c$ modulo $8$. The case of $8\,|\,c$ is quite special. Here, $\Rep(U)$ is the modular tensor category $\mathcal{C}(2_{\II}^{+2})$ associated with the discriminant form $2_{\II}^{+2}$ and we just described its non-trivial simple-current extensions.

The situation is depicted in \autoref{fig:neighbourhood2} (compare with \autoref{fig:neighbourhood1}).
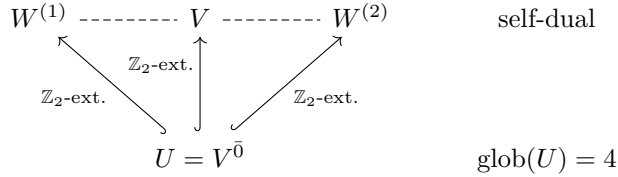
\begin{figure}[ht]
\caption{The notion of $2$-neighbourhood for nice, self-dual \voa{}s.}
~\\
\begin{tikzcd}[ampersand replacement=\&]
W^{(1)}\arrow[-,dashed]{r}\& V\arrow[-,dashed]{r}\& W^{(2)}\&\text{self-dual}\\
\\
\& U=V^{\bar0}\arrow[hookrightarrow]{uul}{\Z_2\text{-ext.}}\arrow[hookrightarrow]{uu}[pos=0.7]{\Z_2\text{-ext.}}\arrow[hookrightarrow]{uur}[swap]{\Z_2\text{-ext.}} \&\& {\glob(U)=4}
\end{tikzcd}
\label{fig:neighbourhood2}
\end{figure}

We now define the \emph{2-neighbourhood graph} in central charge $c\in8\N$. The \emph{nodes} are the isomorphism classes of nice, self-dual \voa{}s of central charge $c$. Conjecturally, the number of nodes is finite, but this is only known for $c=0$, $8$, $16$ and almost known for $c=24$, where the uniqueness of the moonshine module is still open.

There is an \emph{edge} between two (possibly identical) nodes $W^{(1)}$ and $W^{(2)}$ if there is a nice \voa{} $U$ with $\Rep(U)\cong\mathcal{C}(2_{\II}^{+2})$ whose two non-trivial simple-current extensions to \voa{}s are isomorphic to $W^{(1)}$ and $W^{(2)}$, and we label that edge by the isomorphism class of $U$. Conjecturally, the number of edges between two nodes is finite, and this is known to be true for $c=0$, $8$, $16$ and $24$ since the automorphism groups of the self-dual \voa{}s corresponding to the nodes are sufficiently understood \cite{HM22,BLS23}. For the possibly existent fake copies $V$ of $V^\natural$, it follows from character calculations (see \autoref{rem:fakeshort}) that the only involutions are Miyamoto involutions for Ising vectors inside $V_2$ corresponding to loops (see below). As shown by Conway~\cite{Con85}, there can be only finitely many of these.

We obtain the following correspondence:
\begin{prop}[Neighbourhood Graph Method]\label{prop:neighbourhood}
Let $c\in8\N$. There is a bijection between the isomorphism classes of nice, self-dual \svoa{}s $V$ of central charge $c$ and the edges in the 2-neighbourhood graph in that central charge, labelled by the isomorphism classes of nice \voa{}s with representation category $\mathcal{C}(2_{\II}^{+2})$.
\end{prop}

We shall see that many edges in the $2$-neighbourhood graph are loops, i.e.\ that the two nodes $W^{(1)}\cong W^{(2)}$ are isomorphic. There is a sufficient (but in general not necessary) criterion for loops \cite{Hoe95}:
\begin{prop}\label{prop:suffloop}
In the situation of \autoref{prop:neighbourhood}, if the self-dual \svoa{} $V$ satisfies $\dim(V_{1/2})>0$, then it corresponds to a loop.
\end{prop}
\begin{proof}
Let $c\in8\Ns$ denote the central charge of $V$. If $\dim(V_{1/2})>0$, then $V$ splits off a non-trivial tensor power of the Clifford \svoa{}~$F$ (see \autoref{prop:split}). In particular, $V\cong V'\otimes F$ where $V'$ is a nice, self-dual \svoa{} of central charge $c-\sfrac{1}{2}$. The two $2$-neighbour \voa{}s $W^{(1)}$ and $W^{(2)}$ must both be extensions of $V^{\bar0}\cong(V'\otimes F)^{\bar0}$ and hence also extensions of $(V')^{\bar0}\otimes F^{\bar0}\cong (V')^{\bar0}\otimes L(\sfrac{1}{2},0)$.

The representation categories of $(V')^{\bar0}$ and $L(\sfrac{1}{2},0)$ are described in \autoref{sec:holcat}, and they are braid-reversed equivalent. Let $M(0)=(V')^{\bar0}$, $M(\sfrac{1}{2})$, $M(\sfrac{15}{16})$ and $L(\sfrac{1}{2},0)$, $L(\sfrac{1}{2},\sfrac{1}{2})$, $L(\sfrac{1}{2},\sfrac{1}{16})$ denote the irreducible $(V')^{\bar0}$- and $L(\sfrac{1}{2},0)$-modules, respectively, which are labelled by their lowest $L_0$-weights (modulo~$1$ for $V'$). By the theory of mirror extensions \cite{CKM22,Lin17}, both $W^{(1)}$ and $W^{(2)}$, since they are self-dual, are \voa{} extensions of the form
$
W^{(i)}\cong M(0)\otimes L(\sfrac{1}{2},0)$
$\oplus\,M(\sfrac{1}{2})\otimes L(\sfrac{1}{2},\sfrac{1}{2})\,\oplus\,M(\sfrac{15}{16})\otimes L(\sfrac{1}{2},\sfrac{1}{16})
$
and are hence isomorphic.
\end{proof}

\medskip

Finally, we reformulate \autoref{prop:neighbourhood} in the language of orbifold theory. Two nice, self-dual \voa{}s $W^{(1)}$ and $W^{(2)}$ being $2$-neighbours is equivalent to $W^{(2)}$ being an orbifold construction of order~$2$ of $W^{(1)}$ \cite{FLM88,EMS20a} (or equivalently vice versa). Indeed, setting $g|_{W^0}=\id$ and $g|_{W^{\gamma_1}}=-\id$ (with the notation from above) defines an automorphism $g$ of $W^{(1)}$ of order~$2$ and type~$0$ satisfying the positivity condition (see \cite{EMS20a} for a definition of these terms). The \fpvosa{} is $U=W^0$ and the corresponding cyclic orbifold construction is $(W^{(1)})^{\orb(g)}=W^{(2)}$. The inverse-orbifold automorphism on $W^{(2)}$ is defined by $g|_{W^0}=\id$ and $g|_{W^{\gamma_2}}=-\id$.

This implies the following correspondence, with a slight complication arising because, in the case of a loop, an automorphism may or may not be conjugate to its own inverse-orbifold automorphism.
\begin{prop}\label{prop:order2conj}
Let $c\in8\N$. There is a surjective map from the conjugacy classes $g$ of automorphisms of order~2 and type~0 of the isomorphism classes $W^{(1)}$ of nice, self-dual \voa{}s of central charge~$c$ to the isomorphism classes of nice, self-dual \svoa{}s of that central charge. It maps $g\in\Aut(W^{(1)})$ to the unique self-dual \svoa{} extending $(W^{(1)})^g$. This map is 2-to-1 if
\begin{enumerate}
\item $W^{(1)}\ncong W^{(2)}$ or
\item $W^{(1)}\cong W^{(2)}$ and $g$ is not conjugate to its inverse-orbifold automorphism
\end{enumerate}
and injective if
\begin{enumerate}
\item[(3)] $W^{(1)}\cong W^{(2)}$ and $g$ is conjugate to its inverse-orbifold automorphism
\end{enumerate}
where $W^{(2)}\coloneqq(W^{(1)})^{\orb(g)}$.
\end{prop}

We remark that in cases (1) and (3) the isomorphism class of the \fpvosa{} $(W^{(1)})^g$ determines a unique conjugacy class in $\Aut(W^{(1)})$, while in case~(2) it determines two distinct ones.

We shall observe that for central charge $c=24$, there is only one edge of type~(2) (cf.\ the corresponding lattice statement above). One may view this as a consequence of the fact that, similar to Lie algebras, for automorphisms of small order the conjugacy class is often determined by the fixed-point subalgebra.

We also comment on \autoref{prop:suffloop}. If $V$ splits off at least one free fermion $F$, then the automorphism of order~$2$ on $W^{(1)}$ and the inverse-orbifold automorphism on $W^{(1)}\cong W^{(2)}$ are Miyamoto involutions \cite{Miy96} associated with the conformal vector of $F^{\bar0}\cong L(\sfrac{1}{2},0)$ (also called Ising vector in this context) and in fact are conjugate under the isomorphism $W^{(1)}\cong W^{(2)}$, meaning that case~(3) holds.

\medskip

To conclude, we have shown that in order to classify the nice, self-dual \svoa{}s of some central charge $c\in8\Ns$, and hence by \autoref{prop:holsplit} those of central charge less than or equal to $c$, it suffices to enumerate the edges in the $2$-neighbourhood graph of the nice, self-dual \voa{}s of central charge $c$ by studying all possible $\Z_2$-orbifold constructions.

This will be one of the two strategies used in this text to obtain a classification up to central charge $24$.


\subsection{Results in Small Central Charges}

As motivating examples, in this section we give the $2$-neighbourhood graphs up to central charge $16$. The corresponding classification of the nice, self-dual \svoa{}s (obtained independently in \cite{BKLTZ24}) is already a new result, with the classification up to central charge $12$ given in \cite{CDR18} and conjectured up to central charge $15\shs\sfrac{1}{2}$ in \cite{Hoe95} (but see \cite{KLT86} in the physics literature).

\medskip

By mapping a positive-definite, even, unimodular lattice $M$ to the corresponding lattice \voa{} $W=V_M$, the nodes of the lattice $2$-neighbourhood graph in rank $d\in8\N$ \cite{Bor85,CS99} form a subset of the nodes of the \voa{} $2$-neighbourhood graph of central charge $c=d$. Similarly, we map each edge labelled by a lattice $K$ with $K'/K\cong 2_{\II}^{+2}$ to the edge (of glueing type~I, see \autoref{sec:glueing}) labelled by the corresponding lattice \voa{} $U=V_K$. In total, this means that the lattice $2$-neighbourhood graph is a subgraph of the \voa{} $2$-neighbourhood graph.

We note that this subgraph need not be an induced subgraph, i.e.\ two lattice \voa{}s may have edges between them that do not come from an edge in the lattice graph. This happens exactly when the \voa{} $U$ labelling that edge is not isomorphic to a lattice \voa{} (in which case the edge has glueing type~III).


\begin{figure}[ht]
\caption{The $2$-neighbourhood graphs for the nice, self-dual \voa{}s of central charges $c=0$, $8$ and $16$.}
~\\
\begin{tikzpicture}[thick, main/.style = {draw, circle, minimum size=10mm, inner sep=0.05, fill=black!15}]
\node at (0.8,2.5) {$c=0$};
\node[main] (1) at (0.8,0) {$V_{\{0\}}$};
\node at (3,2.5) {$c=8$};
\node[main] (2) at (3,0) {$V_{E_8}$};
\draw (2) to [out=1*360/6,in=2*360/6,looseness=8] node[above] {$V_{D_8}$} (2);
\node at (8,2.5) {$c=16$};
\node[main] (3) at (6.8,0) {$V_{E_8^2}$};
\node[main] (4) at (9.2,0) {$V_{D_{16}^+}$};
\draw (3) to node[above] {$V_{(D_8^2)^+}$} (4);
\draw (3) to [out=1*360/7,in=2*360/7,looseness=8] node[above] {$V_{(A_1^2E_7^2)^+}\quad$} (3);
\draw (3) to [out=3*360/7,in=4*360/7,looseness=8] node[left] {$V_{D_8E_8}$} (3);
\draw [dashed] (3) to [out=5*360/7,in=6*360/7,looseness=8] node[below] {$(V_{E_8}^{\otimes2})^{S_2}\quad$} (3);
\draw (4) to [out=1*360/7+180,in=2*360/7+180,looseness=8] node[below] {$\quad V_{(A_1(2)A_{15})^{++}}$} (4);
\draw (4) to [out=3*360/7+180,in=4*360/7+180,looseness=8] node[right] {$V_{D_{16}}$} (4);
\draw (4) to [out=5*360/7+180,in=6*360/7+180,looseness=8] node[above] {$\quad V_{(D_4D_{12})^+}$} (4);
\end{tikzpicture}
\label{fig:small}
\end{figure}
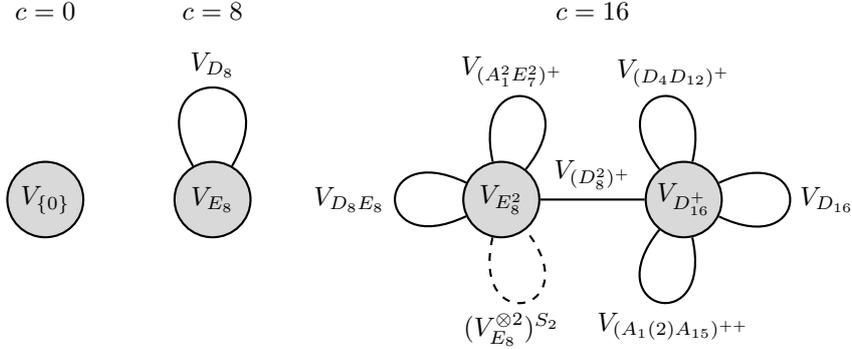

In \autoref{fig:small}, we present the \voa{} $2$-neighbourhood graphs of central charges $c=0$, $8$ and $16$. We do not give any details of the computations as these will be discussed in detail in the case of central charge $c=24$, and moreover the classification result is an immediate consequence of the one for $c=24$ (see \autoref{sec:results}) by splitting off free fermions. The graphs for $c=0$ and $8$ are identical to the lattice graphs for ranks $d=0$ and $8$, i.e.\ they have the same nodes and edges. (But recall that certain loops were omitted from the lattice graphs depicted in \cite{CS99}.) For $c=16$, all nodes are lattice nodes, but there is one extra edge compared to the lattice graph (see also \autoref{table:numbers}).

We explain the notation. The lattices are either root lattices, like $E_8$ or $D_8E_8$, or even, small-index extensions of root lattices, like $D_{16}^+$ or $(D_4D_{12})^+$. The plus sign denotes an index-$2$ extension (except for $A_{15}^+$, see below). This notation is in general not well-defined, but all lattices here are uniquely determined from context. The lattice $A_1(2)$ denotes the root lattice $A_1$ with the bilinear form scaled by $2$. The \voa{} for the trivial lattice $V_{\{0\}}$ is nothing but the finite-dimensional \voa{} $\C\vac$ spanned by the vacuum vector.

The only non-lattice edge (drawn as a dashed line) is labelled by the permutation orbifold $(V_{E_8}^{\otimes2})^{S_2}$, i.e.\ the \fpvosa{} of $V_{E_8^2}\cong V_{E_8}^{\otimes2}$ under the automorphism that permutes the two tensor factors.

\begin{table}[ht]
\caption{The nice, self-dual \svoa{}s of central charge $16$.}
\renewcommand{\arraystretch}{1.2}
\begin{tabular}{llllllr}
$V$ & $V_1$ & $V^{\bar0}$ & $W^{(1)}$ & $W^{(2)}$ & Type & $c_\text{st}~$\\\hline\hline
$F^{\otimes32}$&$D_{16}$&$V_{D_{16}}$&$V_{D_{16}^+}$&$V_{D_{16}^+}$&I& $0\ph$\\
$V_{D_{12}^+}\otimes F^{\otimes8}$&$D_4D_{12}$&$V_{(D_4D_{12})^+}$&$V_{D_{16}^+}$&$V_{D_{16}^+}$&I&$12\ph$\\
$V_{A_{15}^+}\otimes F^{\otimes2}$&$A_{15}\C$&$V_{(A_1(2)A_{15})^{++}}$&$V_{D_{16}^+}$&$V_{D_{16}^+}$&I&$15\ph$\\[2pt]\hline
$V_{(D_8^2)^{++}}$&$D_8^2$&$V_{(D_8^2)^+}$&$V_{D_{16}^+}$&$V_{E_{8}^2}$&I&$16\ph$\\[2pt]\hline
$V_{E_8}\otimes F^{\otimes16}$&$D_8E_8$&$V_{D_8E_8}$&$V_{E_{8}^2}$&$V_{E_{8}^2}$&I&$8\ph$\\
$V_{(E_7^2)^+}\otimes F^{\otimes4}$&$A_1^2E_7^2$&$V_{(A_1^2E_7^2)^+}$&$V_{E_{8}^2}$&$V_{E_{8}^2}$&I & $14\ph$\\
$\bar{V}\otimes F$&$E_{8,2}$&$(V_{E_8}^{\otimes2})^{S_2}$&$V_{E_{8}^2}$&$V_{E_{8}^2}$&III&$15\shs\sfrac{1}{2}$
\end{tabular}
\label{table:small}
\end{table}

The corresponding classification of the nice, self-dual \svoa{}s of central charge $c=16$ is given in \autoref{table:small}. We list the nice, self-dual \svoa{} $V$ (written as the stump times a certain tensor power of $F$), the weight-$1$ Lie algebra $V_1$ (together with the affine levels, see \autoref{sec:cartan} for the notation), its even part $V^{\bar0}$, the two nice, self-dual \voa{}s $W^{(1)}$ and $W^{(2)}$ extending $V^{\bar0}$ and the central charge $c_\text{st}$ of the stump of $V$. The meaning of the glueing type (here I or III) will be explained in \autoref{sec:glueing}.

We remark that the names for the positive-definite, odd, unimodular lattices $(E_7^2)^+$, $D_{12}^+$ and $A_{15}^+$ follow standard convention. The same is true for the positive-definite, even, unimodular lattice $D_{16}^+$. In the last row, $\bar{V}$ denotes the unique nice, self-dual \svoa{} of central charge $c=15\shs\sfrac{1}{2}$ without vectors of $L_0$-weight $\sfrac{1}{2}$. Its even part $\bar{V}^{\bar0}\cong L_{E_8}(2,0)$ is the simple affine \voa{} for $E_8$ at level $2$.

\begin{prop}[Classification, Preliminary]
Let $V$ be a nice, self-dual \svoa{} of central charge $c=8$ or $16$. Then $V$ is isomorphic to one of the following:
\begin{enumerate}
\item[$c=8$:] $F^{\otimes16}$,
\item[$c=16$:] the seven non-isomorphic \svoa{}s listed in \autoref{table:small}.
\end{enumerate}
\end{prop}
The main result of this paper is the analogous statement for $c=24$, which is given in \autoref{thm:class} and \autoref{table:969}.


\subsection{Characters}\label{sec:chars}

In the following, we describe the characters of self-dual \svoa{}s (cf.\ \cite{Hoe95}), focusing in particular on the case of central charge $c=24$.

Given a nice, self-dual \svoa{} $V$ of central charge $c$, the vector-valued character of its even part $V^{\bar0}$ is
\begin{equation*}
\Ch_{V^{\bar0}}(q)=\!\!\sum_{W\in\Irr(V^{\bar0})}\!\!\ch_W(q)\ee_W\quad\text{with}\quad\ch_W(q)=\tr_Wq^{L_0-c/24}
\end{equation*}
where $\Irr(V^{\bar0})$ is the set of inequivalent irreducible $V^{\bar0}$-modules. By \cite{Zhu96}, setting $q=\e^{2\pi\i\tau}$ with $\tau\in\H=\{z\in\C\,|\,\Im(z)>0\}$, this character converges to a weakly holomorphic (i.e.\ holomorphic on $\H$ but with possible poles at the cusp $\infty$), vector-valued modular form for $\SLZ$ of weight~$0$. The central charge $c$ determines a bound on the pole orders of $\Ch_{V^{\bar0}}(\tau)$ at the cusp $\infty$ so that the corresponding space of modular forms is finite-dimensional. In fact, since $c>0$ and $\vac\in V^{\bar0}$ has $L_0$-weight~$0$, $\Ch_{V^{\bar0}}(\tau)$ must actually have a pole at $\infty$.

If $c\in\Z$, we saw in \autoref{sec:holcat} that the representation category $\Rep(V^{\bar0})$ is pointed, i.e.\ of the form $\mathcal{C}(D)$ for some discriminant form $D$ depending on the central charge $c$ modulo $8$ (see \autoref{table:16}), which is also the signature of $D$. In this situation, $\Ch_{V^{\bar0}}(\tau)$ is a weight-$0$ vector-valued modular form transforming under a representation of $\SLZ$ that is up to a character the well-known Weil representation of $\SLZ$ (or its metaplectic double cover) on $\C[D]$ \cite{Moe16,EMS20a}. This representation depends on the residue class of the central charge $c$ modulo $24$.

If $24\,|\,c$, this character is trivial so that $\Ch_{V^{\bar0}}(\tau)$ transforms with weight~$0$ under the Weil representation $\rho_D\colon\SLZ\to\C[D]$ for the discriminant form $D=2_{\II}^{+2}$ of signature $0\pmod{8}$. If we label the irreducible modules as $\Irr(V^{\bar0})=\{W^\gamma\,|\,\gamma\in D\}$ so that we may write
\begin{equation*}
\Ch_{V^{\bar0}}(\tau)=\sum_{\gamma\in D}\ch_\gamma(\tau)\ee_\gamma,
\end{equation*}
with $\ch_\gamma\coloneqq\ch_{W^\gamma}$, then $\Ch_{V^{\bar0}}$ transforms explicitly as
\begin{equation*}
\Ch_{V^{\bar0}}(M.\tau)=\rho_D(M)\Ch_{V^{\bar0}}(\tau)
\end{equation*}
for all $M\in\SLZ$ with $\rho_D$ defined via
\begin{align*}
\rho_D(T)_{\beta,\gamma}&=\delta_{\beta,\gamma}\e^{2\pi\i q(\beta)}=\diag(1,1,1,-1),\\
\rho_D(S)_{\beta,\gamma}&=\frac{1}{\sqrt{|D|}}\e^{-2\pi\i b(\beta,\gamma)}=\frac{1}{2}\,\begin{psmallmatrix*}[r]1&1&1&1\\1&1&-1&-1\\1&-1&1&-1\\1&-1&-1&1\end{psmallmatrix*}
\end{align*}
for $\beta$, $\gamma\in D$ where $q$ is the quadratic form of $D$ and $b$ the associated bilinear form. Here, $S=\bigl(\begin{smallmatrix}0&-1\\1&\,\ 0\end{smallmatrix}\bigr)$ and $T=\bigl(\begin{smallmatrix}1&1\\0&1\end{smallmatrix}\bigr)$ are the standard generators of $\SLZ$, and we ordered the elements of $D\cong\Z_2\times\Z_2$ such that the quadratic form (i.e.\ the $L_0$-weights of the corresponding irreducible modules) have values $0$, $0$, $0$, $\sfrac{1}{2}$ modulo~$1$ and the first element is the trivial one.

\medskip

Specialising to $c=24$, since $V^{\bar0}$ is of CFT-type and satisfies the positivity condition, the smallest $q$-powers in the components of $\Ch_{V^{\bar0}}(\tau)$ are $(q^{-1},q^0,q^0,q^{-1/2})$. The space of weakly holomorphic modular forms of weight~$0$ for $\rho_D$ with these pole-order bounds is $4$-dimensional, and we shall state a basis in the following.

In general, a (meromorphic) modular form $f$ for the congruence subgroup $\Gamma_0(2) = \{(\begin{smallmatrix}a&b\\c&d\end{smallmatrix})\in\SLZ\,|\,c=0\pmod{2}\}$ of some weight~$k$ can be lifted to a vector-valued modular form for $\rho_D$, $D=2_{\II}^{+2}$, of the same weight (see \cite{Sch09}, where also an explicit formula for this lift is given). Decomposing
\begin{equation*}
f|_S=f_++f_-
\end{equation*}
where $f_\pm|_T=\pm f_\pm$ (with the Petersson slash operator), the lift $F$ of $f$ is
\begin{equation*}
F=(f+f_+,f_+,f_+,f_-)^\mathrm{T}.
\end{equation*}
Now, consider the standard Hauptmodul
\begin{equation*}
t_2(\tau)=\left(\frac{\eta(\tau)}{\eta(2\tau)}\right)^{24}=q^{-1}-24+276\,q-2048\,q^2+\dots
\end{equation*}
for $\Gamma_0(2)$, generating the field of meromorphic weight-$0$ modular forms for $\Gamma_0(2)$, where $\eta$ is the Dedekind eta function. Then, one can show that a basis of the above $4$-dimensional space of vector-valued modular forms is given by the lift of $t_2$, that of $1/t_2$ and the constant functions $(1,1,0,0)^\mathrm{T}$ and $(1,0,1,0)^\mathrm{T}$.

In fact, as the leading term of $\ch_{V^{\bar0}}=\ch_0$ (corresponding to $0\in D$) is $q^{-1}$, the vector-valued character lies in a $3$-dimensional affine subspace. We describe this space in more detail. Let $f\coloneqq t_2$ and $g\coloneqq 2^{12}/t_2$. Then the corresponding lifts are
\begin{equation*}
F=\begin{pmatrix}f+f_+\\f_+\\f_+\\f_-\end{pmatrix}=\begin{pmatrix}j-f_+-48\\f_+\\f_+\\f_-\end{pmatrix}=\begin{pmatrix}q^{-1}-24+98580\,q+O(q^2)\\98304\,q+O(q^2)\\98304\,q+O(q^2)\\4096\,q^{1/2}+O(q^{3/2})\end{pmatrix}
\end{equation*}
with the modular $j$-function $j(\tau)=q^{-1}+24+196884\,q+\dots$ (take note of the non-standard constant term) and
\begin{equation*}
G=\begin{pmatrix}g+g_+\\g_+\\g_+\\g_-\end{pmatrix}=\begin{pmatrix}-g_+-48\\g_+\\g_+\\g_-\end{pmatrix}=\begin{pmatrix}-24+2048\,q+O(q^2)\\-24-2048\,q+O(q^2)\\-24-2048\,q+O(q^2)\\q^{-1/2}+276\,q^{1/2}+O(q^{3/2})\end{pmatrix}.
\end{equation*}
We can write the components as eta products
\begin{align*}
f(\tau)&=\left(\frac{\eta(\tau)}{\eta(2\tau)}\right)^{24}\!\!,\quad f(S.\tau)=2^{12}\left(\frac{\eta(\tau)}{\eta(\tau/2)}\right)^{24}\!\!,\\
f_\pm(\tau)&=2^{11}\left(\left(\frac{\eta(\tau)}{\eta(\tau/2)}\right)^{24}\!\!\mp\left(\frac{\eta(\tau/2)\eta(2\tau)}{\eta(\tau)^2}\right)^{24}\right)
\end{align*}
and
\begin{align*}
g(\tau)&=2^{12}\left(\frac{\eta(2\tau)}{\eta(\tau)}\right)^{24}\!\!,\quad g(S.\tau)=\left(\frac{\eta(\tau/2)}{\eta(\tau)}\right)^{24}\!\!,\\
g_\pm(\tau)&=\frac{1}{2}\left(\left(\frac{\eta(\tau/2)}{\eta(\tau)}\right)^{24}\!\!\mp\left(\frac{\eta(\tau)^2}{\eta(\tau/2)\eta(2\tau)}\right)^{24}\right).
\end{align*}
The vector-valued character of $V^{\bar0}$ must then be of the form
\begin{equation*}
\Ch_{V^{\bar0}}=F+a\cdot(1,1,0,0)^\mathrm{T}+b\cdot(1,0,1,0)^\mathrm{T}+l\cdot G
\end{equation*}
with coefficients $a$, $b$ and $l$ determined by the low-order terms of $\Ch_{V^{\bar0}}$ via
\begin{align*}
a&=24+2\dim(V_1)-\dim(W^{(1)}_1),\\
b&=24+2\dim(V_1)-\dim(W^{(2)}_1),\\
l&=\dim(V_{1/2})=(24+3\dim(V_1)-\dim(W^{(1)}_1)-\dim(W^{(2)}_1))/24
\end{align*}
where $W^{(1)}$ and $W^{(2)}$ are the two \voa{} neighbours described in \autoref{sec:graphmeth}. The non-trivial identity in the last row is precisely the statement of the dimension formula in \cite{Mon98}, a special case of the much more general formula in \cite{MS23}. Using the explicit expressions for the components of the chosen basis, we see that the fact that all Fourier coefficients of the character must lie in $\N$ is equivalent to
\begin{equation*}
a,\,b,\,l\in\Z
\end{equation*}
and
\begin{align*}
0\leq l\leq 48,\quad a,\,b\geq 24l\quad\text{and}\quad a+b\geq24(l+1).
\end{align*}
The bound on $l$ also follows from its meaning as the number of free fermions that $V$ splits, which is at most $2c=48$ many (see \autoref{prop:holsplit}).


\section{Dual Pairs in \VOA{}s}\label{sec:voas}

In this section we describe results on the internal structure of nice vertex operator (super)algebras. In particular, we introduce the notion of the associated lattice.

Later, two special cases shall be of particular interest to us, namely self-dual \voa{}s $V$, with $\Rep(V)\cong\Vect$, and \voa{}s with $\Rep(V)\cong\mathcal{C}(2_{\II}^{+2})$. They will be studied in \autoref{sec:holvoa} and \autoref{sec:evensub}, respectively. The associated lattice is also central in the description of the \voa{} automorphism groups in \autoref{sec:orbifold}.


\subsection{Cartan Subalgebras in \VOA{}s}\label{sec:cartan}

We describe a theory of Cartan subalgebras in \strat{} \voa{}s $V$ and associate an (in a certain sense) maximal positive-definite, even lattice $L$ with each such \voa{} (cf.\ \cite{Hoe17,Mor21}). This is an application of the main result of \cite{Mas14} (with preliminary results in \cite{Li01,DM06b}). As we remark below, it is straightforward to extend this result to \svoa{}s, in which case the associated lattice may be even or odd (cf.\ \cite{DH12}). Then we explain how this can be used to write the \voa{} $V$ as a simple-current extension of a certain dual pair in $V$ (see, e.g., \cite{CKLR19,CGN21}).

\medskip

Given a \voa{} $V$ of CFT-type, it is well known that $[a,b]\coloneqq a_0b$ for $a,b\in V_1$ equips $V_1$ with the structure of a Lie algebra. In the following, we assume that $V$ is \strat{} and has central charge $c$. Then the Lie algebra~$V_1$ is reductive, i.e.\ a direct sum
\begin{equation*}
V_1\cong\g_1\oplus\dots\oplus\g_s\oplus\mathfrak{a}
\end{equation*}
of a semisimple part $\g=\g_1\oplus\dots\oplus\g_s$ with simple ideals $\g_i$ and an abelian part~$\mathfrak{a}$ \cite{DM04b}. Suppose that $V_1$ has rank $r\coloneqq\rk(V_1)$. We remark that if $V$ is nice, i.e.\ also satisfies the positivity condition, then $r\leq c$~\cite{DM04b}.

The \voa{} $V$ is endowed with a non-degenerate, invariant bilinear form $\langle\cdot,\cdot\rangle$, which is unique up to a non-zero scalar and symmetric \cite{FHL93}. As usual, we normalise it such that $\langle\vac,\vac\rangle=-1$. The form restricts to a non-degenerate, invariant bilinear form on the Lie algebra $V_1$ (and on each of the direct summands $\g_i$) and is given by $a_1b=b_1a=\langle a,b\rangle\vac$ for $a,b\in V_1$.

The (in general non-full) \vosa{} $\langle V_1\rangle$ of $V$ generated by $V_1$ is isomorphic to the tensor product
\begin{equation*}
\langle V_1\rangle\cong L_{\hat\g_1}(k_1,0)\otimes\dots\otimes L_{\hat\g_s}(k_s,0)\otimes M_{\hat{\mathfrak{a}}}(1,0)
\end{equation*}
of simple affine \voa{}s with positive integer levels $k_i\in\Ns$ and the Heisenberg \voa{} associated with the abelian Lie algebra~$\mathfrak{a}$ together with the restriction of $\langle\cdot,\cdot\rangle$ \cite{DM06b}. This entails that $\langle\cdot,\cdot\rangle=k_i(\cdot,\cdot)$ restricted to $\g_i$ where $(\cdot,\cdot)$ is the invariant bilinear form on $\g_i$ normalised such that the long roots have norm $2$. We introduce the notation
\begin{equation*}
V_1\cong\g_{1,k_1}\oplus\dots\oplus\g_{s,k_s}\oplus\mathfrak{a}
\end{equation*}
to denote the type of the reductive Lie algebra $V_1$ together with the levels $k_i$. This is called the \emph{affine structure} of $V$.

\medskip

Now, let $\hh$ be a Cartan subalgebra of $V_1$. Then $\hh$ is an abelian Lie algebra of dimension $\dim(\hh)=r$ and is equipped with the non-degenerate bilinear form obtained as restriction of $\langle\cdot,\cdot\rangle$. We consider the corresponding Heisenberg \voa{}
\begin{equation*}
H\coloneqq\langle\hh\rangle=M_{\hat\hh}(1,0)
\end{equation*}
with standard conformal vector $\omega^\hh$ of central charge $r$, which is in general a non-full \vosa{} of $\langle V_1\rangle$ and of $V$. Let the \emph{Heisenberg commutant}
\begin{equation*}
W\coloneqq\Com_V(H)=\ker_V(\omega^\hh_0)
\end{equation*}
be the commutant \cite{FZ92,LL04} of $H$ in $V$, which has conformal vector $\omega-\omega^\hh$ and central charge $c-r$, where $\omega$ is the conformal vector of $V$. Then, we consider the double commutant of $H$, i.e.\ the kernel of $(\omega-\omega^\hh)_0$ in $V$:
\begin{prop}[\cite{Mas14}]\label{prop:asslat}
Let $V$ be a \strat{} \voa{}. Then $V$ contains the (in general non-full) \vosa{}
\begin{equation*}
\Com_V(\Com_V(H))=\ker_V((\omega-\omega^\hh)_0)\cong V_L
\end{equation*}
for some positive-definite, even lattice $L$ of rank $\rk(L)=r=\rk(V_1)$, which is unique up to isomorphism.
\end{prop}
The double commutant $\Com_V(\Com_V(H))$ is a conformal extension of $H$, and it is the largest such extension of $H$ inside $V$. The main point of \cite{Mas14} is that it is isomorphic to a lattice \voa{} $V_L$ for some positive-definite, even lattice $L$ of rank $\rk(L)=r$. As all Cartan subalgebras $\hh$ are conjugate under inner automorphisms of $V_1$, which extend to automorphisms in $\Aut(V)$, the lattice \voa{} $V_L$ and hence the lattice $L$ are unique up to isomorphism.

We call $L=L_V$ the \emph{associated lattice} of $V$.

We denote the bilinear form on the lattice $L$ also by $\langle\cdot,\cdot\rangle$. This is justified as the Cartan subalgebra $(\hh,\langle\cdot,\cdot\rangle)$ is naturally isometric to $(L\otimes_\Z\C,\langle\cdot,\cdot\rangle)$. Since $\hh$ acts semisimply on $V$ via the zero-mode, let $P\subseteq\hh^*$, identified with $\hh$ via $\langle\cdot,\cdot\rangle$, denote the weights of this action, i.e.\ $\alpha\in\hh$ is in $P$ if and only if there is a non-zero $v\in V$ such that $h_0v=\langle h,\alpha\rangle v$ for all $h\in\hh$. Then $P$ is a positive-definite lattice and $L$ is an even, full-rank sublattice of $P$ (i.e.\ $L\otimes_\Z\C=P\otimes_\Z\C$) such that $P\subseteq L'$.

As a side note, we point out that the \vosa{} $\langle\g\rangle$ generated by the semisimple part of $V_1$ contains an (in general non-full) \vosa{} isomorphic to the lattice \voa{} $V_{Q_\g}$ where $Q_\g\coloneqq\sqrt{k_1}\,Q^l_1\oplus\dots\oplus\sqrt{k_s}\,Q^l_s$ and $Q^l_i$ is the lattice spanned by the long roots of $\g_i$ normalised to have squared norm~$2$ (see, e.g., \cite{DM06b}). Hence, the associated lattice $L$ must be isomorphic to an extension of $Q_\g$, which is of the same rank if and only if $V_1$ is semisimple. In other words, there are the conformal extensions $H\subseteq V_{Q_\g}\otimes M_{\hat{\mathfrak{a}}}(1,0)\subseteq V_L$ inside $V_L$ up to isomorphism.

\begin{rem}\label{rem:asslatsuper}
The assertion of \autoref{prop:asslat} holds analogously for \strat{} \svoa{}s $V$, but the associated lattice~$L$ can now be an odd lattice (cf.\ \cite{DH12}). Similarly, the Heisenberg commutant $W$ is in general a \svoa{}. Indeed, we can apply the proposition to the \strat{} \vosa{} $V^{\bar0}$ and consider $V$ as an extension of $V^{\bar0}$, which has the same weight-$1$ Lie algebra (see also \autoref{sec:evensub}).
\end{rem}

\smallskip

We turn our attention to the Heisenberg commutant $W=\Com_V(H)$, which by construction satisfies $W_1=\{0\}$. The \voa{}s $V_L$ and $W$ form a dual (or Howe or commuting) pair in $V$, i.e.\ they are their mutual commutants, and $V$ is isomorphic to a (conformal) extension of $W\otimes V_L$, i.e.\
\begin{equation*}
V\supseteq W\otimes V_L.
\end{equation*}
Crucially, $W$ is \strat{} by Section~4.3 in \cite{CKLR19} (see also \cite{CGN21}) so that $\Rep(W)$ is a modular tensor category. Let $\{W^i\}$ denote the finitely many irreducible $W$-modules up to isomorphism (with $W^0\cong W$). The \voa{} $V$ decomposes into a direct sum of irreducible $W\otimes V_L$-modules. More precisely, by the theory of mirror extensions \cite{CKM22,Lin17} we obtain:
\begin{prop}\label{prop:asslatdecomp}
Let $V$ be a \strat{} \voa{}. Let $L$ denote the associated lattice and $W$ the Heisenberg commutant of $V$. Then $V$ is a simple-current extension of the \strat{} \voa{} $W\otimes V_L$,
\begin{equation*}
V=\bigoplus_{\alpha+L\in A}W^{\tau(\alpha+L)}\otimes V_{\alpha+L}
\end{equation*}
for some subgroup $A\leq L'/L$, with $\Rep(V_L)_V\cong\mathcal{C}(A)$ the corresponding full subcategory of $\Rep(V_L)\cong\mathcal{C}(L'/L)$, for some pointed full subcategory $\Rep(W)_V$ of $\Rep(W)$ and for some braid-reversing equivalence $\tau\colon\Rep(V_L)_V\to\Rep(W)_V$.
\end{prop}
We remark that while $\Rep(W)$ and $\Rep(V_L)\cong\mathcal{C}(L'/L)$ are modular tensor categories, the full subcategories $\Rep(V_L)_V\cong\mathcal{C}(A)$ and $\Rep(W)_V$ are in general only premodular categories. In particular, the quadratic form restricted to $A$ is in general degenerate. The subgroup $A\leq L'/L$ can be written as $A=P/L$ with the lattice $P\subseteq L'$ described above.

\medskip

With the representation category of $V_L$ known, the representation category of~$V$ can be described in terms of that of $W$ and vice versa \cite{CKM17,YY21}. The following two propositions consider the pointedness and evenness of these categories. Recall that $\Rep(V_L)\cong\mathcal{C}(L'/L)$ is pointed and even.

In general, for a discriminant form $D$ and a subgroup $A\leq D$ the orthogonal complement of $A$ in $D$ is denoted by $A^\bot$. It satisfies $|A||A^\bot|=|D|$.
\begin{prop}\label{prop:simplecurrents}
In the situation of \autoref{prop:asslatdecomp}, $\Rep(V)$ is pointed if and only if $\Rep(W)$ is.
\end{prop}
\begin{proof}
By Corollaries~3.10 and 3.14 in \cite{YY21} (see also \cite{CKM17}, Section~4.3), the number of irreducible (simple-current) $V$- and $W$-modules are related via
\begin{align*}
|\Irr(V)|=\frac{|A^\bot|}{|A|}|\Irr(W)|\quad\text{and}\quad|\Irr(V)_\text{sc}|=\frac{|A^\bot|}{|A|}|\Irr(W)_\text{sc}|.
\end{align*}
By definition, $\Rep(V)$ is pointed if and only if $|\Irr(V)|=|\Irr(V)_\text{sc}|$ and analogously for $W$. Hence, the assertion follows.
\end{proof}

Suppose in the following that the modular tensor categories $\Rep(V)$ and $\Rep(W)$ are pointed. As braided fusion categories they are hence characterised by some discriminant forms $A_V=(A_V,q_V)$ and $A_W=(A_W,q_W)$. We study the evenness (or pseudo-unitarity) discussed in \autoref{sec:mtc}.
\begin{prop}\label{prop:even}
In the situation of \autoref{prop:asslatdecomp}, suppose that $\Rep(V)$ and $\Rep(W)$ are pointed. Then $\Rep(V)$ is even, i.e.\ $\Rep(V)\cong\mathcal{C}(A_V)$, if and only if $\Rep(W)$ is, i.e.\ $\Rep(W)\cong\mathcal{C}(A_W)$.
\end{prop}
The statement of the proposition probably also holds in the absence of the pointedness assumption.
\begin{proof}
Suppose that $\Rep(W)\cong\mathcal{C}(A_W)$ is even or equivalently that all categorical dimensions equal $1$. The same is true for $\Rep(V_L)\cong\mathcal{C}(L'/L)$ and hence for the Deligne product $\Rep(W\otimes V_L)\cong\mathcal{C}(A_W)\boxtimes\mathcal{C}(L'/L)\cong\mathcal{C}(A_W\times L'/L)$. Any \voa{} extension of $W\otimes V_L$, such as $V$, corresponds to an isotropic subgroup $I\leq A_W\times L'/L$, and its representation category is equivalent as a modular tensor category to $\mathcal{C}(I^\bot/I)$, which is again even (see, e.g., \cite{EMS20a,Moe16}).

Conversely, assume that $\Rep(V)\cong\mathcal{C}(A_V)$ is even. Let $W^\beta$, $\beta\in A_W$, be any irreducible $W$-module. Then, by the results in \cite{YY21}, there is a $V_L$-module $V_{\gamma+L}$ such that $W^\beta\otimes V_{\gamma+L}$ is contained in an irreducible (untwisted) $V$-module, which is of the form $\bigoplus_{\alpha+L\in A}W^{\beta+\tau(\alpha+L)}\otimes V_{\gamma+\alpha+L}\supseteq W^\beta\otimes V_{\gamma+L}$. It is not difficult to see that the categorical dimension of $W^\beta$ must be $1$, as this is true for all irreducible $V$- and $V_L$-modules. Hence, $\Rep(W)$ is even.
\end{proof}
If the conditions in both propositions are satisfied, then $\Rep(W)\cong\mathcal{C}(A_W)$. Also set $A_L\coloneqq L'/L$ so that $\Rep(V_L)\cong\mathcal{C}(A_L)$. The braid-reversing equivalence $\tau$ can then be interpreted as an \emph{anti-isometry}\footnote{Here, we take this to mean that $\tau$ is a group isomorphism that maps one quadratic form to the negative of the other.} $\tau\colon A\to A'$ of the subgroups
\begin{equation*}
A\leq A_L\quad\text{and}\quad\tau(A)=A'\leq A_W
\end{equation*}
where $A'\leq A_W$ corresponds to the subcategory $\Rep(W)_V$ of $\Rep(W)$. The \voa{} extensions of $W\otimes V_L$, which has representation category equivalent to $\mathcal{C}(A_W\times A_L)$, correspond bijectively to the isotropic subgroups of $A_W\times A_L$, and the isotropic subgroup $I=\{(\tau(x),x)\,|\,x\in A\}$ yields the \voa{} $V$ with $\Rep(V)\cong\mathcal{C}(A_V)\cong\mathcal{C}(I^\bot/I)$ \cite{EMS20a,Moe16}.

The first equation in the proof of \autoref{prop:simplecurrents} implies that the number of irreducible $V$-modules is given by $|\Irr(V)|=|I^\bot/I|=[A_W:A'][A_K:A]$.


\subsection{Weight-1 Lie Algebra}\label{sec:root}

Now, based on the previous section, we describe the root system of the weight-$1$ Lie algebra of a nice \voa{}.

\medskip

Let $V$ be a \strat{} \voa{}, $L$ its associated lattice, $W$ its Heisenberg commutant and consider the decomposition of $V$ as simple-current extension of $W\otimes V_L$ given in \autoref{prop:asslatdecomp}. For now, also assume that $W$ satisfies the positivity condition. Then also $V$ satisfies the positivity condition as an extension of $W\otimes V_L$ (see, e.g., Lemma~5.2 in \cite{ELMS21}), and both $W$ and~$V$ are nice (and their representation categories hence both even). Then, recalling that $W_1=\{0\}$ and that also $V_L$ satisfies the positivity condition, the weight-$1$ Lie algebra of $V$ is given by
\begin{equation*}
V_1=\!\!\bigoplus_{\alpha+L\in A}\!\!\bigl(W^{\tau(\alpha+L)}\otimes V_{\alpha+L}\bigr)_1=\hh\oplus\!\!\bigoplus_{q\in(0,1]}\bigoplus_{\substack{\alpha+L\in A\\\langle\alpha,\alpha\rangle/2\in q+\Z}}\!\!\!\!\!\!W^{\tau(\alpha+L)}_{1-q}\otimes\Bigl(\!\!\bigoplus_{\substack{\beta\in\alpha+L\\\langle\beta,\beta\rangle/2=q}}\!\!\!\!\C\ee_\beta\Bigr)
\end{equation*}
where
\begin{equation*}
\hh\coloneqq\C\vac\otimes\{k(-1)\ee_0\,|\,k\in L\otimes_\Z\C\}
\end{equation*}
is a Cartan subalgebra of $V_1$. This agrees with the original choice of Cartan subalgebra $\hh$ made in order to obtain the above decomposition in the first place. The root system of $V_1$ corresponding to $\hh$ is
\begin{equation*}
\Phi=\bigcup_{q\in(0,1]}\bigl\{\beta\in P\,\big|\,\langle\beta,\beta\rangle/2=q,\ W^{\tau(\beta+L)}_{1-q}\neq\{0\}\bigr\}\subseteq P\subseteq L'
\end{equation*}
with the positive-definite sublattice $P\subseteq L'$ such that $A=P/L\leq L'/L$. As the Lie algebra $V_1$ is in general reductive (and not necessarily semisimple), the roots $\Phi$ may span a proper (or even zero) subspace of $P\otimes_\Z\C=L\otimes_\Z\C$, the latter having dimension equal to the rank of $V_1$.

The root multiplicities, which can be at most $1$, are given by
\begin{equation*}
\dim(W^{\tau(\beta+L)}_{1-q})\in\{0,1\}
\end{equation*}
for all $q\in(0,1]$ and all $\beta\in P$ of norm $\langle\beta,\beta\rangle/2=q$. We remark that the dimensions of the root-multiplicity spaces correspond to certain coefficients of the vector-valued modular form of weight~$0$ with components $\{\ch_{W^i}\}$ \cite{Zhu96}.

The roots $\beta\in\Phi\subseteq P$ have norms $\langle\beta,\beta\rangle/2=q$ for $q\in(0,1]$ with respect to the invariant bilinear form $\langle\cdot,\cdot\rangle$ on $V$ normalised such that $\langle\vac,\vac\rangle=-1$ (or with respect to the bilinear form $\langle\cdot,\cdot\rangle$ on $L\otimes_\Z\C$). On the other hand, the root norms $\langle\beta,\beta\rangle/2$ are in $\{\sfrac{1}{k_i},\sfrac{1}{k_il_i}\}$ for each simple ideal $\g_i$ of $V_1$ with level $k_i\in\Ns$ and lacing number $l_i\in\{1,2,3\}$. (Recall that we view the roots in $\hh$ by identifying $\hh$ with $\hh^*$ via $\langle\cdot,\cdot\rangle$.)


\section{Self-Dual \VOA{}s}\label{sec:holvoa}
In this section we study nice, self-dual \voa{}s. We recall how the results from \autoref{sec:voas} can be used to classify these \voa{}s, in particular in central charge $24$ \cite{Hoe17}. This classification result and the corresponding statement about the internal structure of these \voa{}s (like their associated lattices) will be the starting point for the determination of the $2$-neighbourhood graph of central charge $24$ in \autoref{sec:class}, together with the orbifold construction described in \autoref{sec:orbifold}.

\medskip

Let $V$ be a \strat{}, self-dual (and hence nice) \voa{}. The central charge $c$ of $V$ is in $8\N$ by the modular invariance result in \cite{Zhu96}. Recall that $V_1$ is a reductive Lie algebra of rank $r$ at most $c$. Following \autoref{sec:cartan}, let $L$ be the associated lattice of $V$ (of rank $r$) and $W=\Com_V(V_L)$ the Heisenberg commutant (of central charge $c-r$). The corresponding decomposition from \autoref{prop:asslatdecomp} simplifies (cf.\ \cite{Hoe17}):
\begin{prop}\label{prop:holdecomp}
Let $V$ be a nice, self-dual \voa{}. Then $V$ is a simple-current extension of the \strat{} \voa{} $W\otimes V_L$,
\begin{equation*}
V=\bigoplus_{\alpha+L\in L'/L}W^{\tau(\alpha+L)}\otimes V_{\alpha+L}
\end{equation*}
for some braid-reversing equivalence $\tau\colon\Rep(V_L)\to\Rep(W)$.
\end{prop}
\begin{proof}
Since $V$ is self-dual, the first formula in the proof of \autoref{prop:simplecurrents} shows that $\Rep(V_L)_V=\Rep(V_L)$, i.e.\ $A=L'/L$ or equivalently $A^\bot=\{0\}$, and that $\Rep(W)_V=\Rep(W)$, i.e.\ $|\Irr(W)|=|A|$, so that there is a braid-reversing equivalence $\tau\colon\Rep(V_L)\to\Rep(W)$ and $V$ decomposes as asserted.
\end{proof}
In particular, like $\Rep(V_L)\cong\mathcal{C}(L'/L)$, $\Rep(W)$ is pointed, i.e.\ all irreducible $W$-modules are simple currents, and even. Hence, $\Rep(W)\cong\mathcal{C}(A_W)$ for some discriminant form $A_W$, and $\tau$ descends to an anti-isometry $\tau\colon L'/L\to A_W$. In other words, $\Rep(W)\cong\mathcal{C}(\overline{L'/L})$ where $\overline{L'/L}$ is the discriminant form of $L$ with the quadratic form multiplied by $-1$.

\medskip

We describe the weight-$1$ Lie algebra of $V$, specialising the results in \autoref{sec:root}. The decomposition of $V$ depends on a choice $\hh=\C\vac\otimes\{k(-1)\ee_0\,|\,k\in L\otimes_\Z\C\}\cong L\otimes_\Z\C$ of Cartan subalgebra of the reductive Lie algebra $V_1$. If $W$ satisfies the positivity condition (always true, e.g., if $c=0$, $8$, $16$ or $24$ and $V_1\neq\{0\}$ \cite{Hoe17}), then the Lie algebra $V_1$ decomposes as
\begin{equation*}
V_1=\hh\oplus\!\!\bigoplus_{q\in(0,1]}\bigoplus_{\substack{\alpha+L\in L'/L\\\langle\alpha,\alpha\rangle/2\in q+\Z}}\!\!\!\!W^{\tau(\alpha+L)}_{1-q}\otimes\Bigl(\!\!\bigoplus_{\substack{\beta\in\alpha+L\\\langle\beta,\beta\rangle/2=q}}\!\!\!\!\C\ee_\beta\Bigr)
\end{equation*}
with root system
\begin{equation*}
\Phi=\bigcup_{q\in(0,1]}\bigl\{\beta\in L'\,\big|\,\langle\beta,\beta\rangle/2=q,\ W^{\tau(\beta+L)}_{1-q}\neq\{0\}\bigr\}\subseteq L'.
\end{equation*}
We shall give more details in the special case of central charge~$24$ in \autoref{sec:lie24}.


\subsection{Enumeration of Self-Dual \VOA{}s}\label{sec:enumvoa}

In the following, we explain how \autoref{prop:holdecomp} can be used to classify the nice, self-dual \voa{}s $V$ in a given central charge, provided that the Heisenberg commutants $W$ that can appear are known, which then also fix the genera of the associated lattices $L$. Note, however, that it is in general very difficult to constrain the possible Heisenberg commutants $W$.

We remark that the following enumeration procedure can, with slight modifications, also be applied to classify any nice, not necessarily self-dual \voa{}s. For instance, in \autoref{sec:enumsvoa} we enumerate \voa{}s with representation category $\mathcal{C}(2_{\II}^{+2})$ rather than $\Vect$.

Also note that the following results can be used to study the automorphism group of $V$ (see \autoref{sec:outv}).

\medskip

In general, for a positive-definite, even lattice $L$ with discriminant form $A_L=L'/L$ there is a short exact sequence
\begin{equation*}
1\longrightarrow\O(L)_0\longrightarrow\O(L)\longrightarrow\overline{\O}(L)\longrightarrow1
\end{equation*}
where $\overline{\O}(L)\leq\O(A_L)$ describes the induced permutation action of $\O(L)$ on the discriminant form $A_L$.

Let $U$ be any nice \voa{} whose representation category is pointed, i.e.\ $\Rep(U)\cong\mathcal{C}(A_U)$ with some discriminant form $A_U=(A_U,q)$. Then, similarly, there exists a short exact sequence
\begin{equation*}
1\longrightarrow\Aut(U)_0\longrightarrow\Aut(U)\longrightarrow\overline{\Aut}(U)\longrightarrow1
\end{equation*}
where $\overline{\Aut}(U)\leq\O(A_U)$ is the group of permutations of the isomorphism classes of irreducible $U$-modules induced by the action of $\Aut(U)$ on them (see, e.g., \cite{Hoe17} for further details). We note that if $U=V_L$ is a lattice \voa{} (so that $A_U=A_L$), then $\overline{\Aut}(V_L)=\overline{\O}(L)$ as subgroups of $\O(A_L)$.

Now, if $U'=\bigoplus_{\alpha\in I}U^\alpha$ is a simple-current extension of $U$ by an isotropic subgroup $I\leq A_U$, then the subgroup $\Aut(U')_U$ of $\Aut(U')$ fixing $U$ setwise fits into the sequence
\begin{equation*}
1\longrightarrow\widehat{I}\longrightarrow\Aut(U')_U\longrightarrow\overline{\Aut}(U)_I\longrightarrow1
\end{equation*}
where $\overline{\Aut}(U)_I$ is the subgroup of $\overline{\Aut}(U)\leq\O(A_U)$ that fixes $I\leq A_U$ setwise and $\widehat{I}$ is the dual group of $I$. In particular, the induced action of $\Aut(U')_U$ on $U$ is the subgroup of $\Aut(U)$ projecting onto $\overline{\Aut}(U)_I$ (cf.\ \cite{Yam04b}).

\medskip

Returning to the setting of this section, we apply the above to the nice, self-dual \voa{} $U'=V$ that is a simple-current extension of $U=W\otimes V_L$ by the isotropic subgroup $I=\{(\tau(x),x)\,|\,x\in A_L\}\leq A_W\times A_L$ with anti-isometry $\tau\colon A_L\to A_W$. To describe the induced action of
\begin{equation*}
\Aut(V)_{\tau}\coloneqq\Aut(U')_U|_U=\Aut(V)_{W\otimes V_L}|_{W \otimes V_L}
\end{equation*}
on $U=W \otimes V_L$, we note that there is an exact sequence
\begin{equation*}
1\longrightarrow\Aut(W)_0\times\Aut(V_L)_0\longrightarrow\Aut(V)_\tau\longrightarrow(\overline{\Aut}(W)\times\overline{\Aut}(V_L))_\tau\longrightarrow1
\end{equation*}
where $S\coloneqq(\overline{\Aut}(W)\times\overline{\Aut}(L))_\tau$ denotes the setwise stabiliser of $I\leq A_W\times A_L$ in $\O(A_W)\times \O(A_L)$.

In order to describe $S=(\overline{\Aut}(W)\times\overline{\O}(L))_\tau$ further, we fix an arbitrary anti-isometry $\tau_0\colon A_L\to A_W$, which induces a group isomorphism $\iota_{\tau_0}\colon\O(A_L)\to\O(A_W)$. We use $\iota_{\tau_0}$ to identify both groups, suppressing the dependence on $\tau_0$ from the notation. Thus, we consider both $\overline{\Aut}(W)$ and $\overline{\Aut}(V_L)=\overline{\O}(L)$ as subgroups of $\O(A_L)$. An arbitrary anti-isometry $\tau\colon A_L\to A_W$ is then of the form $\tau=\tau_0 g$ with $g\in\O(A_L)$, allowing us to consider $\tau$ as an element in $\O(A_L)$.

It follows (cf.\ \cite{Nik80,Hoe17}) that the inequivalent self-dual extensions $V$ of $W\otimes V_L$, necessarily parametrised by an isotropic group $I$ of the above form, are given by the double cosets
\begin{equation*}
\overline{\Aut}(W)\backslash\!\O(A_L)/\overline{\O}(L)
\end{equation*}
and that $S=(\overline{\Aut}(W)\times\overline{\O}(L))_\tau$ is the stabiliser of $\tau\in\O(A_L)$ under the action of $\overline{\Aut}(W)\times\overline{\O}(L)$ on $\O(A_L)$.

\medskip

Let $S_L$ be the projection of $S$ onto the second factor. Then the induced action of $({\Aut}(W)\times\Aut(V_L))_\tau\leq\Aut(V_L)$ on $V_L\cong\C\vac\otimes V_L\subseteq W\otimes V_L$, given by some group $\Aut(V_L)_\tau\leq\Aut(V_L)$, can be obtained from the exact sequence
\begin{equation*}
1\longrightarrow\Aut(V_L)_0\longrightarrow\Aut(V_L)_\tau\longrightarrow S_L\longrightarrow 1.
\end{equation*}


\subsection{Central Charge 24}\label{sec:24}

We specialise to central charge $c=24$. The classification of these \voa{}s was conjectured by Schellekens, who proved that the weight-$1$ structure $V_1$ must be one of $71$ Lie algebras (either zero, abelian of rank $24$ or semisimple), referred to as Schellekens' list \cite{Sch93}. Rigorous and systematic classification proofs are given in \cite{Hoe17,MS23,ELMS21,HM22,MS21} (see also \cite{DSW23}), where it is shown using different approaches that each of the $70$ non-zero Lie algebras is realised by exactly one nice, self-dual \voa{} of central charge $24$, with the notable exception that it is not known whether the moonshine module $V^\natural$ \cite{FLM88} is the unique such \voa{} $V$ with $V_1=\{0\}$.

In this text, we shall focus on the approach in \cite{Hoe17} based on associated lattices. If $V_1$ vanishes, like for the moonshine module $V^\natural$, then also the associated lattice $L$ does and the corresponding decomposition $V\cong V\otimes\C\vac$ is ineffective. In the following, we assume that $V_1\neq\{0\}$ and hence $L\neq\{0\}$. For central charge $24$ it was shown in \cite{Hoe17,Lam20} that the Heisenberg commutant comes from the Leech lattice $\Lambda$ and an element in the Conway group $\Co_0=\O(\Lambda)$:
\begin{prop}\label{prop:commconway}
Let $V$ be a nice, self-dual \voa{} of central charge $24$ with $V_1\neq\{0\}$. Then its Heisenberg commutant $W$ is isomorphic to
\begin{equation*}
W\cong V_{\Lambda_\mu}^{\hat\mu}
\end{equation*}
for one of 11 conjugacy classes $\mu\in\O(\Lambda)$ of the Leech lattice $\Lambda$, listed in~\autoref{table:11} and labelled by the letters A to K.
\end{prop}
Here, $\Lambda_\mu=(\Lambda^\mu)^\bot$ denotes the coinvariant lattice, and $\hat\mu$ is a lift of $\mu$ (restricted to $\Lambda_\mu$) to $\Aut(V_{\Lambda_\mu})$. Note that since $\mu$ acts fixed-point freely on $\Lambda_\mu$, all its lifts are standard lifts and conjugate (see, e.g., \cite{EMS20a}).

The genus of the associated lattice $L$ is determined by the rank $r=\rk(\Lambda^\mu)$ and the condition that $\Rep(V_{\Lambda_\mu}^{\hat\mu})\cong\mathcal{C}(\overline{L'/L})$, fixing the discriminant form of $L$. Hence, the associated lattice $L$ is from one of $11$ matching lattice genera (see \autoref{table:11}).

\begin{table}[ht]\caption{The $12$ Heisenberg commutants appearing in the nice, self-dual \voa{}s of central charge~$24$ and the corresponding lattice genera.}
\begin{tabular}{c|l|r|l||l|r||r}
\multicolumn{2}{l|}{Commutant} & $c$ & $\O(\Lambda)$ & Lattice genus\tablefootnote{Note that there is a typo for lattice genus I in \cite{Hoe17}, Table~4 and \cite{HM22}, Table~1.} & No.\ &  VOAs\ \\\hline\hline
A & $\mathcal{C}(1)$                              &  0 & $\sAA$ & $\gAA$ & 24 & 24\\
B & $\mathcal{C}(2_{\II}^{+10})$                  &  8 & $\sBB$ & $\gBB$ & 17 & 17\\
C & $\mathcal{C}(3^{-8})$                         & 12 & $\sCC$ & $\gCC$ &  6 &  6\\
D & $\mathcal{C}(2_{\II}^{-10}4_{\II}^{-2})$      & 12 & $\sDD$ & $\gDD$ &  2 &  9\\
E & $\mathcal{C}(2_6^{+2}4_{\II}^{-6})$           & 14 & $\sEE$ & $\gEE$ &  5 &  5\\
F & $\mathcal{C}(5^{+6})$                         & 16 & $\sFF$ & $\gFF$ &  2 &  2\\
G & $\mathcal{C}(2_{\II}^{+6}3^{-6})$             & 16 & $\sGG$ & $\gGG$ &  2 &  2\\
H & $\mathcal{C}(7^{+5})$                         & 18 & $\sHH$ & $\gHH$ &  1 &  1\\
I & $\mathcal{C}(2_1^{+1}4_5^{-1}8_{\II}^{-4})$   & 18 & $\sII$ & $\gII$ &  1 &  1\\
J & $\mathcal{C}(2_{\II}^{+4}4_{\II}^{-2}3^{-5})$ & 18 & $\sJJ$ & $\gJJ$ &  1 &  2\\
K & $\mathcal{C}(2_{\II}^{-2}4_{\II}^{-2}5^{+4})$ & 20 & $\sKK$ & $\gKK$ &  1 &  1\\\hline
L & $\mathcal{C}(1)$                              & 24 &        & $\gLL$ &  1 &(1)\\
\end{tabular}
\label{table:11}
\end{table}

The result also shows that $W\cong V_{\Lambda_\mu}^{\hat\mu}$ satisfies the positivity condition, which also follows more conceptually from the fact that $W=\Com_V(H)$ is an extension of the parafermion vertex operator algebra $\Com_{\langle V_1\rangle}(H)$ (cf.\ proof of Proposition~5.3 in \cite{ELMS21}), noting that, except for the case of $V_1$ abelian, $\langle V_1\rangle$ is a tensor product of simple affine \voa{}s at positive integer levels and a full \vosa{} of $V$.\footnote{Both statements are not true anymore in central charges greater than $24$.}

The main consequence of the above result is that the classification of the nice, self-dual \voa{}s of central charge $24$ is governed by the Leech lattice~$\Lambda$ and the Conway group $\Co_0=\O(\Lambda)$. We shall see in this text that this is still true for self-dual \svoa{}s (see \autoref{thm:commconway}).

\medskip

The existence of a nice, self-dual \voa{} $V$ of central charge $24$ with $V_1$ isomorphic to one of the $70$ non-zero Lie algebras on Schellekens' list is then proved in \cite{Hoe17} by realising $V$ explicitly as an extension of $V_{\Lambda_\mu}^{\hat\mu}\otimes V_L$, going through all $11$ lattice genera and all lattices in each genus (their numbers are given in \autoref{table:11}), i.e.\ by applying the enumeration method described in \autoref{sec:enumvoa}. We detail in \autoref{sec:lie24} how to determine the Lie algebra structure of $V_1$ in the above extension picture. In principle, $V_1$ and the \voa{} structure of $V$ depend not only on the lattice $L$ but also on the choice of the braid-reversing equivalence $\tau\colon\Rep(V_L)\to\Rep(V_{\Lambda_\mu}^{\hat\mu})$. In practice, this happens only for commutants of types D and J (compare the last two columns in \autoref{table:11}).

In fact, knowing the automorphism groups of the Heisenberg commutants $V_{\Lambda_\mu}^{\hat\mu}$, the approach in \cite{Hoe17} (see \autoref{sec:enumvoa}) yields a classification of the nice, self-dual \voa{}s of central charge $24$ with non-zero $V_1$. This computation is performed in \cite{BLS23}, with almost complete results already in \cite{Hoe17}. The outcome is that for each non-zero weight-$1$ Lie algebra on Schellekens' list there is an up to isomorphism unique nice, self-dual \voa{} $V$ of central charge $24$ with that weight-$1$ structure.


\subsection{Weight-1 Lie Algebra}\label{sec:lie24}

In the following, specialising the results in \autoref{sec:root}, we describe the weight\nobreakdash-$1$ Lie algebra $V_1$ of a nice, self-dual \voa{} $V$ of central charge $24$ with $V_1\neq\{0\}$ based on the decomposition of $V$ given in \autoref{prop:holdecomp}. Indeed, $V$ is a simple-current extension of $V_{\Lambda_\mu}^{\hat\mu}\otimes V_L$ where $\mu\in\Co_0$ and the lattice $L$ are as in \autoref{table:11}, i.e.\
\begin{equation*}
V=\bigoplus_{\alpha+L\in L'/L}V_{\Lambda_\mu}^{\hat\mu}(\tau(\alpha+L))\otimes V_{\alpha+L}
\end{equation*}
with braid-reversing equivalence $\tau\colon\Rep(V_L)\to\Rep(V_{\Lambda_\mu}^{\hat\mu})$. Then, as $(V_{\Lambda_\mu}^{\hat\mu})_1=\{0\}$ and because both $V_{\Lambda_\mu}^{\hat\mu}$ and $V_L$ satisfy the positivity condition, the weight-$1$ Lie algebra of $V$ is given by
\begin{equation*}
V_1=\hh\oplus\bigoplus_{d\mid n}\!\!\bigoplus_{\substack{\alpha+L\in L'/L\\\langle\alpha,\alpha\rangle/2\in 1/d+\Z}}\!\!\!\!\!\!V_{\Lambda_\mu}^{\hat\mu}(\tau(\alpha+L))_{1-1/d}\otimes\Bigl(\!\!\!\!\bigoplus_{\substack{\beta\in\alpha+L\\\langle\beta,\beta\rangle/2=1/d}}\!\!\!\!\!\!\C\ee_\beta\Bigr)
\end{equation*}
with Cartan subalgebra $\hh=\C\vac\otimes\{k(-1)\ee_0\,|\,k\in L\otimes_\Z\C\}$ where $n\coloneqq|\hat\mu|$ is the order of a standard lift in $\Aut(V_\Lambda)$ of $\mu\in\O(\Lambda)$. Typically, $n$ equals the order of $\mu$, but for types D, J and K the order doubles. The roots of $V_1$, viewed in $\hh\cong L\otimes_\Z\C$ by identifying $\hh^*$ with $\hh$ via the invariant bilinear form $\langle\cdot,\cdot\rangle$, are
\begin{equation*}
\Phi=\bigcup_{d\mid n}\bigl\{\beta\in L'\,\big|\,\langle\beta,\beta\rangle/2=1/d,\ V_{\Lambda_\mu}^{\hat\mu}(\tau(\beta+L))_{1-1/d}\neq\{0\}\bigr\}\subseteq L'
\end{equation*}
with root multiplicities
\begin{equation*}
\dim(V_{\Lambda_\mu}^{\hat\mu}(\tau(\beta+L))_{1-1/d})\in\{0,1\}
\end{equation*}
for all $d\,|\,n$ and all $\beta\in L'$ with $\langle\beta,\beta\rangle/2=\sfrac{1}{d}$. The roots $\beta\in\Phi\subseteq L'$ have norms $\langle\beta,\beta\rangle/2=\sfrac{1}{d}$ for $d\,|\,n$.

Note that the root multiplicities correspond to the singular terms of the vector-valued modular form for $\SLZ$ of weight $-\rk(L)/2$ with components
\begin{equation*}
\ch_{V_{\Lambda_\mu}^{\hat\mu}(\tau(\alpha+L))}(\tau)/\eta(\tau)^{\rk(L)},\quad\alpha+L\in L'/L.
\end{equation*}
This modular form was described (as lift of some eta product) for genera A, B, C, F, G, H in \cite{Sch06,Moe21}, for genus D in \cite{HS14} and for genera E and I in \cite{Sch09}. Nils Scheithauer informed us that he also has a nice description of the vector-valued modular forms for genera J and K. (See also \cite{DSW23}.)

In particular, singular terms can only appear in components $\alpha+L\in L'/L$ with $\langle\alpha,\alpha\rangle/2\in\sfrac{1}{d}+\Z$ and $d\alpha\in L$ for $d\,|\,n$. Hence,
\begin{align*}
V_1&=\hh\oplus\bigoplus_{d\mid n}\!\!\bigoplus_{\substack{\alpha+L\in L'/L\\\langle\alpha,\alpha\rangle/2\in 1/d+\Z\\d\alpha\in L}}\!\!\!\!V_{\Lambda_\mu}^{\hat\mu}(\tau(\alpha+L))_{1-1/d}\otimes\Bigl(\!\!\!\!\bigoplus_{\substack{\beta\in\alpha+L\\\langle\beta,\beta\rangle/2=1/d}}\!\!\!\!\!\!\C\ee_\beta\Bigr)
\end{align*}
with root system
\begin{equation*}
\Phi=\bigcup_{d\mid n}\bigl\{\beta\in L'\cap L/d\,\big|\,\langle\beta,\beta\rangle/2=1/d,\ V_{\Lambda_\mu}^{\hat\mu}(\tau(\beta+L))_{1-1/d}\neq\{0\}\bigr\}\subseteq L'.
\end{equation*}
For genera A, B, C, F, G, H all the root multiplicities in the above expression are non-zero so that $\Phi=\bigcup_{d\mid n}\{\beta\in L'\cap L/d\,|\,\langle\beta,\beta\rangle/2=1/d\}\subseteq L'$. For all genera, the roots in $\Phi$ of length $1$ are $\{\beta\in L\,|\,\langle\beta,\beta\rangle/2=1\}$.


\section{Even \VOSA{}}\label{sec:evensub}

In the following, in analogy to \autoref{sec:holvoa} and again using the results from \autoref{sec:voas}, we describe the internal structure of the even part of nice, self-dual \svoa{}s of central charge $c\in8\Ns$, i.e.\ of nice \voa{}s with representation category $\mathcal{C}(2_{\II}^{+2})$, and based thereon their self-dual extensions. We also explain how to enumerate these \voa{}s based on this structure. We shall use these results in \autoref{sec:class} as one of the two main tools to classify the nice, self-dual \svoa{}s of central charge~$24$.


\subsection{Glueing Types}\label{sec:glueing}

We describe the Heisenberg commutant and the associated lattice of the even part of a nice, self-dual \svoa{} $V$ of central charge $c\in8\Ns$, and the different ways they can be extended (or \emph{glued}) to $V$.

\medskip

Let $U$ be a nice \voa{} with $\Rep(U)\cong\mathcal{C}(2_{\II}^{+2})$, let $K$ be its associated lattice with corresponding discriminant form $A_K\coloneqq K'/K$ and $W$ its Heisenberg commutant. By \autoref{prop:asslatdecomp} there is a subgroup $A\leq A_K$ and a full subcategory $\Rep(W)_V$ of $\Rep(W)$ such that $U$ is a simple-current extension
\begin{equation*}
U=\bigoplus_{\alpha+K\in A}W^{\tau(\alpha+K)}\otimes V_{\alpha+K}
\end{equation*}
where $\tau\colon\Rep(V_K)_V\to\Rep(W)_V$ is a braid-reversing equivalence between the full subcategory $\Rep(V_K)_V\cong\mathcal{C}(A)$ of $\Rep(V_K)\cong\mathcal{C}(A_K)$ and $\Rep(W)_V$.

By assumption, the representation category $\Rep(U)\cong\mathcal{C}(2_{\II}^{+2})$ is pointed and even with four irreducible modules indexed by the discriminant form $2_{\II}^{+2}$. Then, by \autoref{prop:simplecurrents} and \autoref{prop:even}, also the representation category of the Heisenberg commutant $W$ is pointed and even, i.e.\ $\Rep(W)\cong\mathcal{C}(A_W)$ for some discriminant form $A_W$.

Recall that the braid-reversing equivalence $\tau\colon\Rep(V_K)_V\to\Rep(W)_V$ can be interpreted as an anti-isometry $\tau\colon A\to A'$ with $A\leq A_K$ and $\tau(A)=A'\leq A_W$. The representation category of $W\otimes V_K$ is given by the Deligne product $\mathcal{C}(A_W)\boxtimes\mathcal{C}(A_K)\cong\mathcal{C}(A_W\times A_K)$, which is pointed with discriminant form $A_W\times A_K$.

Also recall that the \voa{} extensions of nice \voa{}s with pointed representation categories $\mathcal{C}(D)$ have a rather simple description in terms of isotropic subgroups (see, e.g., \cite{EMS20a,Moe16}). Specifically, such extensions correspond bijectively to isotropic subgroups $I$ of the discriminant form~$D$ and the representation category of the extension is again pointed and even with discriminant form $I^\bot/I$, where $I^\bot$ denotes the orthogonal complement of $I$ in $D$.

In the present setting, $U$ is a simple-current extension of $W\otimes V_K$ corresponding to the isotropic subgroup
\begin{equation*}
I=I_U=\{(\tau(x),x)\,|\,x\in A\}\leq A_W\times A_K
\end{equation*}
of order $|I|=|A|=|A'|$ satisfying $I^\bot/I\cong 2_{\II}^{+2}$. Since $|I^\bot/I|=[A_W:A'][A_K:A]$, this implies that
\begin{equation*}
[A_W:A'][A_K:A]=4,
\end{equation*}
prompting us to distinguish three cases (or \emph{glueing types}):
\begin{enumerate}
\item[(I)] $[A_W:A']=1$ and $[A_K:A]=4$ (so that $|A_K|=4\,|A_W|$),
\item[(II)] $[A_W:A']=2$ and $[A_K:A]=2$ (so that $|A_K|=|A_W|$),
\item[(III)] $[A_W:A']=4$ and $[A_K:A]=1$ (so that $4\,|A_K|=|A_W|$).
\end{enumerate}
It is apparent that
\begin{equation*}
I+A'^\bot\times A^\bot\subseteq I^\bot
\end{equation*}
where $A^\bot$ is the orthogonal complement of $A$ in $A_K$ and $A'^\bot$ that of $A'$ in $A_W$. Note that $|A^\bot|=[A_K:A]$ and $|A'^\bot|=[A_W:A']$.

In types I and III, $I$ has trivial intersection with $A'^\bot\times A^\bot$ so that in fact
\begin{equation*}
I^\bot=I+A'^\bot\times A^\bot.
\end{equation*}
In type~II, however, this cannot be the case as $I^\bot/I\cong 2_{\II}^{+2}$ is indecomposable but $A'^\bot\times A^\bot$ is not.

In the following, we describe the three cases in more detail, and in particular the self-dual extensions of $U$ as extensions of $W\otimes V_K$. Recall from \autoref{sec:graphmeth} that $U$ has three $\Z_2$-extensions, to the self-dual \voa{}s $W^{(1)}$ and $W^{(2)}$ and to the self-dual \svoa{} $V$. Let $\gamma_1$, $\gamma_2$ and $\gamma_3$, respectively, be the corresponding elements in $2_{\II}^{+2}\cong I^\bot/I$, of norms $0$, $0$ and $\sfrac{1}{2}\pmod{1}$.


\subsection*{Glueing Type I}

Here, $A'^\bot$ is trivial while $A^\bot$ has order~$4$. The orthogonal complement of $I$ in $A_W\times A_K$ is
\begin{equation*}
I^\bot=I+\{0\}\times A^\bot\quad\text{so that}\quad I^\bot/I\cong A^\bot,
\end{equation*}
which is isomorphic to $2_{\II}^{+2}$. Viewing $\gamma_i\in 2_{\II}^{+2}\cong A^\bot\leq A_K$, we set
\begin{equation*}
I_i\coloneqq I+\{0\}\times\{0,\gamma_i\}=I\cup(I+(0,\gamma_i))
\end{equation*}
for $i=1$, $2$ and $3$. Then $I_1$ and $I_2$ are isotropic subgroups of $A_W\times A_K$ satisfying $I_i^\bot=I_i$, while the subgroup $I_3$ is not quite isotropic. The self-dual vertex operator (super)algebra extensions $W^{(1)}$, $W^{(2)}$ and $W^{(3)}=V$ of $U$ corresponding to $I_1$, $I_2$ and $I_3$, respectively, are given by
\begin{equation*}
W^{(i)}=\bigoplus_{\alpha+K\in A}W^{\tau(\alpha+K)}\otimes(V_{\alpha+K}\oplus V_{\gamma_i+\alpha+K})
\end{equation*}
for $i=1$, $2$ and $3$. In each, $W$ forms a dual pair with $V_K\oplus V_{\gamma_i+K}\cong V_{K\cup(\gamma_i+K)}$.

The Cartan subalgebra $\hh=\C\vac\otimes\{k(-1)\ee_0\,|\,k\in K\otimes_\Z\C\}$ of $U_1$ is still a Cartan subalgebra of $W^{(1)}_1$ and $W^{(2)}_1$, and trivially one of $V_1=U_1$. Indeed, a Cartan subalgebra of $W^{(i)}_1$ is given by the centraliser of $\hh$ in $W^{(i)}_1$ \cite{Kac90}, but it follows from the above decomposition of $W^{(i)}$ that this centraliser is just
\begin{equation*}
C_{W^{(i)}_1}(\hh)=\hh
\end{equation*}
as only vectors with $\ee_0$, i.e.\ with zero lattice momentum, can appear. The above decomposition also implies that
\begin{equation*}
\Com_{W^{(i)}}(\Com_{W^{(i)}}(\langle\hh\rangle))=V_{K\cup(\gamma_i+K)}
\end{equation*}
for $i=1$, $2$ and $3$. Hence, $K\cup(\gamma_i+K)$, which is an index-$2$ extension of $K$, is the associated lattice of $W^{(i)}$ for $i=1$, $2$ and that of $V$ for $i=3$ (in which case it is odd), and $W$ is the Heisenberg commutant in each case (see \autoref{prop:asslat} and \autoref{rem:asslatsuper}).

In view of \autoref{sec:orbifold} we note that $U$ is the \fpvosa{} under an inner automorphism of $W^{(1)}$ and under one of $W^{(2)}$.

We summarise the discussion for glueing type~I:
\[
\begin{tabular}{l|l|l}
&Dual pair&Assoc.\ lattice\\\hline\hline
$U$ & $W\otimes V_K$ & $K$\\
$W^{(1)}$ & $W\otimes V_{K\cup(\gamma_1+K)}$ & $K\cup(\gamma_1+K)$\\
$W^{(2)}$ & $W\otimes V_{K\cup(\gamma_2+K)}$ & $K\cup(\gamma_2+K)$\\
$V$ & $W\otimes V_{K\cup(\gamma_3+K)}$ & $K\cup(\gamma_3+K)$
\end{tabular}
\]


\subsection*{Glueing Type III}

This case is similar to type~I, but now $A^\bot$ is trivial while $A'^\bot$ has order~$4$. The orthogonal complement of $I$ in $A_W\times A_K$ is
\begin{equation*}
I^\bot=I+A'^\bot\times\{0\}\quad\text{so that}\quad I^\bot/I\cong A'^\bot,
\end{equation*}
which is isomorphic to $2_{\II}^{+2}$. Viewing $\gamma_i\in 2_{\II}^{+2}\cong A'^\bot\leq A_W$, we define
\begin{equation*}
I_i\coloneqq I+\{0,\gamma_i\}\times\{0\}=I\cup(I+(\gamma_i,0))
\end{equation*}
for $i=1$, $2$, $3$. Then $I_1$ and $I_2$ are isotropic subgroups of $A_W\times A_K$ with $I_i^\bot=I_i$. The self-dual vertex operator (super)algebra extensions $W^{(1)}$, $W^{(2)}$ and $W^{(3)}=V$ of $U$ corresponding to $I_1$, $I_2$ and $I_3$, respectively, are
\begin{equation*}
W^{(i)}=\bigoplus_{\alpha+K\in A}(W^{\tau(\alpha+K)}\oplus W^{\gamma_i+\tau(\alpha+K)})\otimes V_{\alpha+K}
\end{equation*}
for $i=1$, $2$ and $3$. In each, $W\oplus W^{\gamma_i}$ and $V_K$ form a dual pair.

However, for $i=1$ and $2$ this is typically not the dual pair we shall be interested in as $K$ is typically not the associated lattice of $W^{(1)}$ or $W^{(2)}$.

Indeed, the rank of the associated lattice of $W^{(i)}$ equals the rank of the weight-$1$ Lie algebra of $W^{(i)}$, a Cartan subalgebra of which is given by the centraliser of $\hh=\C\vac\otimes\{k(-1)\ee_0\,|\,k\in K\otimes_\Z\C\}$ in $W^{(i)}_1$. Based on the above decomposition of $W^{(i)}$, this centraliser is:
\begin{lem}\label{lem:csaIII}
For type~III, a Cartan subalgebra of $W^{(i)}_1$, $i=1$, $2$, is given by
\begin{equation*}
C_{W^{(i)}_1}(\hh)=\hh\oplus(W^{\gamma_i})_1\otimes\C\ee_0.
\end{equation*}
\end{lem}
\begin{proof}
The centraliser of $\hh$ consists of those vectors in $W^{(i)}_1$ whose lattice part has zero momentum, i.e.\ with $\ee_0$. Based on the above decomposition of $W^{(i)}$, these lie in $\hh$ or $(W^{\gamma_i}\otimes V_K)_1$. The \voa{} $W$ may or may not satisfy the positivity condition, but $W^{\gamma_i}$ can only have positive $L_0$-eigenvalues, or else $W^{(i)}$ would not be of CFT-type, contradicting the assumption that $U$ satisfies the positivity condition. Hence, $(W^{\gamma_i}\otimes V_K)_1=(W^{\gamma_i})_1\otimes\C\ee_0$.
\end{proof}

Typically, $(W^{\gamma_i})_1\neq\{0\}$ so that the ranks of $W^{(1)}_1$ and $W^{(2)}_1$ are strictly larger than the rank of $U_1$, though not necessarily both the same. Hence, the associated lattices $L_1$ and $L_2$ of $W^{(1)}$ and $W^{(2)}$, respectively, are typically infinite-index extensions of $K$, i.e.\ of larger rank than $K$.

For instance, if $U$ has central charge $c=8$, $16$ or $24$, then unless $W^{(i)}_1=\{0\}$, like for the moonshine module $V^\natural$, the ranks do indeed satisfy $\rk(W^{(i)}_1)>\rk(U_1)$. If not, then $U$ would be a \fpvosa{} of $W^{(i)}$ under an inner automorphism of $W^{(i)}$ of order~$2$ (see \cite{HM22,BLS23}), a contradiction.

The Heisenberg commutants $\tilde{W}$ and $\tilde{\tilde{W}}$ of $W^{(1)}$ and $W^{(2)}$, respectively, are typically non-full subalgebras of $W$.

On the other hand, the Cartan subalgebra $\hh$ of $U_1$ is trivially also a Cartan subalgebra of $V_1=U_1$ so that $K$ is the associated lattice of $V$ and $W\oplus W^{\gamma_3}$ is the Heisenberg commutant.

In view of \autoref{sec:orbifold} we note that $U$ is the \fpvosa{} under a non-inner automorphism of $W^{(1)}$ and under one of $W^{(2)}$, as it follows from the discussion there that inner automorphisms can only yield glueing type I or II.

We summarise the discussion for glueing type~III, typically:
\[
\begin{tabular}{l|l|l|l}
&Dual pair&Lattice&Assoc.\\\hline\hline
$U$ & $W\otimes V_K$ & $K$ & yes\\\hline
\multirow{2}{*}{$W^{(1)}$} & $(W\oplus W^{\gamma_1})\otimes V_K$ & $K$ & no\\
 & $\tilde{W}\otimes V_{L_1}$ & $L_1$, $\rk(L_1)>\rk(K)$ & yes\\\hline
\multirow{2}{*}{$W^{(2)}$} & $(W\oplus W^{\gamma_2})\otimes V_K$ & $K$ & no\\
 & $\tilde{\tilde{W}}\otimes V_{L_2}$ & $L_2$, $\rk(L_2)>\rk(K)$ & yes\\\hline
$V$ & $(W\oplus W^{\gamma_3})\otimes V_K$ & $K$ & yes
\end{tabular}
\]

We shall later find two (non-typical) \voa{}s $U$ of central charge~$24$ for glueing type~III, namely the \fpvosa{}s of the moonshine module $V^\natural$ under the 2B- and under the 2A-involution, for which the ranks of the self-dual extensions do not (both) increase (see \autoref{sec:orbcomp}). They are the even parts of the odd moonshine module $\VO$ and of the tensor product $\VB\otimes F$ of the shorter moonshine module with a free fermion, respectively.


\subsection*{Glueing Type II}

This case is slightly more complicated than the other two. Both $A^\bot$ and $A'^\bot$ have order~$2$. We already argued that $A'^\bot\times A^\bot$ has non-trivial intersection with $I$ or equivalently that $I^\bot$ is strictly larger than $I+A'^\bot\times A^\bot$, recalling that both $A'^\bot\times A^\bot$ and $I^\bot/I$ have order~$4$.

Let $A^\bot=\{0,b\}$ and $A'^\bot=\{0,b'\}$ with $b$ and $b'$ of order~$2$. Since $(0,b)$ and $(b',0)$ cannot be in $I$, it follows that $(b',b)\in I$, i.e.\ $b\in A$, $b'\in A'$ and $b'=\tau(b)$. So, $I\cap A'^\bot\times A^\bot=\{(0,0),(b',b)\}$.

This also shows that $A^\bot\subseteq A$ and $A'^\bot\subseteq A'$, which means that the bilinear form on $A^\bot$ and $A'^\bot$ vanishes. This implies that either
\begin{enumerate}
\item[(a)] $b$ and $b'$ have norm $\sfrac{1}{2}\pmod{1}$ or
\item[(b)] $b$ and $b'$ have norm $0\pmod{1}$.
\end{enumerate}

The quotient $I^\bot/I$, isomorphic to $\Z_2\times\Z_2$ as a group, must contain two further elements besides $(0,0)+I$ and $(b',0)+I=(0,b)+I$. In the following, we identify the elements $\gamma_1$, $\gamma_2$ and $\gamma_3\in I^\bot/I$ and describe the corresponding extensions.

\subsubsection*{Glueing Type IIa}
Here, $\gamma_3\coloneqq(b',0)+I=(0,b)+I\in I^\bot/I$ has norm $\sfrac{1}{2}\pmod{1}$ so that the two further non-zero elements of $I^\bot/I\cong 2_{\II}^{+2}$ must both have norm $0\pmod{1}$. These are the two elements $\gamma_1=(c_1',c_1)+I$ and $\gamma_2=(c_2',c_2)+I$ that correspond to the extensions of $U$ to the self-dual \voa{}s $W^{(1)}$ and $W^{(2)}$, respectively. Neither $\gamma_1$ nor $\gamma_2$ is in $A'^\bot\times A^\bot$.

The self-dual vertex operator (super)algebra extensions $W^{(1)}$, $W^{(2)}$ and $V$ of $U$ corresponding to $\gamma_1$, $\gamma_2$ and $\gamma_3$, respectively, are
\begin{equation*}
W^{(i)}=\bigoplus_{\alpha+K\in A}(W^{\tau(\alpha+K)}\otimes V_{\alpha+K})\oplus(W^{c_i'+\tau(\alpha+K)}\otimes V_{c_i+\alpha+K})
\end{equation*}
for $i=1$, $2$ and
\begin{align*}
V&=\bigoplus_{\alpha+K\in A}W^{\tau(\alpha+K)}\otimes (V_{\alpha+K}\oplus V_{b+\alpha+K})\\
&=\bigoplus_{\alpha+K\in A}(W^{\tau(\alpha+K)}\oplus W^{b'+\tau(\alpha+K)})\otimes V_{\alpha+K}.
\end{align*}
In particular, neither extension $W^{(i)}$ contains modules corresponding to elements of $A_W\times A_K$ of the form $(a',0)$ with $a'\in A_W\setminus A'$ or $(0,a)$ with $a\in A_K\setminus A$. This shows that $W\otimes V_K$ is still a dual pair in $W^{(1)}$ and $W^{(2)}$. On the other hand, the decomposition of $V$ implies that $(W\oplus W^{b'})\otimes V_{K\cup(b+K)}$ is a dual pair in $V$.

The centraliser of the Cartan subalgebra $\hh=\C\vac\otimes\{k(-1)\ee_0\,|\,k\in K\otimes_\Z\C\}$ of $U_1$ in $W^{(i)}_1$ is again $\hh$ so that $\hh$ is also a Cartan subalgebra of $W^{(i)}_1$, and trivially also of $V_1=U_1$. Hence, these dual pairs describe the associated lattices (an odd lattice in the case of $V$) and Heisenberg commutants.

In view of \autoref{sec:orbifold} we note that, like for type~I, $U$ is the \fpvosa{} under an inner automorphism of $W^{(1)}$ and under one of $W^{(2)}$.

In summary, for glueing type~IIa one has:
\[
\begin{tabular}{l|l|l}
&Dual pair&Assoc.\ lattice\\\hline\hline
$U$ & $W\otimes V_K$ & $K$\\
$W^{(1)}$ & $W\otimes V_K$ & $K$\\
$W^{(2)}$ & $W\otimes V_K$ & $K$\\
$V$ & $(W\oplus W^{b'})\otimes V_{K\cup(b+K)}$ & $K\cup(b+K)$
\end{tabular}
\]

\subsubsection*{Glueing Type IIb}
Here, $\gamma_1\coloneqq(b',0)+I=(0,b)+I\in I^\bot/I$ has norm $0\pmod{1}$ and describes one extension of $U$ to a self-dual \voa{} $W^{(1)}$. There are two further non-zero elements $\gamma_2=(c_2',c_2)+I$ and $\gamma_3=(c_3',c_3)+I\in I^\bot/I$ of norm $0$ and $\sfrac{1}{2}\pmod{1}$, respectively, of which the former gives the other extension of $U$ to a self-dual \voa{} $W^{(2)}$.

The extension to $W^{(2)}$ looks exactly like in type~IIa. In particular, $W\otimes V_K$ is a dual pair in $W^{(2)}$, $\hh$ a Cartan subalgebra of $W^{(2)}_1$, $K$ the associated lattice of $W^{(2)}$ and $W$ its Heisenberg commutant. The same holds for the \svoa{} $V$, which is the extension
\begin{equation*}
V=\bigoplus_{\alpha+K\in A}(W^{\tau(\alpha+K)}\otimes V_{\alpha+K})\oplus(W^{c_3'+\tau(\alpha+K)}\otimes V_{c_3+\alpha+K})
\end{equation*}
corresponding to $\gamma_3$. The extension to $W^{(1)}$ is of the form
\begin{align*}
W^{(1)}&=\bigoplus_{\alpha+K\in A}W^{\tau(\alpha+K)}\otimes(V_{\alpha+K}\oplus V_{b+\alpha+K})\\
&=\bigoplus_{\alpha+K\in A}(W^{\tau(\alpha+K)}\oplus W^{b'+\tau(\alpha+K)})\otimes V_{\alpha+K}
\end{align*}
with dual pair $(W\oplus W^{b'})\otimes(V_K\oplus V_{b+K})$, which typically does not describe the associated lattice of $W^{(1)}$.

Indeed, a Cartan subalgebra of $W^{(1)}_1$ is given by the centraliser of the Cartan subalgebra $\hh$ of $U_1$, and, similarly to \autoref{lem:csaIII}, it follows from the decomposition of $W^{(1)}$ that this centraliser is $C_{W^{(1)}_1}(\hh)=\hh\oplus(W^{b'})_1\otimes\C\ee_0$. As for type~III, this Cartan subalgebra is typically strictly larger than $\hh$ so that the rank of the Lie algebra $W^{(1)}_1$ and of the associated lattice $L_1$ of $W^{(1)}$ is strictly larger than the rank of $K$ (for instance, if $c=8$, $16$ or $24$ and $W^{(1)}_1\neq\{0\}$).

The Heisenberg commutant $\tilde{W}$ of $W^{(1)}$ is typically a non-full subalgebra of $W$.

In view of \autoref{sec:orbifold} we note that $U$ is the \fpvosa{} under an inner automorphism of $W^{(2)}$ but under a non-inner one of $W^{(1)}$.

In summary, for glueing type~IIb, one typically has:
\[
\begin{tabular}{l|l|l|l}
&Dual pair&Lattice&Assoc.\\\hline\hline
$U$ & $W\otimes V_K$ & $K$ & yes\\\hline
\multirow{2}{*}{$W^{(1)}$} & $(W\oplus W^{b'})\otimes V_{K\cup b+K}$ & $K\cup b+K$ & no\\
 & $\tilde{W}\otimes V_{L_1}$ & $L_1$, $\rk(L_1)>\rk(K)$ & yes\\\hline
$W^{(2)}$ & $W\otimes V_K$ & $K$ & yes\\\hline
$V$ & $W\otimes V_K$ & $K$ & yes
\end{tabular}
\]


\subsection{Enumeration of Self-Dual \SVOA{}s}\label{sec:enumsvoa}

In analogy to \autoref{sec:enumvoa} for self-dual \voa{}s, it is possible to enumerate the even parts $U$ of nice, self-dual \svoa{}s for given Heisenberg commutant $W$ and associated lattice $K$. The idea is to count the inequivalent simple-current extensions of $W\otimes V_K$ to \voa{}s $U$, now with $\Rep(U)\cong\mathcal{C}(2_{\II}^{+2})$ rather than $\Vect$.

\medskip

As explained, these extensions are in bijection with certain isotropic subgroups $I=\{(\tau(x),x)\,|\,x\in A\}\leq A_W\times A_K$ of the discriminant form $A_W\times A_K$ associated with the pointed representation category of $W\otimes V_K$ satisfying $I^\bot/I\cong 2_{\II}^{+2}$. Here, $\tau\colon A\to A'$ is an anti-isometry between subgroups $A\leq A_K$ and $A'\leq A_W$, with the details depending on the glueing type (I, IIa, IIb or III).

In order to count the number of inequivalent extensions, we can proceed like in \autoref{sec:enumvoa}, only that now the extensions have representation category $\mathcal{C}(2_{\II}^{+2})$ rather that being self-dual. To this end, a slight modification of the double-coset expression presented there is needed.

More precisely, to enumerate the inequivalent anti-isometries $\tau\colon A\to A'$, we first determine the orbits of subgroups $A\leq A_K$ and $A'\leq A_W$ with the right properties (depending on the glueing type) under the action of $\overline{\O}(K)$ and $\overline{\Aut}(W)$, respectively.

Then, we count the equivalence classes of anti-isometries $\tau$ as in \autoref{sec:enumvoa}, but we replace $\overline{\O}(K)$ and $\overline{\Aut}(W)$ by the stabiliser of $A$ and $A'$, respectively, under the above actions. We can also determine the two \voa{} extensions $W^{(1)}$ and $W^{(2)}$ of $U$ corresponding to the two non-trivial isotropic subgroups of $I^\bot/I$ and hence, in particular, if a given edge is a loop.

In the case of a loop, we study the induced action of the stabiliser $S$ as in \autoref{sec:enumvoa} on $I^\bot/I$. If $S$ permutes the two subgroups that correspond to the extensions $W^{(1)}$ and $W^{(2)}$ of $U$, then $U$ belongs to a loop of type (3) as in \autoref{prop:order2conj}.


\section{Self-Dual Orbifolds of Order~2}\label{sec:orbifold}

In this section we describe orbifold constructions associated with nice, self-dual \voa{}s $V$ and automorphisms $g$ of $V$ of order~$2$ \cite{FLM88,EMS20a}.

As we saw in \autoref{sec:graphmeth}, the nice, self-dual \svoa{}s of central charge $c\in8\N$ correspond bijectively to the edges of the $2$-neighbourhood graph of the nice, self-dual \voa{}s of that central charge. And two such \voa{}s are $2$-neighbours if and only if they are their mutual orbifold constructions of order~$2$. We shall use this (together with the results in \autoref{sec:holvoa}) as one of the two main tools to determine the $2$-neighbourhood graph in central charge $24$ in \autoref{sec:class}.

As the \fpvosa{} $U=V^g$ has representation category $\Rep(V^g)\cong\mathcal{C}(2_{\II}^{+2})$, this section is complementary to the previous one, in which we started from $U$ and considered its self-dual extensions.

\medskip

Let $V$ be a nice, self-dual \voa{}, necessarily of central charge $c\in8\N$. Recall from \autoref{prop:holdecomp} the decomposition of $V$ in terms the dual pair $W\otimes V_L$, depending on a choice $\hh$ of Cartan subalgebra of $V_1$, with associated lattice $L$ and Heisenberg commutant $W$, as well as the corresponding structure of the reductive Lie algebra $V_1$ and its root system $\Phi$.

Let $g$ be an automorphism of $V$ of order~$2$. If $g$ has type~$0$, i.e.\ if the $L_0$-weights of the unique irreducible $g$-twisted $V$-module $V(g)$ are in $\frac{1}{2}\Z$, and if $V^g$ satisfies the positivity condition, i.e.\ if the $L_0$-weights of $V(g)$ are positive, then the \fpvosa{} $V^g$ has the representation category $\Rep(V^g)\cong\mathcal{C}(2_{\II}^{+2})$ and the unique self-dual \voa{} extension of $V^g$ other than $V$ is called the \emph{orbifold construction} $V^{\orb(g)}$ \cite{FLM88,EMS20a}.

In \autoref{sec:orbinner} and \autoref{sec:orbouter} we shall describe the orbifold constructions under inner and non-inner automorphisms, respectively, and shall see that they correspond to different glueing types (I, IIa, IIb or III) from \autoref{sec:glueing}.


\subsection{Automorphism Group}\label{sec:outv}

In the following, we study the automorphism group $\Aut(V)$ of $V$ \cite{HM22,BLS23}. We remark that, although we continue to assume this for simplicity, the results do not depend crucially on $V$ being self-dual.

The normal subgroup of \emph{inner automorphisms} of $V$ is $K\coloneqq\langle\{\e^{v_0}\,|\,v\in V_1\}\rangle$ and it contains the abelian subgroup $T\coloneqq\{\e^{v_0}\,|\,v\in\hh\}$ for a given choice of Cartan subalgebra $\hh$ of $V_1$. The \emph{outer automorphism group} of $V$ is the quotient $\Out(V)\coloneqq\Aut(V)/K$.

We make the somewhat technical assumption that the kernel of the restriction map $r\colon\Aut(V)\to\Aut(V_1)$ is contained in $T$. This is true, for instance, if $V$ has central charge $c=0$, $8$, $16$ or $24$ and $V_1\neq\{0\}$ \cite{BLS23}. Then the restriction map $r\colon\Aut(V)\to\Aut(V_1)$ induces an injective map
\begin{equation*}
\Out(V)\hookrightarrow\Out(V_1),
\end{equation*}
i.e.\ we may view $\Out(V)$ as an (in general proper) subgroup of $\Out(V_1)$. The outer automorphism group has been computed for all nice, self-dual \voa{}s of central charges $c=0$, $8$, $16$ and $24$ with $V_1\neq\{0\}$ \cite{BLS23}.

We describe the above embedding in more detail (cf.\ \cite{HM22}). Fixing also a choice of simple roots $\Delta\subseteq\Phi$ of $V_1$, $\Out(V_1)$ is isomorphic to
\begin{equation*}
\Out(V_1)\cong\Aut(V_1)_{\{\Delta\}}/\Aut(V_1)_{\Delta}\cong\O(\hh)_{\{\Delta\}},
\end{equation*}
the quotient of the setwise stabiliser by the pointwise stabiliser of $\Delta$. This identity holds in any reductive Lie algebra,\footnote{More precisely, we assume that the abelian part of this reductive Lie algebra is equipped with a non-degenerate bilinear form, which is always the case for the reductive Lie algebra $V_1$ of a nice \voa{} $V$ (see \autoref{sec:cartan}).} and for a semisimple Lie algebra precisely gives the diagram automorphisms. Consequently, we may regard $\Out(V)\leq\O(\hh)_{\{\Delta\}}$. If an automorphism of $V$ fixes the choice $\hh$ of Cartan subalgebra setwise, then it must also fix $W=\Com_V(\hh)$ and $V_L=\Com_V(W)$ setwise. This shows that the action of $\Aut(V)_{\{\Delta\}}\leq\Aut(V)_{\{\hh\}}$ on $\hh\cong L\otimes_\Z\C$ must also preserve the lattice $L$ and hence $L'$ as a set so that we may view
\begin{equation*}
\Out(V)\leq\O(L)_{\{\Delta\}}\leq\O(\hh)_{\{\Delta\}}\cong\Out(V_1).
\end{equation*}

We show in \cite{HM22} that any finite-order automorphism $g$ of $V$ is conjugate in $\Aut(V)$ to an element of $\Aut(V)_{\{\Delta\}}$. This can be used to determine the conjugacy classes in $\Aut(V)$, as we shall demonstrate below for the special case of inner automorphisms.

\medskip

In what follows, we describe the outer automorphism group $\Out(V)\leq\O(L)_{\{\Delta\}}$ based on the description of $V$ as a simple-current extension of $W\otimes V_L$ given in \autoref{sec:enumvoa}.

Recall that the setwise stabiliser $\Aut(V)_{\{\hh\}}$ of the Cartan subalgebra $\hh$ also fixes $L$ setwise. Thus, it induces a subgroup $\Aut(V)_{\{\hh\}}/\Aut(V)_{\hh}\eqqcolon\O(L)_V\leq\O(L)$ where $\Aut(V)_{\hh}$ is the pointwise stabiliser. We obtain the exact sequence
\begin{equation*}
1\longrightarrow\O(L)_0\longrightarrow\O(L)_V\longrightarrow S_L\longrightarrow1
\end{equation*}
where $S_L$ is the group that was defined at the end of \autoref{sec:enumvoa}

By choosing a set of simple roots $\Delta\subseteq\Phi$ we finally obtain the short exact sequence
\begin{equation*}
1\longrightarrow(\O(L)_0)_{\{\Delta\}}\longrightarrow\Out(V)\longrightarrow(S_L)_{\{\Delta\}}\longrightarrow1
\end{equation*}
for the stabiliser $\Out(V)\leq\O(L)_{\{\Delta\}}$ of $\Delta$. This provides an explicit description of $\Out(V)$ inside $\O(L)$ in terms of $\Aut(W)$ and $\Delta\subseteq\Phi$.

In \autoref{sec:class}, in the case of central charge $24$, we shall demonstrate explicitly how to describe the outer automorphism group $\Out(V)$ using information about $\Aut(W)$ taken from \cite{BLS23b,BLS23} and the vector-valued character of $W$.


\subsection{Inner Automorphisms}\label{sec:orbinner}

We now study orbifold constructions, starting with inner automorphisms, i.e.\ those in $K=\langle\{\e^{v_0}\,|\,v\in V_1\}\rangle$. Let $V$ continue to be a nice, self-dual \voa{} satisfying $\ker(r)\subseteq T=\{\e^{v_0}\,|\,v\in\hh\}$.

By \cite{HM22}, every inner automorphism of $V$ is conjugate to an automorphism in~$T$, i.e.\ of the form
\begin{equation*}
\sigma_h\coloneqq\e^{2\pi\i h_0}
\end{equation*}
for some $h\in\hh\cong L\otimes_\Z\C$. Such an automorphism $\sigma_h$ has finite order if and only if $h\in L\otimes_\Z\Q$. More precisely, the inner automorphisms of order dividing $n$ are parametrised by $\frac{1}{n}L$. In particular, $\sigma_h=\id_V$ if and only if $h\in L$. We obtain the following classification of the conjugacy classes in $\Aut(V)$ (cf.\ \cite{HM22,MS23}), noting again that the self-duality assumption is not essential.
\begin{prop}\label{prop:classesinner}
Let $V$ be a nice, self-dual \voa{} with associated lattice $L$. Suppose that $\ker(r)\subseteq T$. Then the map $h\mapsto\sigma_h$ is a bijection from the orbits of the action of $\O(L)_V$ on $(L\otimes_\Z\Q)/L$ to the conjugacy classes of finite-order, inner automorphisms of $V$.
\end{prop}

\medskip

In the following, let $\sigma_h$ be an inner automorphism in $T$ of finite order. With regard to the decomposition of $V$ into irreducible $W\otimes V_L$-modules, $\sigma_h$ only acts non-trivially on the $V_{\alpha+L}$, $\alpha+L\in L'/L$, and the \fpvosa{} is given by
\begin{equation*}
V^{\sigma_h}=\bigoplus_{\alpha+L\in L'/L}W^{\tau(\alpha+L)}\otimes V_{(\alpha+L)^h}
\end{equation*}
where $(\alpha+L)^h\coloneqq\{\beta\in\alpha+L\,|\,\langle\beta,h\rangle\in\Z\}$. The latter are cosets in $(L^h)'/L^h$ if they are non-empty, i.e.\ the $V_{(\alpha+L)^h}$ are irreducible modules for $V_{L^h}$. Indeed:
\begin{prop}\label{prop:asslatinner}
Let $V$ be a \strat{} \voa{} with associated lattice $L$ and Heisenberg commutant $W$. Let $\sigma_h\in T$ of finite order. Then $L^h=\{\beta\in L\,|\,\langle\beta,h\rangle\in\Z\}$ is the associated lattice of $V^{\sigma_h}$, and $W\otimes V_{L^h}$ is a dual pair in $V^{\sigma_h}$, i.e.\ $W$ is also the Heisenberg commutant of $V^{\sigma_h}$.
\end{prop}
\begin{proof}
$W$ and $V_L$ form a dual pair in $V$, i.e.\ they are their mutual commutants. Equivalently, we can write the conformal vector of $V$ as $\omega=\omega^W+\omega^L$ with $\omega^W\in W$ and $\omega^L\in V_L$ so that $W=\ker_V(\omega^L_0)$ and $V_L=\ker_V(\omega^W_0)$ \cite{FZ92,LL04}.

Since $\omega\in V^{\sigma_h}$, it follows that $\ker_{V^{\sigma_h}}(\omega^L_0)=V^{\sigma_h}\cap W=W$ and $\ker_{V^{\sigma_h}}(\omega^W_0)=V^{\sigma_h}\cap V_L=V_{L^h}$. This shows that $W$ and $V_{L^h}$ form a dual pair in $V^{\sigma_h}$. Moreover, since $\sigma_h$ is inner, $\rk(L^h)=\rk(L)=\rk(V_1)=\rk(V^{\sigma_h}_1)$. This shows that $L^h$ is the associated lattice of $V$.
\end{proof}

It is straightforward to determine the weight-$1$ Lie algebra of the \fpvosa{} $V^{\sigma_h}$ (which must also be reductive) and its root system explicitly. If the Heisenberg commutant $W$ satisfies the positivity condition,
\begin{equation*}
V_1^{\sigma_h}=\hh\oplus\!\!\bigoplus_{q\in(0,1]}\bigoplus_{\substack{\alpha+L\in L'/L\\\langle\alpha,\alpha\rangle/2\in q+\Z}}\!\!\!\!W^{\tau(\alpha+L)}_{1-q}\otimes\Bigl(\!\!\bigoplus_{\substack{\beta\in(\alpha+L)^h\\\langle\beta,\beta\rangle/2=q}}\!\!\!\!\C\ee_\beta\Bigr)
\end{equation*}
with Cartan subalgebra $\hh$ and root system
\begin{equation*}
\Phi^{\sigma_h}=\bigcup_{q\in(0,1]}\bigl\{\beta\in(L')^h\,\big|\,\langle\beta,\beta\rangle/2=q,\ W^{\tau(\beta+L)}_{1-q}\neq\{0\}\bigr\}\subseteq(L')^h.
\end{equation*}

\medskip

The unique irreducible $\sigma_h$-twisted $V$-module \cite{DLM00} is given by
\begin{equation*}
V(\sigma_h)=V^{(h)}=\!\!\!\bigoplus_{\alpha+L\in L'/L}\!\!\!W^{\tau(\alpha+L)}\otimes V_{\alpha+L}^{(h)}=\!\!\!\bigoplus_{\alpha+L\in L'/L}\!\!\!W^{\tau(\alpha+L)}\otimes V_{\alpha+h+L}
\end{equation*}
(see \cite{Li96} for details), where we may view $V_{\alpha+L}^{(h)}=V_{\alpha+h+L}$ as a (not necessarily irreducible) module for $V_{L^h}\subseteq V^{\sigma_h}$. If $W$ satisfies the positivity condition, then the (possibly zero) weight-$1$ space $V(\sigma_h)_1$ of the $\sigma_h$-twisted module $V(\sigma_h)$ is
\begin{equation*}
V(\sigma_h)_1=\delta_{h\in L'}\delta_{\langle h,h\rangle/2\in\Z}W^{\tau(-h+L)}_1\!\otimes\C\vac\oplus\!\!\bigoplus_{q\in(0,1]}\!\!\!\!\bigoplus_{\substack{\,\,\,\,\alpha+L\in L'/L\\\,\,\,\,\langle\alpha,\alpha\rangle/2\in q+\Z}}\!\!\!\!\!\!\!\!\!\!W^{\tau(\alpha+L)}_{1-q}\!\otimes\!\Bigl(\!\!\!\!\bigoplus_{\substack{\,\,\,\,\beta\in\alpha+h+L\\\,\,\,\,\langle\beta,\beta\rangle/2=q}}\!\!\!\!\!\!\!\!\C\ee_\beta\Bigr).
\end{equation*}
Note that $V(\sigma_h)_1$ is naturally a module for the Lie algebra $V_1^{\sigma_h}$ and the centraliser (i.e.\ the kernel) of the Cartan subalgebra $\hh$ under the module action is
\begin{equation*}
C_{V(\sigma_h)_1}(\hh)=\delta_{h\in L'}\delta_{\langle h,h\rangle/2\in\Z}W^{\tau(-h+L)}_1\otimes\C\vac.
\end{equation*}
The last equation holds even if $W$ does not satisfy the positivity condition, as long as we assume that $V(\sigma_h)$ has positive $L_0$-grading, which implies the same statement for $W^{\tau(-h+L)}$ (cf.\ the proof of \autoref{lem:csaIII}).

\medskip

We now specialise to the case that $\sigma_h$ has order $2$, i.e.\ $h\in\frac{1}{2}L\setminus L$. We assume that $V(\sigma_h)$ has positive $L_0$-grading or equivalently that $V^{\sigma_h}$ satisfies the positivity condition. This is true for all inner automorphisms if $c=0$, $8$, $16$ or $24$ and $V_1\neq\{0\}$ since then $W$ satisfies the positivity condition \cite{Hoe17} (see \autoref{sec:24}).

The $L_0$-grading of $V(\sigma_h)$ takes values in $\frac{1}{2}\Z+\langle h,h\rangle/2$. Hence, $\sigma_h$ has type~$0$ and thus affords an orbifold construction if and only if $\langle h,h\rangle/2\in\frac{1}{2}\Z$. We also assume this in the following. Since the order is prime, the orbifold construction has the simple form
\begin{equation*}
V^{\orb(\sigma_h)}=V^{\sigma_h}\oplus V(\sigma_h)_\Z.
\end{equation*}
The \fpvosa{} $U=V^{\sigma_h}$ satisfies the conditions of \autoref{sec:evensub}, i.e.\ it is a nice \voa{} with $\Rep(U)\cong\mathcal{C}(2_{\II}^{+2})$. The two self-dual, bosonic $\Z_2$-extensions of $V^{\sigma_h}$ are $V$ and $V^{\orb(\sigma_h)}$.

One important datum is the weight-$1$ Lie algebra
\begin{equation*}
V_1^{\orb(\sigma_h)}=V_1^{\sigma_h}\oplus V(\sigma_h)_1
\end{equation*}
(recall that for central charge $c=0$, $8$, $16$ or $24$ and $V_1^{\orb(\sigma_h)}\neq\{0\}$ this fixes the \voa{} $V^{\orb(\sigma_h)}$ uniquely up to isomorphism), which is a $\Z_2$-graded Lie algebra with zero component $V_1^{\sigma_h}$. The centraliser of $\hh$ in $V^{\orb(\sigma_h)}_1$,
\begin{equation*}
C_{V^{\orb(\sigma_h)}_1}(\hh)=\hh\oplus\delta_{h\in L'}\delta_{\langle h,h\rangle/2\in\Z}W^{\tau(-h+L)}_1
\end{equation*}
is a Cartan subalgebra of $V^{\orb(\sigma_h)}_1$ \cite{Kac90}.

If this centraliser is just $\hh$, i.e.\ if the Cartan subalgebra $\hh$ of $V_1$ and of $V^{\sigma_h}_1$ is also a Cartan subalgebra of $V^{\orb(\sigma_h)}_1$ (see types I and IIa below), then we can directly determine the associated lattice and the weight-$1$ Lie algebra of the orbifold construction $V^{\orb(\sigma_h)}$. Indeed, the associated lattice of $V^{\orb(\sigma_h)}$ is
\begin{equation*}
L^h\cup(h+L)_\text{ev},
\end{equation*}
and if $W$ satisfies the positivity condition, the root system of $V^{\orb(\sigma_h)}_1$ is
\begin{equation*}
\Phi^{\orb(\sigma_h)}=\Phi^{\sigma_h}\cup\bigcup_{q\in(0,1]}\bigl\{\beta\in h+L'\,\big|\,\langle\beta,\beta\rangle/2=q,\ W^{\tau(\beta-h+L)}_{1-1/d}\neq\{0\}\bigr\},
\end{equation*}
a subset of $(L')^h\cup(h+L')$.

On the other hand, if the rank of the associated lattice and the weight-$1$ Lie algebra increases, i.e.\ if the centraliser is strictly larger than $\hh$ (see type~IIb below), then we cannot compute these data directly. Rather, one can proceed by elimination (see \autoref{sec:orbcomp} for details in the case of central charge $24$).

\medskip

We make a case distinction. By definition, $(L^h)'=\langle L',h\Z\rangle$, which is an extension of $L'$ of index $1$ or $2$ since $h\in\frac{1}{2}L'$. We distinguish (cf.\ \autoref{sec:glueing}):
\begin{enumerate}
\item[(I)]$L\neq L^h$,
\item[(IIa)]$L=L^h$ and $\langle h,h\rangle/2\in\Z+\sfrac{1}{2}$,
\item[(IIb)]$L=L^h$ and $\langle h,h\rangle/2\in\Z$.
\end{enumerate}


\subsection*{Type I}

For this case we assume that $L\neq L^h$, which is equivalent to $h\notin L'$ and to $[L:L^h]=[(L^h)':L']=2$.

Then, since $L$ contains elements $\beta$ such that $\langle\beta,h\rangle\in\Z$ as well as elements with $\langle\beta,h\rangle\in\Z+\sfrac{1}{2}$, all $(\alpha+L)^h\neq\emptyset$ for $\alpha\in L'/L$ and they are cosets in $(L^h)'/L^h$. Hence, $V^{\sigma_h}$ is the simple-current extension
\begin{equation*}
V^{\sigma_h}=\bigoplus_{\alpha+L\in L'/L}W^{\tau(\alpha+L)}\otimes V_{(\alpha+L)^h}\supseteq W\otimes V_{L^h}
\end{equation*}
with associated lattice $L^h$. The tensor product $W\otimes V_{L^h}$ is a dual pair in $V^{\sigma_h}$ with $4\,|R(W)|=|(L^h)'/L^h|$, and so this case corresponds to glueing type~I in \autoref{sec:glueing}.

Since $h\notin L'$, the Cartan subalgebra $\hh$ of $V$ and $V^{\sigma_h}$ is also a Cartan subalgebra of $V^{\orb(\sigma_h)}$, i.e.\ $\rk(V^{\orb(\sigma_h)}_1)=\rk(V^{\sigma_h}_1)=\rk(V_1)$. In fact, we argued that the inverse-orbifold automorphism on $V^{\orb(\sigma_h)}$ is inner. The associated lattice of $V^{\orb(\sigma_h)}$ is $L^h\cup(h+L)_\text{ev}$, which is an index-$2$ extension of $L^h$.


\subsection*{Type II}

We assume that $L=L^h$, which is the case if and only if $h\in L'$ and if and only if $[L:L^h]=[(L^h)':L']=1$.

Then, since $\langle\beta,h\rangle\in\Z$ for all $\beta\in L$,
\begin{equation*}
(\alpha+L)^h=\begin{cases}\alpha+L&\text{if }\langle\alpha,h\rangle\in\Z,\\\emptyset&\text{if }\langle\alpha,h\rangle\notin\Z.\end{cases}
\end{equation*}
Hence, $V^{\sigma_h}$ is the simple-current extension
\begin{equation*}
V^{\sigma_h}=\bigoplus_{\substack{\alpha+L\in L'/L\\\langle\alpha,h\rangle\in\Z}}W^{\tau(\alpha+L)}\otimes V_{\alpha+L}\supseteq W\otimes V_L
\end{equation*}
with associated lattice $L^h=L$ and dual pair $W\otimes V_L$ in $V^{\sigma_h}$ with $|R(W)|=|L'/L|$. Therefore, this type corresponds to glueing type~II in \autoref{sec:glueing}.

\subsubsection*{Type IIa}
We assume that $\langle h,h\rangle/2\in\Z+\sfrac{1}{2}$, which corresponds to glueing type~IIa in \autoref{sec:glueing}. Since $\langle h,h\rangle/2\notin\Z$, the Cartan subalgebra $\hh$ of $V$ and of $V^{\sigma_h}$ is also a Cartan subalgebra of $V^{\orb(\sigma_h)}$, i.e.\ $\rk(V^{\orb(\sigma_h)}_1)=\rk(V^{\sigma_h}_1)=\rk(V_1)$. In fact, we argued that the inverse-orbifold automorphism on $V^{\orb(\sigma_h)}$ is inner.

Also the orbifold construction $V^{\orb(\sigma_h)}$ contains the dual pair $W\otimes V_{L}$ with associated lattice $L^h\cup(h+L)_\text{ev}=L$.

\subsubsection*{Type IIb}
We assume that $\langle h,h\rangle/2\in\Z$. This corresponds to glueing type~IIb in \autoref{sec:glueing}. Since $\langle h,h\rangle/2\in\Z$, the $\sigma_h$-twisted module typically has a non-trivial contribution to the Cartan subalgebra of $V^{\orb(\sigma_h)}$. This means that typically
\begin{equation*}
\rk(V^{\orb(\sigma_h)}_1)>\rk(V^{\sigma_h}_1)=\rk(V_1).
\end{equation*}
The orbifold construction $V^{\orb(\sigma_h)}$ contains the dual pair $\tilde{W}\otimes V_{\tilde{L}}$ with associated lattice $\tilde{L}$ where typically $\rk(\tilde{L})>\rk(L)$ and the Heisenberg commutant $\tilde{W}$ is a non-full subalgebra of $W$.

\medskip

We have described all orbifold constructions of order~$2$ for which the automorphism (or the inverse-orbifold automorphism) is inner. It remains to study those orbifold constructions where both are non-inner. By exhaustion (see \autoref{table:autcases}), they correspond to glueing type~III in \autoref{sec:glueing}.


\subsection{Non-Inner Automorphisms}\label{sec:orbouter}

We study orbifold constructions under non-inner automorphisms, corresponding to glueing types IIb or III in \autoref{sec:glueing}.

We still assume that $V$ is a nice, self-dual \voa{} satisfying $\ker(r)\subseteq T$ and let $g$ be a finite-order automorphism of $V$. It is shown in \cite{HM22} that any such $g$ is conjugate in $\Aut(V)$ to an automorphism in $\Aut(V)_{\{\Delta\}}$, the setwise stabiliser of the choice of simple roots $\Delta$. Hence, without loss of generality, we let $g\in\Aut(V)_{\{\Delta\}}\leq\Aut(V)_{\{\hh\}}$. As explained before, this implies that $g$ fixes each component of the dual pair $W\otimes V_L\subseteq V$ setwise.

It is rather difficult to describe the precise action of $g$ on $V$ if $g$ is non-inner, so we shall content ourselves with only studying the restriction $g|_{V_L}$ to the lattice \voa{} $V_L$. This will still determine the associated lattice of the \fpvosa{} $V^g$ and give us at least partial information about the root system of $V^g_1$.

\medskip

Automorphisms of lattice \voa{}s are well understood \cite{DN99,HM22}. As $g|_{V_L}$ fixes $\Delta$ (and $\hh$) setwise, the action of $g|_{V_L}$ on $\hh\cong L\otimes_\Z\C$ defines a lattice automorphism $\nu\in\O(L)_{\{\Delta\}}$, which coincides with the projection of $g$ to $\Out(V)=\Aut(V)/K$, viewed as a subgroup of $\O(L)_{\{\Delta\}}$.
\begin{prop}\label{prop:autconj}
Let $V$ be a nice, self-dual \voa{} with associated lattice $L$. Suppose that $\ker(r)\subseteq T$. Any finite-order automorphism of $V$ is conjugate to an automorphism $g\in\Aut(V)_{\{\Delta\}}$ whose restriction to $V_L$ is
\begin{equation*}
g|_{V_L}=\hat\nu\sigma_h
\end{equation*}
where $\nu\in\Out(V)\leq\O(L)_{\{\Delta\}}$ is the projection of $g$ to $\Out(V)$, $\hat\nu$ is a standard lift of $\nu\in\O(L)$ to $\O(\hat{L})\leq\Aut(V_L)$ and $h\in(L\otimes_\Z\Q)^\nu/\pi_\nu(L')$.
\end{prop}
Here, $\pi_\nu\colon(L\otimes_\Z\C)\to(L\otimes_\Z\C)^\nu$ denotes the projection map $\pi_\nu=\frac{1}{|\nu|}\sum_{i=0}^{|\nu|-1}\nu^i$. It is possible to describe a system of representatives for the conjugacy classes of automorphisms of $V_L$ (see \cite{HM22}, a special case is treated in \autoref{prop:classesstandard}), however one $g|_{V_L}$ may have several non-conjugate extensions to $V$ so that we forgo a precise enumeration (cf.\ \autoref{prop:classesinner}) of non-inner automorphisms of $V$.

\medskip

In the following, let $g$ be a finite-order automorphism of $V$ of the form described in \autoref{prop:autconj}, restricting in particular to $g|_{V_L}=\hat\nu\sigma_h$.
\begin{prop}\label{prop:asslatouter}
Let $V$ be a nice, self-dual \voa{} with associated lattice $L$ and Heisenberg commutant $W$. Let $g\in\Aut(V)$ as in \autoref{prop:autconj}. Then $L^{\nu,h}=\{\alpha\in L^\nu\,|\,\langle\alpha,h\rangle\in\Z\}$ is the associated lattice of $V^g$, the Heisenberg commutant $\Com_{V^g}(V_{L^{\nu,h}})$ of $V^g$ contains $W$ and is independent of $h$, and together they form the dual pair $\Com_{V^g}(V_{L^{\nu,h}})\otimes V_{L^{\nu,h}}$ in $V^g$.
\end{prop}
\begin{proof}
Again, we write the conformal vector of $V$ as $\omega=\omega^W+\omega^L$ with $\omega^W\in W$ and $\omega^L\in V_L$ and $W=\ker_V(\omega^L_0)$ and $V_L=\ker_V(\omega^W_0)$. Then, $W^g=V^g\cap\ker_V(\omega^L_0)=\ker_{V^g}(\omega^L_0)=\Com_{V^g}(V_L^g)$ and similarly $V_L^g=\Com_{V^g}(W^g)$. Hence, $W^g\otimes V_L^g$ forms a dual pair in $V^g$.

It follows from the definition of the conformal vector for a lattice \voa{} that we can write $\omega^L\in\langle\hh\rangle\subseteq V_L$ as $\omega^L=\omega^{\hh^\nu}+\omega^{(\hh^\nu)^\bot}$ with $\omega^{\hh^\nu}\in\langle\hh^\nu\rangle$ and $\omega^{(\hh^\nu)^\bot}\in\langle(\hh^\nu)^\bot\rangle$, noting that $\langle\hh\rangle\cong\langle(\hh^\nu)^\bot\rangle\otimes\langle\hh^\nu\rangle$. Then $\omega=\omega^W+\omega^{\hh^\nu}+\omega^{(\hh^\nu)^\bot}$. This shows that also $\ker_{V^g}((\omega^W+\omega^{(\hh^\nu)^\bot})_0)$ and $\ker_{V^g}(\omega^{\hh^\nu}_0)$ form a dual pair in $V^g$.

Then, we note that $\Com_V(\Com_V(\hh^\nu))=\ker_V((\omega^W+\omega^{(\hh^\nu)^\bot})_0)=V_{L^\nu}$. Indeed, $\Com_V(\Com_V(\hh^\nu))$ must be a lattice \voa{} contained in $\Com_V(\Com_V(\hh))=V_L$, and it contains $V_{L^\nu}$. But $L^\nu$ is already the largest same-rank extension of $L^\nu$ inside $L$. Hence $\Com_V(\Com_V(\hh^\nu))=V_{L^\nu}$.

Finally, we show that $\Com_{V^g}(\Com_{V^g}(\hh^\nu))=\ker_{V^g}((\omega^W+\omega^{(\hh^\nu)^\bot})_0)=V_{L^{\nu,h}}$. Indeed, we write $\ker_{V^g}((\omega^W+\omega^{(\hh^\nu)^\bot})_0)=V^g\cap\ker_V((\omega^W+\omega^{(\hh^\nu)^\bot})_0)=V^g\cap V_{L^\nu}=V_{L^{\nu,h}}$. Since $\hh^g=\hh^\nu$ is a Cartan subalgebra of $V^g$, this implies that $L^{\nu,h}$ is the associated lattice of $V^g$.

Consequently, the Heisenberg commutant of $V^g$ is $\Com_{V^g}(\hh^\nu)=\ker_{V^g}(\omega^{\hh^\nu}_0)$, which we write as $V^g\cap\ker_{V}(\omega^{\hh^\nu}_0)$. Multiplying $g$ and equivalently $g|_{V_L}$ by an inner automorphism $\sigma_{h'}$ with $h'\in(L\otimes_\Z\Q)^\nu\subseteq\hh^\nu$ does not change the action of $g$ on $\ker_{V}(\omega^{\hh^\nu}_0)$. This shows that the Heisenberg commutant is independent of $h$.
\end{proof}

In contrast to the case of inner automorphisms, we do not fully determine the reductive weight-$1$ Lie algebra of the \fpvosa{} $V^g$ because we only study the restriction of $g$ to $V_L\subseteq V$. We can, however, determine the subset of roots of $V_1^g$ that descend from the norm-$1$ roots of $V$.

If $W$ satisfies the positivity condition, the norm-$1$ roots of $V_1$ are simply
\begin{equation*}
\Phi_1=\bigl\{\beta\in L\,\big|\,\langle\beta,\beta\rangle/2=1\bigr\}\subseteq L
\end{equation*}
with corresponding $1$-dimensional root spaces $\C\vac\otimes\C\ee_\beta\in W\otimes V_L$. This is nothing but the simply-laced root system of the lattice \voa{} $V_L$.

Then, omitting the tensor factor $\vac\in W$, on which $g$ acts trivially, the root spaces of $V_L^g$ are spanned by $\sum_{i=0}^{|g|-1}g^i(\ee_\beta)$, $\beta\in\Phi_1$, so that the roots of $V_L^g$ are
\begin{equation*}
\Phi_1^g=\Bigl\{\pi_\nu(\beta)\,\Big|\,\beta\in\Phi_1,\ \sum_{i=0}^{|g|-1}g^i(\ee_\beta)\neq0\Bigr\}\subseteq\pi_\nu(L),
\end{equation*}
recalling that $\hh^\nu$ is a Cartan subalgebra of $V_L^g$. This is a subset of the roots of $V^g$.

\medskip

Again, we would like to specialise to the case that $g\in\Aut(V)$ has order~$2$ and type~$0$ and that the unique irreducible $g$-twisted $V$-module $V(g)$ has positive $L_0$-grading. Then the orbifold construction is
\begin{equation*}
V^{\orb(g)}=V^g\oplus V(g)_\Z,
\end{equation*}
which, like $V$, is a nice, self-dual \voa{}. But note that neither the order of $g$ nor the value of the lowest $L_0$-weight of $V(g)$ are directly accessible from the description above, in contrast to inner automorphisms. A necessary condition is that $g|_{V_L}=\hat\nu\sigma_h$ has order dividing~$2$, which is equivalent to the order of $\nu$ dividing~$2$ and $h\in\frac{1}{2}(L')^\nu$ (with a slight modification if the lattice automorphism~$\nu$ exhibits order doubling, see \cite{HM22}).

We assume in the following that $g$ is non-inner, i.e.\ that $\nu\in\Out(V)\leq\O(L)_{\{\Delta\}}$ is non-trivial. We distinguish two cases:
\begin{enumerate}
\item[(IIb)]$L^\nu\neq L^{\nu,h}$,
\item[(III)]$L^\nu=L^{\nu,h}$.
\end{enumerate}
They correspond to the glueing types of the same names in \autoref{sec:glueing}.

As the inverse-orbifold automorphism is inner for type~IIb, it is easier to use the description in \autoref{sec:orbinner} for these orbifold constructions so that we only need to consider non-inner automorphisms of type~III (see \autoref{sec:orbcomp}).

\medskip

In \autoref{table:autcases}, we summarise the different types in this section and in \autoref{sec:glueing}. Here, $K=L^{\nu,h}$ is the associated lattice of $U=V^g$, and $W$, in contrast to the rest of this section, denotes the Heisenberg commutant of $V^g$ (rather than that of $V$).

\begin{table}[ht]\caption{Different glueing types of nice, self-dual \svoa{}s of central charge $c\in8\Ns$.}
\begin{tabular}{l|l|l|l|l|l}
Type & Orb. & Inv.\ orb. & \multicolumn{2}{l|}{Conditions on automorphisms} & Dual pair in $V^g$\\\hline\hline
I   & \multicolumn{2}{l|}{\multirow{2}{*}{both inner}} & \multicolumn{2}{l|}{$L^h\neq L$}                             & $|A_K|=4\,|A_W|$\\\cline{1-1}\cline{4-6}
IIa & \multicolumn{2}{l|}{}                            & \multicolumn{2}{l|}{$L^h=L$, $\langle h,h\rangle/2\notin\Z$} & \multirow{2}{*}{$|A_K|=|A_W|$} \\\cline{1-5}
IIb & inner & non-inner                                & $L^h=L$, $\langle h,h\rangle/2\in\Z$ & $L^{\nu,h}\neq L^\nu$ & \\\hline
III & \multicolumn{2}{l|}{both non-inner}              & \multicolumn{2}{l|}{$L^{\nu,h}=L^\nu$}                       & $4\,|A_K|=|A_W|$
\end{tabular}
\label{table:autcases}
\end{table}


\section{Classification}\label{sec:class}

In this section we finally classify the nice, self-dual \svoa{}s of central charge~$24$ by applying the two main tools developed in \autoref{sec:orbifold} (together with \autoref{sec:holvoa}) and in \autoref{sec:evensub}. We proceed in several steps.


\subsection{Step 1: Orbifold Constructions}\label{sec:orbcomp}

The starting point is the classification of the nice, self-dual \voa{}s $V$ of central charge~$24$ described in \autoref{sec:holvoa}. We assume that either $V_1\neq\{0\}$ or that $V$ is isomorphic to the moonshine module $V\cong V^\natural$. That is to say, we ignore hypothetical fake copies of $V^\natural$ (but see the discussion in \autoref{rem:fakeshort}). The results of \cite{BLS23} show that the technical assumption $\ker(r)\subseteq T$ made on $V$ in \autoref{sec:orbifold} is satisfied if $V_1\neq\{0\}$.

We then determine the $2$-neighbourhood graph with these \voa{}s as nodes (see \autoref{sec:graphmeth}) by considering all possible $\Z_2$-orbifold constructions (see \autoref{sec:orbifold}). For edges of types I, IIa or IIb we shall obtain the exact answer, while for those of type~III we shall initially only obtain an upper bound.

\subsubsection*{Automorphism Group}
Let $V$ be a nice, self-dual \voa{} of central charge~$24$ with associated lattice $L$ and Heisenberg commutant $W$. Following \autoref{sec:outv}, we realise the outer automorphism group $\Out(V)\leq\O(L)_{\{\Delta\}}$ concretely by using the description of $V$ as a simple-current extension of $W\otimes V_L$ in \autoref{sec:enumvoa}, information about $\Aut(W)$ taken from \cite{BLS23b,BLS23} and the vector-valued character of $W$ given in \autoref{sec:lie24}.

Suppose further that $V_1\neq\{0\}$. We saw in \autoref{prop:commconway} that the Heisenberg commutant is of the form $W\cong V_{\Lambda_\mu}^{\hat\mu}$ for $\mu\in\Co_0$ from one of $11$ conjugacy classes labelled A to K.

We determine the root system $\Phi$ of $V_1$ as described in \autoref{sec:lie24} by using the vector-valued characters of $V_L$ and $W$. For the lattice \voa{}~$V_L$ the vector-valued character simply consists of the theta series $\vartheta_{\alpha+L}(\tau)$ for the lattice cosets in $L'/L$ divided by the eta function $\eta(\tau)^{\rk(L)}$. The vector-valued character of $W\cong V_{\Lambda_\mu}^{\hat\mu}$ can be computed explicitly from the twisted characters of the lattice \voa{} $V_{\Lambda_\mu}$. This is described in \cite{Moe16,Moe21} for types B, C, F, G and H, where one obtains certain lifts of eta products, and poses no fundamentally greater difficulty for the other types. For example, for types D, E and I one obtains vector-valued modular forms described in \cite{HS14,Sch09}, with similar expressions for J and K, again as lifts of eta products (see also \autoref{sec:lie24}).

The automorphism groups of the lattice \voa{}s $V_L$ are understood \cite{DN99,HM22}. The automorphism groups $\Aut(W)$ for types B to K are described in \cite{BLS23b,BLS23}, but we need explicit realisations of $\overline{\Aut}(W)\leq\O(A_W)$ up to conjugacy, which we determine by exclusion. To this end, we first observe that $\overline{\Aut}(W)$ keeps the vector-valued character $\chi$ of $W$ invariant, i.e.\ must be in a subgroup $\O(A_W)_{\chi}\leq\O(A_W)$. We then compute all conjugacy classes of subgroups $H\leq\O(A_W)_\chi$ of the correct order. (There are usually only a few.) This allows us, for a choice of lattice $L$ in the correct genus and anti-isometry $\tau$, to compute the group $\O(L)_V$ and the roots $\Phi$ and thus the Weyl group of $V_1$ inside $\O(L)_V$. Then, we impose the condition on $H$ that for all choices of $L$ and $\tau$ the Weyl group must be contained in $\O(L)_V$. This way, we are able to determine for all cases a unique conjugacy class $H$ and hence the groups $\O(L)_V$ and $\Out(V)$ concretely.

\subsubsection*{Inner Automorphisms}
We begin by studying $\Z_2$-orbifold constructions under inner automorphisms, as described in \autoref{sec:orbinner}. They correspond to types I, IIa and IIb. We may assume that $V_1\neq\{0\}$ since otherwise $V$ would not afford any inner automorphisms. Furthermore, by \cite{HM22}, all \fpvosa{}s under inner automorphisms satisfy the positivity condition.

Using \autoref{prop:classesinner}, we enumerate the conjugacy classes of inner automorphisms $g$ of order~$2$ and type~$0$. Via $h\mapsto\sigma_h$ they are in bijection with the orbits $h$ of the action of
\begin{equation*}
\O(L)_V\quad
\text{on}
\quad\bigl\{h\in\tfrac{1}{2}L\setminus L\,\big|\,\langle h,h\rangle/2\in\tfrac{1}{2}\Z\bigr\}/L.
\end{equation*}
By \autoref{prop:order2conj}, this then enumerates the isomorphism classes of nice, self-dual \svoa{}s of central charge~$24$ of types I, IIa or IIb.

Recall that when counting \svoa{}s corresponding to loops in the neighbourhood graph we must distinguish between two cases, depending on whether there are two non-conjugate automorphisms $g$ with the same \fpvosa{} $V^g$ or just one. By simply looking at invariants of the \fpvosa{} $V^g$ like its weight-$1$ Lie algebra $V^g_1$, we see that the latter is the case for all but one pair of automorphisms. For the remaining pair of non-conjugate automorphisms (of the self-dual \voa{} $V$ with $V_1\cong A_{1,1}^2D_{6,5}$ such that $V^g_1\cong A_{1,1}A_{5,5}\C^2$, of type~I) we need to determine if their \fpvosa{}s $V^g$ are isomorphic. We shall see in \autoref{sec:extcomp} that this is indeed the case. Hence, they only correspond to one self-dual \svoa{} (of type~(2) in \autoref{prop:order2conj}).

\medskip

In order to list the self-dual \svoa{}s (see \autoref{sec:results}), we determine the isomorphism type (i.e.\ the root system) of the weight-$1$ Lie algebras (together with the affine structure, see \autoref{sec:cartan}) and the associated lattices of $V^g$ and $V^{\orb(g)}$, as is explained in \autoref{sec:orbinner}. These can be computed directly for types I and IIa because the Cartan subalgebra $\hh$ of $V$ and of $V^g$ is also a Cartan subalgebra of $V^{\orb(g)}$.

For type~IIb we can still compute $V^g_1$ directly, but we have to determine the type of the Lie algebra $V_1^{\orb(g)}$ by exclusion. It must be one of the $70$ non-zero Lie algebras on Schellekens' list and a (possibly trivial) $\Z_2$-extension of $V^g_1$. We also know its dimension and its rank. The latter is the dimension of the centraliser of $\hh$ in $V^{\orb(g)}_1$, which is a Cartan subalgebra of $V^g_1$. The possible fixed-point Lie subalgebras of simple\footnote{The extension of this result to semisimple or reductive Lie algebras is straightforward (see, e.g., \cite{Moe16}, where the semisimple case is discussed).} Lie algebras under finite-order automorphisms are classified in \cite{Kac90}. This exclusion method is particularly effective for small orders, and it fixes the type of $V_1^{\orb(g)}$ in all cases. As the nice, self-dual \voa{}s of central charge $24$ are uniquely determined by their weight-$1$ Lie algebras, provided these are non-zero, this also fixes the \voa{} $V^{\orb(g)}$ up to isomorphism (and hence all other data like the associated lattice).

Finally, we also list the Heisenberg commutant $W$ of $V^g$. As explained in \autoref{sec:evensub} and \autoref{sec:orbifold}, for all inner automorphisms, i.e.\ for types I, IIa and IIb, it coincides with the Heisenberg commutant of the original \voa{}~$V$. In particular, the Heisenberg commutants are of the form $W\cong V_{\Lambda_\mu}^{\hat\mu}$ where $\mu\in\O(\Lambda)$ is from one of the $11$ conjugacy classes listed in \autoref{table:11} (see \autoref{thm:commconway} below).

\subsubsection*{Non-Inner Automorphisms}
We continue by studying $\Z_2$-orbifold constructions under non-inner automorphisms (see \autoref{sec:orbouter}). As type~IIb was already treated using inner automorphisms, we shall restrict to type~III.

There are two (outer) automorphisms of the moonshine module $V^\natural$ of order~$2$, namely the conjugacy classes 2A and 2B in the monster group $M=\Aut(V^\natural)$ \cite{Gri82,CCNPW85}. Both have type~$0$ and the \fpvosa{} satisfies the positivity condition \cite{CN79}. The corresponding $\Z_2$-orbifold constructions
\begin{equation*}
V^\natural\hookleftarrow(V^\natural)^\mathrm{2A}\hookrightarrow V^\natural\quad\text{and}\quad V^\natural\hookleftarrow(V^\natural)^\mathrm{2B}\cong V_\Lambda^+\hookrightarrow V_\Lambda
\end{equation*}
are given in \cite{Yam05} and \cite{FLM88}, respectively, and have type~III. They correspond to the tensor product $\VB\otimes F$ of the shorter moonshine module \cite{Hoe95} with a free fermion and to the odd moonshine module $\VO$~\cite{DGH88,Hua96b}, respectively.

We list these two edges in \autoref{sec:results}. These are the only non-typical edges in central charge $24$, i.e.\ edges for which the automorphism is outer, but the ranks of the Lie algebras of $V$ and of $V^g$ are the same (cf.\ \autoref{sec:glueing}).

In the following, we assume that $V$ is not isomorphic to $V^\natural$, and, since we are ignoring hypothetical fake copies of $V^\natural$, we can assume that $V_1\neq\{0\}$.

Up to conjugacy, any non-inner automorphism $g$ of $V$ of order~$2$ restricts to a non-inner automorphism $g|_{V_L}$ of $V_L$ of order $2$, and we shall only enumerate the latter. Following \autoref{prop:autconj}, the non-inner automorphisms $g|_{V_L}$ of $V_L$ of order~$2$ are parametrised by elements $\nu\in\Out(V)\leq\O(L)_{\{\Delta\}}$ of order~$2$ and elements
\begin{equation*}
h\in\bigl(\tfrac{1}{2}(L')^\nu\bigr)/\bigl(\pi_\nu(L')\cap\tfrac{1}{2}(L')^\nu\bigr),
\end{equation*}
with a shift if the lattice automorphism $\nu$ exhibits order doubling (see \cite{HM22}). As we are restricting to type~III, we further assume that $L^\nu=L^{\nu,h}$ or equivalently that $h\in\pi_\nu(L')$. Hence, we only consider automorphisms of $V_L$ of the form
\begin{equation*}
g|_{V_L}=\hat\nu.
\end{equation*}
In particular, only $\nu\in\O(L)_{\{\Delta\}}$ of order~$2$ that do not exhibit order doubling on $V_L$ contribute to type~III. By \cite{HM22} (see also \cite{EMS20a}) we can enumerate automorphisms of $V_L$:
\begin{prop}\label{prop:classesstandard}
The conjugacy classes of standard lifts $\hat\nu$ in $\Aut(V_L)$ of order~2 descending from $\Aut(V)$ are in bijection with the conjugacy classes of order~2 in $\Out(V)\leq\O(L)_{\{\Delta\}}$ that do not exhibit order doubling on $V_L$.
\end{prop}

As a word of caution, we point out that a given automorphism $\hat\nu$ of $V_L$ can have several non-conjugate extensions to an automorphism $g$ of $V$. Moreover, the extension could have order larger than~$2$, and the type of the extension could be non-zero, depending on the value of the $L_0$-weights of $V(g)$ modulo~$1$. In other words, a given automorphism $\hat\nu$ of $V_L$ could have no, one or several extensions to an automorphism $g$ of $V$ admitting a $\Z_2$-orbifold construction.

\medskip

Then, for each automorphism $\hat\nu$ from \autoref{prop:classesstandard} we consider all hypothetical orbifold constructions $V^{\orb(g)}$ of type~III, with constraints coming from the fact that the automorphism $g$ restricts to $\hat\nu$ on $V_L$. We are mainly interested in the Lie algebra structure of $V_1^g$ and $V_1^{\orb(g)}$, both of which we cannot compute directly. Instead, we determine a list of possible candidates for $V_1^g$ such that:
\begin{enumerate}
\item $V_1^g$ is the fixed-point Lie subalgebra of $V_1$ under an automorphism of order~$2$
\item projecting to $\nu\in\Out(V)\leq\Out(V_1)$.
\item The root system of $V_1^g$ contains $\Phi_1^g$, the latter only depending on $\hat\nu$.
\item The associated lattice $L^\nu$ of $V^g$ must be an extension of $Q_\g$ where $\g$ is the semisimple part of $V_1^g$ (see \autoref{sec:cartan}).
\end{enumerate}
This usually leaves a small (and maybe empty) list of possible $V_1^g$. In the next step, for each potential $V_1^g$, we determine a list of candidates for $V_1^{\orb(g)}$ or equivalently for $V^{\orb(g)}$ such that:
\begin{enumerate}
\setcounter{enumi}{4}
\item $V_1^{\orb(g)}$ is one of the $70$ non-zero\footnote{When $V_1^{\orb(g)}=\{0\}$, $V$ must be the Leech lattice \voa{} $V_\Lambda$, $g$ the lift of the $(-1)$-involution on $\Lambda$ and $V^{\orb(g)}$ the moonshine module $V^\natural$. The inverse-orbifold automorphism is the 2B-involution in the monster group. This gives the odd moonshine module $\VO$, which we treated above.} Lie algebras on Schellekens' list.
\item $V_1^g$ is the fixed-point Lie subalgebra of $V_1^{\orb(g)}$ under an automorphism of order~$2$
\item projecting to a non-trivial $\xi\in\Out(V^{\orb(g)})\leq\Out(V_1^{\orb(g)})$ that does not exhibit order doubling.
\item $L^\nu\cong\tilde{L}^\xi$ where $\tilde{L}$ is the associated lattice of $V^{\orb(g)}$.
\item Item (4) applied to the automorphism $\hat\xi$ of $V_{\tilde{L}}$ holds.
\item $l=(24+3\dim(V_1^g)-\dim(V_1)-\dim(V_1^{\orb(g)}))/24\in\N$, as this equals $\dim(V(g)_{1/2})$ (see \autoref{sec:chars}).
\item $V\cong V^{\orb(g)}$ if $l>0$ by \autoref{prop:suffloop}.
\item $V^g_1$ must contain the Lie algebra $(F^{\otimes l})_1\cong\so_l$ (with special cases for $l\leq2$) as an ideal (see \autoref{sec:split}).
\end{enumerate}
This usually leaves a small list of potential candidates for $V_1^{\orb(g)}$. The list could again be empty, which would imply that the considered automorphism $\hat\nu$ does not extend to an automorphism of $V$ of order~$2$ and type~$0$ or that the considered fixed-point Lie subalgebra $V_1^g$, if there were several candidates, was incorrect.

We then list all the potential edges obtained in this way. As we did not enumerate the automorphisms $g\in\Aut(V)$ up to conjugacy, but only their restrictions $\hat\nu\in V_L$, we simply obtain a list of possible combinations of $V_1$, $V^g_1$, $V^{\orb(g)}_1$ and the conjugacy classes of outer automorphisms $\nu\in\Out(V)$ and $\xi\in\Out(V^{\orb(g)})$. It is in principle possible that one such quintuple is realised by more than one edge, i.e.\ by several non-isomorphic $V^g$, but we shall see in \autoref{sec:extcomp} and \autoref{sec:unique} that this never happens.

Of course, what can and does happen is that an edge we list here does not actually exist. After all, we only studied the restriction of the automorphisms to $V_L$ and considered compatibility conditions based thereon. These spurious edges shall be removed in \autoref{sec:extcomp}. For now, we have obtained a rigorous upper bound on the edges of type~III, except that an edge could actually split up into several edges with the same invariants.

Finally, we mention one further argument to reduce the number of potential edges of type~III:
\begin{enumerate}
\setcounter{enumi}{12}
\item If there is an edge with associated lattice $K$ of $U=V^g$, there must also be an edge with the associated lattice being any other lattice $\tilde{K}$ in the genus of $K$ and with the same Heisenberg commutant.
\end{enumerate}
Indeed, we can extend $W\otimes V_{\tilde{K}}$ in the same way to a nice \voa{} $\tilde{U}$ of central charge $24$ with representation category $\mathcal{C}(2_{\II}^{+2})$ as we can extend $W\otimes V_K$ to $U$ (see \autoref{sec:evensub}). By removing edges corresponding to partially realised genera, we can eliminate some further entries in the list of potential edges.

\medskip

We state the results obtained so far. It is not very enlightening to list all the potential edges of type~III, but we record their numbers (cf.\ \autoref{table:genera}). Which edges are spurious and which are not, will be determined in \autoref{sec:extcomp} below. The non-spurious edges are listed in detail in \autoref{sec:results}.
\[
\setlength{\tabcolsep}{3.5pt}
\begin{tabular}{l|rrrrrrrrrrrrrrrr}
Commutant      & B    & D    & G   & J   & M   & N   &$\Xi$& O    & O    & O   & P   & Q   & R   & S   & T   \\\hline
Neighbours     & A    & B    & C   & G   & A   & B   & B   & D/E  & D    & E   & C   & F   & J   & A/L & L   \\\hline\hline
Correct edges  & $24$ & $56$ & $3$ & $2$ & $1$ & $1$ & $0$ & $10$ & $14$ & $8$ & $1$ & $1$ & $2$ & $1$ & $1$ \\\hline
Spurious edges & $0$  & $14$ & $3$ & $0$ & $0$ & $0$ & $1$ & $0$  & $1$  & $0$ & $0$ & $0$ & $0$ & $0$ & $0$
\end{tabular}
\]
The columns are labelled by the types of the Heisenberg commutants of $U=V^g$ and of the \voa{} neighbours (see \autoref{thm:commconway} below). All spurious edges are loops. This is not particularly surprising as the reduction methods employed above are less restrictive for loops.

We comment on the spurious edge with commutant labelled by the letter $\Xi$. Its Heisenberg commutant and associated lattice would be as follows (cf.\ \autoref{table:genera}):
\[
\begin{tabular}{l|r||l|r||l|l}
Commutant & $c$ & Lattice genus & No.\ & Neighb.\ & Ranks \\\hline
$\mathcal{C}(2_{\II}^{+6}4_{\II}^{+4})$ & 16 & $\II_{8,0}(2_{\II}^{+4}4_{\II}^{+4})$ & 1 & B/B & $16,8,16$
\end{tabular}
\]
We shall see in the proof of \autoref{thm:commconway} that a \voa{} with this representation category and central charge does not appear at all as Heisenberg commutant of the even part $U$ of a nice, self-dual \svoa{} of central charge $24$. (In view of that theorem we point out that there is no element $\mu\in\Co_0$ such that $V_{\Lambda_\mu}^{\hat\mu}$ has the listed properties.)


\subsection{Step 2: Heisenberg Commutants}

In order to apply the methods in \autoref{sec:extcomp} to compute the exact number of edges of type~III (and for the one particular edge of type~I), we first need to state some intermediate results on the Heisenberg commutants appearing in the \voa{}s $U=V^g$, which are also of interest in their own right.
\begin{thm}\label{thm:commconway}
Let $U$ be the even part of a nice, self-dual \svoa{} of central charge 24 with at least one irreducible $U$-module having non-zero weight-1 space. Then the Heisenberg commutant $W=\Com_U(V_K)$ of $U$ (where $K$ is the associated lattice) is isomorphic to
\begin{equation*}
W\cong V_{\Lambda_\mu}^{\hat\mu}
\end{equation*}
for one of 18 conjugacy classes $\mu\in\O(\Lambda)=\Co_0$ of automorphisms of the Leech lattice $\Lambda$, namely the 11 classes A to K listed in \autoref{table:11} and the seven classes M to S listed in \autoref{table:19}.
\end{thm}
This means in particular that the spurious edge labelled by $\Xi$ in the above table does not exist.

The analogous result for \voa{}s (see \autoref{prop:commconway}) was proved in \cite{Hoe17,Lam20}. There, only the $11$ Heisenberg commutants of type A to K listed in \autoref{table:11} appear.

\begin{table}[ht]\caption{The $19$ Heisenberg commutants appearing in the nice, self-dual \svoa{}s of central charge~$24$.}
\begin{tabular}{l|l|r|l|r}
Name & Repr.\ cat.\  & $c$ & $\mu\in\O(\Lambda)$ / $M$ & $\Lambda^\mu$ \cite{HM18} \\\hline\hline
A--K & \multicolumn{3}{l}{see \autoref{table:11}}\\\hline
M & $\mathcal{C}(2_{\II}^{+10})$              & 16 & $1^{-8}2^{16}$       &  14 \\
N & $\mathcal{C}(2_{\II}^{+8}4_{\II}^{+2})$   & 16 & $1^82^{-8}4^8$       &  14 \\
O & $\mathcal{C}(2_{\II}^{+4}4_0^{+6})$       & 16 & $2^44^4$             &  21 \\
P & $\mathcal{C}(2_{\II}^{-8}3^{-3})$         & 18 & $1^42^13^{-4}6^5$    &  33 \\
Q & $\mathcal{C}(2_{\II}^{-6}5^{-3})$         & 20 & $1^22^15^{-2}10^3$   & 100 \\
R & $\mathcal{C}(2_{\II}^{+2}4_4^{+4}3^{+4})$ & 20 & $2^14^16^112^1$      & 135 \\
S & $\mathcal{C}(2_{\II}^{+2})$               & 24 & $1^{-24}2^{24}$ / 2B & 290 \\\hline
T & $\mathcal{C}(2_{\II}^{+2})$               & 24 & 2A                   &
\end{tabular}
\label{table:19}
\end{table}

In addition, there is the Heisenberg commutant appearing in the \voa{} $(V^\natural)^\mathrm{2A}$, the even part of $\VB\otimes F$ where $\VB$ is the shorter Moonshine module, and labelled by the letter T. Here, as for type~S corresponding to the odd moonshine module $\VO$, since the associated lattice is trivial, the Heisenberg commutant is actually the whole even \vosa{}.

The Heisenberg commutants of types S and T are the only non-isomorphic ones appearing in the self-dual \svoa{}s of central charge $24$ that have the same central charge and representation category (but different vector-valued characters, see \autoref{sec:chars} and \autoref{table:969}).

Moreover, there could exist fake copies of $(V^\natural)^\mathrm{2A}$ (again being isomorphic to their respective Heisenberg commutants, which would have the same central charge and representation category as for S and T and even the same vector-valued character as for T) or equivalently fake copies of the shorter moonshine module $\VB$, provided there are fake copies of the moonshine module $V^\natural$ (see \autoref{rem:fakeshort}).\footnote{There cannot be fake copies of $(V^\natural)^\mathrm{2B}$ (or of the odd moonshine module $\VO$) as they would have to yield $V_\Lambda$ as one extension and hence be isomorphic to $V_\Lambda^+\cong(V^\natural)^\mathrm{2B}$ \cite{FLM88}.}

\begin{proof}[Proof of \autoref{thm:commconway}]
It follows from \autoref{prop:asslatinner} or the corresponding statement in \autoref{sec:evensub} that for types I, IIa and IIb only Heisenberg commutants of the form $V_{\Lambda_\mu}^{\hat\mu}$ for the $11$ conjugacy classes $\nu\in\O(\Lambda)$ in \autoref{prop:commconway} can appear.

We consider type~III, i.e.\ non-inner automorphisms $g$. As stated in \autoref{prop:asslatouter}, the Heisenberg commutant of $V^g$ only depends on the self-dual \voa{} $V$ and the outer automorphism $\nu\in\Out(V)\leq\O(L)_{\{\Delta\}}$.

Given a (potential) edge, let $K=L^\nu$ be the associated lattice of $U=V^g$ and $W$ the Heisenberg commutant. Then $U$ is a simple-current extension of the dual pair $W\otimes V_K$. The central charge of $W$ is $24-\rk(K)$ and the representation category is the pointed representation category $\mathcal{C}(\overline{K'/K}\times2_{\II}^{+2})$ (see \autoref{sec:evensub}).

For all edges (except for the spurious edge $\Xi$ mentioned above, which we shall have to treat separately), there is an element $\mu\in\O(\Lambda)=\Co_0$ (unique up to conjugacy) such that the \voa{} $V_{\Lambda_\mu}^{\hat\mu}$ has the same central charge and representation category (see \cite{Lam20}) as $W$. We determine all extensions of $V_{\Lambda_\mu}^{\hat\mu}\otimes V_K$ of glueing type~III to a \voa{} $\tilde{U}$ with representation category $\mathcal{C}(2_{\II}^{+2})$, which then is the even part of a nice, self-dual \svoa{} of central charge $24$, i.e.\ must appear in the list of (potential) edges. We recall that, as the representation category of $V_{\Lambda_\mu}^{\hat\mu}\otimes V_K$ is pointed, this amounts to enumerating certain isotropic subgroups (see, e.g., \cite{EMS20a,Moe16} and in particular \autoref{sec:extcomp}, where we are additionally interested in the precise number of inequivalent extensions). For each such extension $\tilde{U}$ we are able to identify to which entry in the list of (potential) edges it belongs. In particular, we determine the two \voa{} neighbours $\tilde{V}$ (with associated lattice $\tilde{L}$) and the outer automorphisms $\tilde\nu\in\Out(\tilde{V})$.

We do not realise all of the potential edges (some of them are spurious), but crucially all combinations of self-dual \voa{} $\tilde{V}$ and outer automorphism $\tilde\nu$ (again, except for the spurious edge $\Xi$ mentioned above) are realised by an extension $\tilde{U}$ of $V_{\Lambda_\mu}^{\hat\mu}\otimes V_K$ of type~III. But, since the Heisenberg commutant of $\tilde{U}$ is fully determined by $\tilde{V}$ and $\tilde\nu$ (by \autoref{prop:asslatouter}), this shows that the original Heisenberg commutant $W$ in $U$ already had to be isomorphic to $V_{\Lambda_\mu}^{\hat\mu}$.

\smallskip

It remains to study the potential edge (or edges) of type~$\Xi$ with the Heisenberg commutant of $U$ having representation category $\mathcal{C}(2_{\II}^{+6}4_{\II}^{+4})$ and central charge $16$. We shall show that such a \voa{} does not exist.\footnote{A necessary condition for this, which we readily verify, is that there is no element in $\mu\in\Co_0$ such that $V_{\Lambda_\mu}^{\hat\mu}$ has these properties.}

Suppose that this edge exists. Then we know from \autoref{sec:orbcomp} that it must be realised by an automorphism $g$ of order~$2$ of the Schellekens \voa{} $V$ with $V_1\cong A_{1,2}^{16}$ (called B17 in \cite{Hoe17}) satisfying $V_1\cong A_{1,4}^8$. The \voa{} $V$ belongs to family~B in \autoref{table:11}, i.e.\ its Heisenberg commutant is of type~B or equivalently its associated lattice $L$ is in the genus $\gBB$. The restriction of $g$ to $V_L$ is of the form $\hat\nu$ where $\nu\in\Out(V)\leq\O(L)_{\{\Delta\}}$ has Frame shape $2^8$ and satisfies $\det(L^\nu)/\det(L)=4$.\footnotemark{} Analogously to the above arguments (and those in \autoref{prop:asslatouter}), we can argue that if this lattice automorphism $\nu$ is also realised on other lattices $\tilde{L}$ in the same genus $\gBB$, now called $\tilde\nu\in\O(\tilde{L})$, then $\hat{\tilde\nu}\in\Aut(V_{\tilde{L}})$ also extends to an automorphism $\tilde{g}$ of order~$2$ and type~$0$ of the self-dual \voa{} $\tilde{V}$ in family~B whose associated lattice is $\tilde{L}$, realising an edge of glueing type~III. Such lattice automorphisms exist on the associated lattices of the Schellekens \voa{}s B11, B15, B16 and B17 (cf.~\autoref{table:orbifoldB}), but for the former three we excluded in \autoref{sec:orbcomp} that an edge of type~$\Xi$ exists. Hence, the considered potential edge incident with the Schellekens \voa{} B17 did not exist to begin with.
\footnotetext{There is also a conjugacy class $\nu\in\Out(V)\leq\O(L)_{\{\Delta\}}$ with Frame shape $2^8$ and $\det(L^\nu)/\det(L)=1$. Its standard lift $\hat\nu\in\Aut(V_L)$ lifts to an automorphism $g$ of $V$ of order~$2$ and type~$0$ whose corresponding (correct) \svoa{} is the one with commutant of type~N (see \autoref{table:genera}).}

This shows that the potential edge with Heisenberg commutant of $U$ having representation category $\mathcal{C}(2_{\II}^{+6}4_{\II}^{+4})$ and central charge~$16$ cannot exist.
\end{proof}

In \autoref{table:orbifoldB}, to illustrate the last argument in the proof of \autoref{thm:commconway}, we list all the $\Z_2$-orbifold constructions of type~III incident with the Schellekens \voa{}s with Heisenberg commutant of type~B. Here, we already omitted the spurious edges, which are removed in \autoref{sec:extcomp} below. Analogous tables could be drawn for all families admitting edges of type~III (cf.\ \autoref{table:genera}).
\begin{table}[ht]\caption{Edges of glueing type~III incident with Schellekens \voa{}s of type~B.}
\setlength{\tabcolsep}{3pt}
\begin{tabular}{r|l|l||c||c|c|c|c||c|c|c|c||c}
\multicolumn{3}{l||}{Vert.\ op.\ alg. $V$} & \multicolumn{5}{l||}{Conj.\ cl.\ $\nu$ of ord.\ 2} & \multicolumn{5}{l}{Associated lattice of $V^g$}\\\hline
No. & Nm. & $\Out(V)$ & No. & (1) & (2) & (3) & (4) & $\mathrm{D}_\mathrm{III}1$ & $\mathrm{D}_\mathrm{III}2$ & $\mathrm{D}_\mathrm{III}3$ & $\mathrm{D}_\mathrm{III}4$ & $\mathrm{N}_\mathrm{III}$ \\\hline\hline
62 & B1  & $1$                            & 0 & 0 & 0 & 0 & 0 &   &   &   &   &  \\
56 & B2  & $1$                            & 0 & 0 & 0 & 0 & 0 &   &   &   &   &  \\
52 & B3  & $\Z_2$                         & 1 & 1 & 0 & 0 & 0 & 2 &   &   &   &  \\
53 & B4  & $1$                            & 0 & 0 & 0 & 0 & 0 &   &   &   &   &  \\
50 & B5  & $\Z_2$                         & 1 & 1 & 0 & 0 & 0 &   & 2 &   &   &  \\
47 & B6  & $\Z_2$                         & 1 & 1 & 0 & 0 & 0 &   &   & 2 &   &  \\
48 & B7  & $\Z_2$                         & 1 & 0 & 1 & 0 & 0 &   &   &   &   &  \\
44 & B8  & $\Z_2$                         & 1 & 1 & 0 & 0 & 0 & 6 &   &   &   &  \\
40 & B9  & $\Z_2$                         & 1 & 0 & 1 & 0 & 0 &   &   &   &   &  \\
39 & B10 & $\Z_2$                         & 1 & 1 & 0 & 0 & 0 &   & 6 &   &   &  \\
38 & B11 & $S_4$                          & 2 & 1 & 0 & 1 & 0 & 3 &   &   &   &  \\
33 & B12 & $\Z_2^2$                       & 3 & 2 & 1 & 0 & 0 & 7 &   &   & 2 &  \\
31 & B13 & $D_8$                          & 3 & 2 & 1 & 0 & 0 &   & 3 & 6 &   &  \\
26 & B14 & $D_8$                          & 3 & 1 & 2 & 0 & 0 &   &   &   & 5 &  \\
25 & B15 & $\Z_2\times S_4$               & 5 & 3 & 1 & 1 & 0 & 3 & 7 & 3 &   &  \\
16 & B16 & $W(D_4)$                       & 6 & 3 & 2 & 1 & 0 &   & 3 & 7 & 3 &  \\
5  & B17 & $\mathrm{AGL}_4(\mathbb{F}_2)$ & 4 & 1 & 1 & 1 & 1 &   &   & 3 &   & 1
\end{tabular}
\label{table:orbifoldB}
\end{table}

The isomorphism types of the outer automorphism groups are listed in \cite{BLS23}. There are four types of automorphisms $\nu$ of order~$2$ occurring in $\Out(V)\leq\O(L)_{\{\Delta\}}$, which we distinguish by the Frame shape and the quotient $\det(L^\nu)/\det(L)$:
\begin{enumerate}
\item $1^82^4$, $4$,
\item $1^42^6$ (order doubling), $4$,
\item $2^8$, $4$,
\item $2^8$, $1$.
\end{enumerate}
Exactly the automorphisms $\nu\in\Out(V)$ under (1) and (4) lift to automorphisms $g$ of $V$ yielding edges of glueing type~III. Each conjugacy class in $\Out(V)$ corresponds to a different isomorphism class of associated lattice of $V^g$ and typically lifts to several non-conjugate automorphisms in $g\in\Aut(V)$, whose numbers we record in the last five columns.


\subsection{Step 3: Simple-Current Extensions}\label{sec:extcomp}

In the following, we use the extension results in \autoref{sec:evensub} to determine the precise number of nice \voa{}s $U$ of central charge $24$ with $\Rep(U)\cong\mathcal{C}(2_{\II}^{+2})$. In particular, for type~III we remove the spurious cases listed in \autoref{sec:orbcomp} and show that each remaining case (characterised by a quintuple of invariants) corresponds to exactly one \voa{} $U$, i.e.\ to exactly one self-dual \svoa{} (with the exception of the two edges of types P and Q that we shall treat in \autoref{sec:unique}). The latter is also done for the one edge of type~I that could potentially still split up. All other edges of type I, IIa and IIb are already fully classified.

The approach is to count, following \autoref{sec:enumsvoa}, the inequivalent simple-current extensions of $W\otimes V_K$ to \voa{}s $U$ with $\Rep(U)\cong\mathcal{C}(2_{\II}^{+2})$ where $W\cong V_{\Lambda_\mu}^{\hat\mu}$ is as in \autoref{thm:commconway} and $K$ is a lattice in the corresponding genus.

As input of these computations, like in \autoref{sec:orbcomp}, we need the characters and automorphism groups of the \voa{}s $W$, noting again that for the lattice \voa{}s $V_K$ these are well understood \cite{DN99,HM22}. As types A to K were treated in \autoref{sec:orbcomp}, the only new data are required for $W$ of types M to R, and hence only for glueing type~III.

For types M and S, $W$ is simply the \fpvosa{} of the lattice \voa{} associated with the Barnes-Wall and Leech lattice, respectively, under the $(-1)$-involution and the characters are given in \cite{Shi04,FLM88}. For the remaining types N to R, we compute the characters explicitly (as explained in \autoref{sec:orbcomp}), or at least those Fourier coefficients that are relevant in the following.

The automorphism groups $\Aut(W)$ for types M and S are given in \cite{Shi04} and for type~N in \cite{CL21}. For types O and R, Ching Hung Lam and Hiroki Shimakura communicated the shape of $\Aut(W)$ to us, along with some partial information for types P and Q. The remaining ambiguities from not knowing the precise automorphism groups for types P and Q shall be resolved in \autoref{sec:unique} using another approach.

\medskip

With the data summarised above, we proceed as explained in \autoref{sec:enumsvoa}:
\begin{enumerate}
\item We show that the two non-conjugate inner automorphisms $g$ of type~I of the Schellekens \voa{} $V$ with $V_1\cong A_{1,1}^2D_{6,5}$ and $V^g_1\cong A_{1,1}A_{5,5}\C^2$ have isomorphic \fpvosa{}s $V^g$ and thus only correspond to one self-dual \svoa{} (of type~(2) in \autoref{prop:order2conj}).
\item We prove that none of the $18$ spurious cases (occurring for types D, G and~O) in \autoref{sec:orbcomp} exist.
\item We show that, with the exception of the two cases of types P and Q, each case of glueing type~III from \autoref{sec:orbcomp}, characterised by a quintuple of invariants, corresponds to just one self-dual \svoa{}.
\end{enumerate}

We have hence almost completed the classification of the nice, self-dual \svoa{}s of central charge~$24$, with the two remaining cases being discussed in the next section.


\subsection{Step 4: Remaining Uniqueness}\label{sec:unique}

In the following, we complete the classification proof for the nice, self-dual \svoa{}s of central charge~$24$ by studying the two cases of types P and Q and glueing type~III (see \autoref{table:genera}). In each of the two cases, all possible extensions $U$ of $W\otimes V_K$ with $\Rep(U)\cong\mathcal{C}(2_{\II}^{+2})$ have the same invariants and we show that there is indeed just one such extension~$U$ up to isomorphism.

We consider type~Q and note that type~P can be handled in a similar fashion. By the preliminary classification results in \autoref{sec:orbcomp}, the \voa{} $U$ in question must be a \fpvosa{} $U\cong V^g$ of the Schellekens \voa{} $V$ with $V_1\cong A_{4,5}^2$ under an automorphism of order~$2$ such that $V^g_1\cong A_{4,10}$. We shall show that there is only one such automorphism in $\Aut(V)$ up to conjugacy, thus proving the uniqueness of $U$.

We proceed analogously to \cite{EMS20b}. Suppose that $g$, $\tilde{g}\in\Aut(V)$ have the same fixed-point Lie subalgebra $A_{4,10}$. The restrictions $r(g)$ and $r(\tilde{g})$ to $V_1$ are conjugate in $\Aut(V_1)$. Indeed, both must project to the outer automorphism of $V_1$ permuting the two simple ideals so that they are conjugate under an inner automorphism by Lemma~8.1 in \cite{EMS20b}. By definition, this inner automorphism extends to an inner automorphism of $V$, say $k\in K$, so that $r(kg^{-1}k^{-1}\tilde{g})=\id$. But, $\ker(r)=\{\id\}$ for this Schellekens \voa{} $V$ (see \cite{BLS23}) so that $kg^{-1}k^{-1}\tilde{g}=\id$.

Hence, there is exactly one nice, self-dual \svoa{} of type~Q. With the analogous result for type~P this finally concludes the classification proof, leaving unsolved only the question of possible fake copies of the type-T edge (see \autoref{rem:fakeshort} below).


\subsection{Classification Results}\label{sec:results}

Collecting the results from \autoref{sec:orbcomp} to \autoref{sec:unique}, we state the main classification theorem. Recall that we assume that \svoa{}s are not purely even.
\begin{thm}[Classification]\label{thm:class}
Up to isomorphism, there are exactly 968 nice, self-dual \svoa{}s of central charge~24 for which the weight\nobreakdash-1 space or that of the canonically twisted module is non-zero.

In addition, there is the nice, self-dual \svoa{} $\VB\otimes F$ of central charge 24 obtained by tensoring the shorter moonshine module with a free fermion.
\end{thm}

\begin{rem}\label{rem:fakeshort}
The \svoa{} $\VB\otimes F$, with even part $(V^\natural)^\mathrm{2A}$, is nice, self-dual and of central charge~$24$, but its weight-$1$ space and that of the canonically twisted module vanish.

We cannot rule out that there are further (or \emph{fake}) nice, self-dual \svoa{}s of central charge $24$ with these properties. They are similar to ${\VB\otimes F}$ in the following sense: there must be a nice, self-dual \voa{}~$V$ of central charge $24$ with $V_1=\{0\}$ (i.e.\ a fake copy of the moonshine module) and an automorphism of order~$2$ and type~$0$ of $V$ with $V^g$ satisfying the positivity condition such that also $V^{\orb(g)}_1=\{0\}$. Then, $V^g$ extends to a self-dual \svoa{} with said property.

The vector-valued character of $V^g$ is identical to that of $(V^\natural)^\mathrm{2A}$ following \autoref{sec:chars}, which implies that the self-dual \svoa{} splits exactly one free fermion and hence that $V^{\orb(g)}\cong V$ (see \autoref{prop:suffloop}). In fact, $g$ must be a Miyamoto involution of $V\cong V^{\orb(g)}$ associated with some Ising vector (see \autoref{sec:graphmeth}), like the 2A-involutions of $V^\natural$ \cite{Miy96}.
\end{rem}

We conjecture that such fake copies of $\VB\otimes F$ or equivalently of $(V^\natural)^\mathrm{2A}$ do not exist, which would be a consequence of the moonshine uniqueness conjecture in \cite{FLM88}. This is equivalent to the conjecture in \cite{Hoe95} that there are no \emph{fake} copies of the shorter moonshine module $\VB$ in the sense that this is the unique nice, self-dual \svoa{} of central charge $23\shs\sfrac{1}{2}$ with vanishing weight-$\sfrac{1}{2}$ and weight-$1$ spaces.

\medskip

In \autoref{table:genera} we list some data associated with the $969$ self-dual \svoa{}s of central charge~$24$, in particular the genus of the Heisenberg commutant and of the associated lattice of the even part, as well as the number of non-isomorphic lattices in the genus. We also list the genus of the Heisenberg commutant of the \voa{} neighbours. Then we state the rank of the associated lattice of one neighbour, of the even part and of the other neighbour. In the last two columns we state the precise number of \svoa{}s with these properties, grouped into non-loops and loops.

\begin{table}[p]\caption{Genera associated with nice, self-dual \svoa{}s of central charge $24$.}
\setlength{\tabcolsep}{2.4pt}
\renewcommand{\arraystretch}{1.025}
\begin{tabular}{c|l|r|l||l|r||l|l|r|r}
\multicolumn{10}{c}{Type I} \\\hline
\multicolumn{2}{l|}{Commutant}          & $c$ & $\O(\Lambda)$  & Lattice genus                    & No. & Nb.\     & Ranks      & NL   & L   \\\hline
A & $\mathcal{C}(1)$                    &   0 & $1^{24}$       & $\II_{24,0}(2_{\II}^{+2})$       & 273 & A/A      & $24,24,24$ & $122$ & $151$ \\
B & $\mathcal{C}(2_{\II}^{+10})$        &   8 & $1^82^8$       & $\II_{16,0}(2_{\II}^{+12})$      &  11 & B/B      & $16,16,16$ &  $71$ & $103$ \\
C & $\mathcal{C}(3^{-8})$               &  12 & $1^63^6$       & $\II_{12,0}(2_{\II}^{+2}3^{-8})$ &  33 & C/C      & $12,12,12$ &   $8$ &  $25$ \\
E & $\mathcal{C}(2_6^{+2}4_{\II}^{-6})$ &  14 & $1^42^24^4$    & $\II_{10,0}(2_2^{+4}4^{+6})$     &   3 & E/E      & $10,10,10$ &   $4$ &  $10$ \\
F & $\mathcal{C}(5^{+6})$               &  16 & $1^45^4$       & $\II_{8,0}(2_{\II}^{+2}5^{+6})$  &   6 & F/F      & $8,8,8$    &   $1$ &   $5$ \\
G & $\mathcal{C}(2_{\II}^{+6}3^{-6})$   &  16 & $1^22^23^26^2$ & $\II_{8,0}(2_{\II}^{+8}3^{-6})$  &   1 & G/G      & $8,8,8$    &   $1$ &   $4$ \\
H & $\mathcal{C}(7^{+5})$               &  18 & $1^37^3$       & $\II_{6,0}(2_{\II}^{+2}7^{-5})$  &   1 & H/H      & $6,6,6$    &   $0$ &   $1$ \\
\multicolumn{10}{c}{}\\[-2mm]
\multicolumn{10}{c}{Type IIa} \\\hline
\multicolumn{2}{l|}{Commutant} & $c$ & $\O(\Lambda)$ & Lattice genus & No. & Nb.\ & Ranks & NL & L \\\hline
B & $\mathcal{C}(2_{\II}^{+10})$                  &  8 & $1^82^8$       & $\II_{16,0}(2_{\II}^{+10})$                 & 17 & B/B & $16,16,16$ & $0$ & $69$ \\
D & $\mathcal{C}(2_{\II}^{-10}4_{\II}^{-2})$      & 12 & $2^{12}$       & $\II_{12,0}(2_{\II}^{-10}4_{\II}^{-2})$     &  2 & D/D & $12,12,12$ & $9$ & $52$ \\
E & $\mathcal{C}(2_6^{+2}4_{\II}^{-6})$           & 14 & $1^42^24^4$    & $\II_{10,0}(2_2^{+2}4_{\II}^{-6})$          &  5 & E/E & $10,10,10$ & $0$ & $20$ \\
G & $\mathcal{C}(2_{\II}^{+6}3^{-6})$             & 16 & $1^22^23^26^2$ & $\II_{8,0}(2_{\II}^{+6}3^{-6})$             &  2 & G/G & $8,8,8$    & $0$ &  $6$ \\
I & $\mathcal{C}(2_1^{+1}4_5^{-1}8_{\II}^{-4})$   & 18 & $1^22^14^18^2$ & $\II_{6,0}(2_3^{-1}4_3^{-1}8_{\II}^{-4})$   &  1 & I/I & $6,6,6$    & $0$ &  $3$ \\
J & $\mathcal{C}(2_{\II}^{+4}4_{\II}^{-2}3^{-5})$ & 18 & $2^36^3$       & $\II_{6,0}(2_{\II}^{+4}4_{\II}^{-2}3^{+5})$ &  1 & J/J & $6,6,6$    & $0$ &  $6$ \\
K & $\mathcal{C}(2_{\II}^{-2}4_{\II}^{-2}5^{+4})$ & 20 & $2^210^2$      & $\II_{4,0}(2_{\II}^{-2}4_{\II}^{-2}5^{+4})$ &  1 & K/K & $4,4,4$    & $0$ &  $2$ \\
\multicolumn{10}{c}{}\\[-2mm]
\multicolumn{10}{c}{Type IIb} \\\hline
\multicolumn{2}{l|}{Commutant} & $c$ & $\O(\Lambda)$ & Lattice genus & No. & Nb.\ & Ranks & NL & L \\\hline\hline
B & $\mathcal{C}(2_{\II}^{+10})$                  &  8 & $1^82^8$       & $\II_{16,0}(2_{\II}^{+10})$                 & 17 & A/B & $24,16,16$ & $76$ & $0$ \\\hline
\multirow{2}{*}{D} & \multirow{2}{*}{$\mathcal{C}(2_{\II}^{-10}4_{\II}^{-2})$} & \multirow{2}{*}{12} & \multirow{2}{*}{$2^{12}$} & \multirow{2}{*}{$\II_{12,0}(2_{\II}^{-10}4_{\II}^{-2})$} & \multirow{2}{*}{2} & A/D & $24,12,12$ & $15$ & $0$ \\
                                                                                                                      &&&&&& B/D & $16,12,12$ & $54$ & $0$ \\\hline
E & $\mathcal{C}(2_6^{+2}4_{\II}^{-6})$           & 14 & $1^42^24^4$    & $\II_{10,0}(2_2^{+2}4_{\II}^{-6})$          &  5 & B/E & $16,10,10$ & $15$ & $0$ \\\hline
G & $\mathcal{C}(2_{\II}^{+6}3^{-6})$             & 16 & $1^22^23^26^2$ & $\II_{8,0}(2_{\II}^{+6}3^{-6})$             &  2 & C/G & $12,8,8$   &  $4$ & $0$ \\\hline
I & $\mathcal{C}(2_1^{+1}4_5^{-1}8_{\II}^{-4})$   & 18 & $1^22^14^18^2$ & $\II_{6,0}(2_3^{-1}4_3^{-1}8_{\II}^{-4})$   &  1 & E/I & $10,6,6$   &  $2$ & $0$ \\\hline
\multirow{2}{*}{J} & \multirow{2}{*}{$\mathcal{C}(2_{\II}^{+4}4_{\II}^{-2}3^{-5})$} & \multirow{2}{*}{18} & \multirow{2}{*}{$2^36^3$} & \multirow{2}{*}{$\II_{6,0}(2_{\II}^{+4}4_{\II}^{-2}3^{+5})$} & \multirow{2}{*}{1} & C/J & $12,6,6$ & $2$ & $0$ \\
                                                                                                                      &&&&&& G/J & $8,6,6$    &  $2$ & $0$ \\\hline
K & $\mathcal{C}(2_{\II}^{-2}4_{\II}^{-2}5^{+4})$ & 20 & $2^210^2$      & $\II_{4,0}(2_{\II}^{-2}4_{\II}^{-2}5^{+4})$ &  1 & F/K & $8,4,4$    &  $1$ & $0$ \\
\multicolumn{10}{c}{}\\[-2mm]
\multicolumn{10}{c}{Type III} \\\hline
\multicolumn{2}{l|}{Commutant} & $c$ & $\O(\Lambda)$ / $M$ & Lattice genus & No. & Nb.\ & Ranks & NL & L \\\hline\hline
B & $\mathcal{C}(2_{\II}^{+10})$ & 8 &  $1^82^8$ & $\II_{16,0}(2_{\II}^{+8})$ & 24 & A/A & $24,16,24$ & $0$ & $24$ \\\hline
D & $\mathcal{C}(2_{\II}^{-10}4_{\II}^{-2})$ & 12 & $2^{12}$ & $\II_{12,0}(2_{\II}^{-8}4_{\II}^{-2})$ & 4 & B/B & $16,12,16$ & $17$ & $39$ \\\hline
G & $\mathcal{C}(2_{\II}^{+6}3^{-6})$ & 16 & $1^22^23^26^2$ & $\II_{8,0}(2_{\II}^{+4}3^{-6})$ & 3 & C/C & $12,8,12$ & $0$ & $3$ \\\hline
J & $\mathcal{C}(2_{\II}^{+4}4_{\II}^{-2}3^{-5})$ & 18 & $2^36^3$ & $\II_{6,0}(2_{\II}^{+2}4_{\II}^{-2}3^{+5})$ & 1 & G/G & $8,6,8$ & $0$ & $2$ \\\hline\hline
M & $\mathcal{C}(2_{\II}^{+10})$ & 16 & $1^{-8}2^{16}$ & $\II_{8,0}(2_{\II}^{+8})$ & 1 & A/A & $24,8,24$ & $0$ & $1$ \\\hline
N & $\mathcal{C}(2_{\II}^{+8}4_{\II}^{+2})$ & 16 & $1^82^{-8}4^8$ & $\II_{8,0}(2_{\II}^{+6}4_{\II}^{+2})$ & 1 & B/B & $16,8,16$ & $0$ & $1$ \\\hline
\multirow{3}{*}{O} & \multirow{3}{*}{$\mathcal{C}(2_{\II}^{+4}4_0^{+6})$} & \multirow{3}{*}{16} & \multirow{3}{*}{$2^44^4$} & \multirow{3}{*}{$\II_{8,0}(2_{\II}^{+2}4_0^{+6})$} & \multirow{3}{*}{1} & D/E & $12,8,10$ & $10$ & $0$ \\
&&&&&& D/D & $12,8,12$ & $3$ & $11$ \\
&&&&&& E/E & $10,8,10$ & $1$ & $7$ \\\hline
P & $\mathcal{C}(2_{\II}^{-8}3^{-3})$ & 18 & $1^42^13^{-4}6^5$ & $\II_{6,0}(2_{\II}^{-6}3^{+3})$ & 1 & C/C & $12,6,12$ & $0$ & $1$ \\\hline
Q & $\mathcal{C}(2_{\II}^{-6}5^{-3})$ & 20 & $1^22^15^{-2}10^3$ & $\II_{4,0}(2_{\II}^{-4}5^{-3})$ & 1 & F/F & $8,4,8$ & $0$ & $1$ \\\hline
R & $\mathcal{C}(2_{\II}^{+2}4_4^{+4}3^{+4})$ & 20 & $2^14^16^112^1$ & $\II_{4,0}(4_4^{+4}3^{+4})$ & 1 & J/J & $6,4,6$ & $0$ & $2$ \\\hline
S & $\mathcal{C}(2_{\II}^{+2})$ & 24 & $1^{-24}2^{24}$/2B & $\II_{0,0}(1)$ & 1 & A/L & $24,0,0$ & $1$ & $0$ \\\hline
T & $\mathcal{C}(2_{\II}^{+2})$ & 24 & 2A & $\II_{0,0}(1)$ & 1 & L/L & $0,0,0$ & $0$ & $1$
\end{tabular}
\label{table:genera}
\end{table}

\phantomsection\label{text:expl}Finally, in \autoref{table:969} we list all $969$ nice, self-dual \svoa{}s $V$ of central charge $24$, ignoring potential fake copies of the shorter moonshine module. Here, we include further information such as the weight-$1$ Lie algebra $V_1$ together with its dimension, the central charge $c_\text{st}=c-\sfrac{l}{2}$ of the stump $\bar{V}$, the vector-valued character (expressed in terms of the constants $a$, $b$ and $l$, see \autoref{sec:chars}) and the two \voa{} neighbours labelled as in \cite{Sch93,Hoe17}.

The weight-$1$ Lie algebra $V_1$ decomposes into a direct sum of the contribution from the stump $\bar{V}$ and from the free fermions $F^l$. The latter is always isomorphic to $\so_l$ (with special cases for $l\leq2$, see \autoref{sec:split}) and in the table we place it in brackets.

We explain the naming convention (e.g., $\mathrm{B}_\mathrm{IIa}3\mathrm{f}$), which essentially follows the nomenclature in \cite{Hoe17} in the bosonic case. The first letter stands for the type of the Heisenberg commutant (see \autoref{table:19}), with the glueing type (I, IIa, IIb or III) as a subscript. The number that follows enumerates the different associated lattices, whose genera are listed in \autoref{table:genera}. The last letter indexes the inequivalent extensions of a given dual pair. We drop the last or penultimate symbol if there is only one possibility.

\medskip

As the complete $2$-neighbourhood graph in central charge $24$ (cf.\ \autoref{fig:small}) is too unwieldy to be pictured adequately, we shall only present the coarser graph obtained by projecting each \voa{} to its Heisenberg commutant (see \autoref{fig:course}).

\begin{figure}[ht]
\caption{The Heisenberg commutant $2$-neighbourhood graph for the nice, self-dual \voa{}s of central charge $24$.}
~\\
\begin{tikzcd}[cells={nodes={circle,minimum width=6mm,inner sep=0pt,draw,align=center}}]
&&\text{B}\arrow[-]{dr}[swap]{\text{E}_\text{IIb}\!\!}\arrow[-]{dd}{\text{D}_\text{IIb}}\arrow[-,loop,in=60,out=120,looseness=8]{}{\text{B}_\text{I},\text{B}_\text{IIa},\text{D}_\text{III},\text{N}_\text{III}}\\
\text{L}\arrow[-]{r}{\text{S}_\text{III}}\arrow[-,loop,in=60,out=120,looseness=8]{}{\text{T}_\text{III}}&\text{A}\arrow[-]{ur}[swap]{\!\!\text{B}_\text{IIb}}\arrow[-]{dr}{\!\!\text{D}_\text{IIb}}\arrow[-,loop,in=60,out=120,looseness=8]{}{\text{A}_\text{I},\text{B}_\text{III},\text{M}_\text{III}}&&\text{E}\arrow[-]{r}{\text{I}_\text{IIb}}\arrow[-,loop,in=60,out=120,looseness=8]{}{\text{E}_\text{I},\text{E}_\text{IIa},\text{O}_\text{III}}&\text{I}\arrow[-,loop,in=60,out=120,looseness=8]{}{\text{I}_\text{IIa}}\\
\text{C}\arrow[-]{dr}{\!\!\text{J}_\text{IIb}}\arrow[-]{d}{\text{G}_\text{IIb}}\arrow[-,loop,in=-30,out=30,looseness=8]{}{\text{C}_\text{I},\text{G}_\text{III},\text{P}_\text{III}}&&\text{D}\arrow[-]{ur}{\text{O}_\text{III}\!\!}\arrow[-,loop,in=240,out=300,looseness=8]{}{\text{D}_\text{IIa},\text{O}_\text{III}}&&\text{H}\arrow[-,loop,in=150,out=210,looseness=8]{}{\text{H}_\text{I}}\\
\text{G}\arrow[-]{r}{\text{J}_\text{IIb}}\arrow[-,loop,in=240,out=300,looseness=8]{}{\text{G}_\text{I},\text{G}_\text{IIa},\text{J}_\text{III}}&\text{J}\arrow[-,loop,in=240,out=300,looseness=8]{}{\text{J}_\text{IIa},\text{R}_\text{III}}&&\text{F}\arrow[-]{r}{\text{K}_\text{IIb}}\arrow[-,loop,in=240,out=300,looseness=8]{}{\text{F}_\text{I},\text{Q}_\text{III}}&\text{K}\arrow[-,loop,in=240,out=300,looseness=8]{}{\text{K}_\text{IIa}}
\end{tikzcd}
\label{fig:course}
\end{figure}
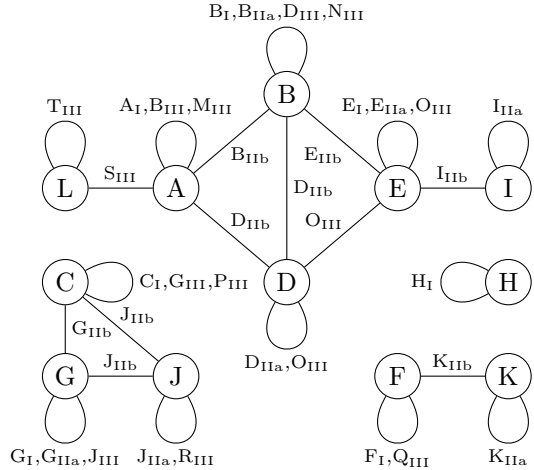

This graph has four connected components, which implies that the actual \voa{} $2$-neighbourhood graph has at least as many. (We again ignore fake copies of the edge of type $\text{T}_\text{III}$ labelled by $\VB\otimes F$ or equivalently by $(V^\natural)^\mathrm{2A}$, all of which are loops.) In fact, we also verify that the \voa{} $2$-neighbourhood graph has precisely these four connected components.

This should be contrasted to the case of positive-definite, even, unimodular lattices of a fixed rank, whose $2$-neighbourhood graph is always connected \cite{Kne57}. That the \voa{} $2$-neighbourhood graph is disconnected was already observed in \cite{Mon98}.

\medskip

We recall from \autoref{prop:split} and \autoref{prop:holsplit} that each nice, self-dual \svoa{} $V$ of central charge $24$ sits inside an infinite family of nice, self-dual vertex operator (super)algebras
\begin{equation*}
\bar{V},\ \bar{V}\otimes F,\ \bar{V}\otimes F^{\otimes2},\dots,\ V\cong\bar{V}\otimes F^{\otimes l},\dots
\end{equation*}
of central charges
\begin{equation*}
c_\text{st},\ c_\text{st}+\sfrac{1}{2},\dots,\ 24=c_\text{st}+\sfrac{l}{2},\dots
\end{equation*}
respectively, where $l=\dim(V_{1/2})$ and the stump of $V$ is denoted by $\bar{V}$, which can be a \voa{} rather than a \svoa{}.

This way we immediately obtain a classification of the nice, self-dual \svoa{}s of central charge at most $24$. In \autoref{table:numbers} we list their numbers. We also list the number of stumps, i.e.\ those \svoa{}s that do not split off any free fermions (and hence do not arise from a theory in smaller central charge). Moreover, we also list the number of (purely even) \voa{}s. The numbers in brackets assume the (shorter) moonshine uniqueness conjecture but are rigorous as lower bounds.

For comparison, we also list the number of cases that can be realised as lattice vertex operator (super)algebras, which equals the number of positive-definite, odd or even, unimodular lattices of dimension $d$ equal to $c$. These numbers are taken from \cite{Kin03}.

A classification of the nice, self-dual \svoa{}s of central charge at most $22\shs\sfrac{1}{2}$ was obtained independently in \cite{Ray23} and agrees with our results. We remark, however, that the majority of cases in central charge $24$ (and in particular all the difficult cases of type~III) stems from stumps in central charge $23$, $23\shs\sfrac{1}{2}$ and $24$.

\begin{table}[ht]\caption{The number of nice, self-dual vertex operator (super)algebras of central charge $c$ at most $24$.}
\renewcommand{\arraystretch}{0.965}
\begin{tabular}{r||r|r|r||r|r|r}
& \multicolumn{3}{c||}{Vert.\ op.\ (sup.)algs.} & \multicolumn{3}{c}{Lattices}\\\cline{2-7}
$c$ / $d$ & Odd & Stump & Even & Odd & Stump & Even \\\hline\hline
$0\ph$               &   1 &   1 &  1 &   0 &   0 &  1 \\\hline
$\sfrac{1}{2}$       &   1 &   0 &    &     &     &    \\
$1\ph$               &   1 &   0 &    &   1 &   0 &    \\
$1\shs\sfrac{1}{2}$  &   1 &   0 &    &     &     &    \\
$2\ph$               &   1 &   0 &    &   1 &   0 &    \\
$2\shs\sfrac{1}{2}$  &   1 &   0 &    &     &     &    \\
\hline
\end{tabular}
\label{table:numbers}
\end{table}

\FloatBarrier

{
\renewcommand{\arraystretch}{1.22}
\setlength{\tabcolsep}{1.5pt}
\SMALL
{
\begin{longtable}{r|l|l|r|l|l|r||r|r|r||r|l||r|l}
\caption{The nice, self-dual \svoa{}s of central charge $24$ (continued on next pages). See p.~\pageref{text:expl} for details.}\label{table:969}\\
\multicolumn{7}{l||}{Vertex Operator Superalgebra} & \multicolumn{3}{l||}{Character} & \multicolumn{4}{l}{Neighbours}\\\hline
No. & Name & Lie Algebra & Dim. & C. & Tp. & $c_\text{st.}$ & $a$ & $b$ & $l$ & No. & Nm. & No. & Nm. \\\hline\hline
\endfirsthead
\caption{The nice, self-dual \svoa{}s of central charge $24$ (continued).}\\
\multicolumn{7}{l||}{Vertex Operator Superalgebra} & \multicolumn{3}{l||}{Character} & \multicolumn{4}{l}{Neighbours}\\\hline
No. & Name & Lie Algebra & Dim. & C. & Tp. & $c_\text{st.}$ & $a$ & $b$ & $l$ & No. & Nm. & No. & Nm. \\\hline\hline\endhead
1 & $\mathrm{A}_{\mathrm{I}}1\mathrm{}$ & $(D_{24,1}) $ & 1128 & A & I & 0\phh & 1152 & 1152 & 48 & 70 & A1 & 70 & A1 \\
2 & $\mathrm{A}_{\mathrm{I}}2\mathrm{}$ & $D_{20,1}(D_{4,1}) $ & 808 & A & I & 20\phh & 512 & 512 & 8 & 70 & A1 & 70 & A1 \\
3 & $\mathrm{A}_{\mathrm{I}}3\mathrm{}$ & $E_{8,1}(D_{16,1}) $ & 744 & A & I & 8\phh & 768 & 768 & 32 & 69 & A2 & 69 & A2 \\
4 & $\mathrm{A}_{\mathrm{I}}4\mathrm{}$ & $D_{8,1}D_{16,1} $ & 616 & A & I & 24\phh & 512 & 128 & 0 & 69 & A2 & 70 & A1 \\
5 & $\mathrm{A}_{\mathrm{I}}5\mathrm{}$ & $D_{16,1}(D_{8,1}) $ & 616 & A & I & 16\phh & 512 & 512 & 16 & 69 & A2 & 69 & A2 \\
6 & $\mathrm{A}_{\mathrm{I}}6\mathrm{}$ & $E_{8,1}^{2}(D_{8,1}) $ & 616 & A & I & 16\phh & 512 & 512 & 16 & 68 & A3 & 68 & A3 \\
7 & $\mathrm{A}_{\mathrm{I}}7\mathrm{}$ & $A_{23,1}\C $ & 576 & A & I & 24\phh & 552 & 48 & 0 & 67 & A4 & 70 & A1 \\
8 & $\mathrm{A}_{\mathrm{I}}8\mathrm{}$ & $D_{12,1}^{2} $ & 552 & A & I & 24\phh & 576 & 0 & 0 & 66 & A5 & 70 & A1 \\
\end{longtable}
}

}

\FloatBarrier


\bibliographystyle{alpha_noseriescomma}
\bibliography{quellen}{}

\end{document}